\documentclass[11pt]{article}
\usepackage{fullpage,amsmath, amsfonts, amssymb, amsthm, natbib}
\usepackage{graphicx, algorithmic, algorithm, enumerate,color,url}
\pdfoutput=1
\title{Convex Banding of the Covariance Matrix}
\author{Jacob Bien\footnote{Department of BSCB, 
Department of Statistical Science,
1178 Comstock Hall,
Cornell University,
Ithaca, NY 14853}, Florentina Bunea\footnote{Department of Statistical Science,
1184 Comstock Hall,
Cornell University,
Ithaca, NY 14853}, and Luo Xiao\footnote{Department of Biostatistics,
Bloomberg School of Public Health,
Johns Hopkins University,
Baltimore, MD 21205}}
\newcommand{\real}{\mathbb R}
\newtheorem{lem}{Lemma}
\newtheorem{cor}{Corollary}
\newtheorem{theorem}{Theorem}
\newtheorem{prop}{Proposition}
\newtheorem{property}{Property}
\newtheorem{rem}{Remark}
\def\half{\frac{1}{2}}
\def\real{\mathbb R}
\def\Al{A^{(\ell)}}
\def\hAl{\hat A^{(\ell)}}
\def\gl{{g_\ell}}
\def\cov{\mathrm{Cov}}
\def\st{~\mathrm{s.t.}~}
\def\tr{\mathrm{tr}}
\def\E{\mathbb{E}}

\def\S{\mathbf S}
\def\hatP{\hat{\Sigma}}
\def\tildeP{\tilde{\Sigma}}
\def\true{\Sigma^*}
\DeclareMathOperator*{\minimize}{\mathrm{Minimize}}

\begin{document}
\maketitle

\begin{abstract}
  We introduce a new sparse estimator of the covariance
  matrix for high-dimensional models 
  in which the variables have a known ordering.  Our
  estimator, which is the solution to a convex optimization problem,
  is equivalently expressed as an estimator which tapers the sample covariance matrix by a
  Toeplitz, sparsely-banded, data-adaptive matrix.  As a result of this adaptivity, the convex banding 
  estimator enjoys theoretical optimality properties not attained by
  previous banding or tapered estimators.  In particular, our  convex banding estimator
  is minimax rate adaptive  in Frobenius and operator norms, up to log factors, over  commonly-studied classes of
  covariance matrices, and  over more general classes.  Furthermore, it correctly recovers
  the bandwidth when the true covariance is exactly banded.   Our convex formulation
  admits a simple and efficient algorithm. Empirical studies
  demonstrate its practical effectiveness and illustrate that  
   our exactly-banded estimator works well even when the true
  covariance matrix is only close to a banded matrix, confirming  our theoretical results. 
  Our method compares favorably with all existing methods, in terms of accuracy and speed. 
  We  illustrate the practical merits of the convex banding estimator by showing that 
  it  can be used to improve the performance of
  discriminant analysis for classifying sound recordings.
\end{abstract}

\section{Introduction}
\label{sec:introduction}

The covariance matrix is one of the most fundamental objects in
statistics, and yet its reliable estimation is greatly challenging in
high dimensions.
The sample covariance matrix is known to degrade rapidly as an estimator as the
number of variables increases, as noted more than four decades  ago, see for instance  \cite{Dempster72}, unless additional
assumptions are placed on the underlying covariance structure, see 
\cite{Bunea12} for an overview.  Indeed, the key to tractable estimation of a
high-dimensional covariance matrix lies in exploiting any knowledge of
the covariance structure.  In this vein, this paper develops an
estimator geared toward a covariance structure that naturally arises
in many applications in which the variables are ordered.  Variables
collected over time or representing spectra constitute two
subclasses of examples.

We observe an independent sample, $\mathbf X_1,\ldots, \mathbf X_n\in\real^p$ of mean zero 
random vectors with true population covariance matrix $\mathbb{E}[\mathbf X_i \mathbf
X_i^T]=\true$.  When the variables have a known ordering, it is
often natural to assume that the dependence between variables is  a function of 
 the distance between variables' indices.  Common examples are in stationary time series modeling,  where  $\true_{j,(j+h)}$ depends only on $h$, 
 the lag of the series, with $ -K \leq h \leq K$, for some $K > 0$.  One can  weaken this assumption, allowing $\{\true_{j,(j+h)}:-K\le h\le K\}$
to depend on $j$, while retaining the assumption that the bandwidth
$K$ does not depend on $j$.  Another typical assumption is that  $\true_{jk}=\cov(X_j,X_k)$ decreases with $|j-k|$. A simple example is the moving-average process, where it 
is assumed that 
\begin{align}
  \true_{jk}=0 \text{ if }|j-k|>K,\label{eq:K-banded}
\end{align}
for some bandwidth parameter $K$, and a covariance matrix  $\true$ with this structure is called banded.  Likewise, in a first-order
autoregressive model, the elements decay exponentially with distance
from the main diagonal, $\true_{jk}\propto\beta^{|j-k|}$, where
$|\beta|<1$,  justifying the term ``approximately banded'' used to
describe such a matrix. 

 More generally, banded or approximately banded covariance matrices  can  be used to  model any  generic vector 
$(X_1, \ldots, X_p)$ whose entries are ordered such that any
entries that are more than $K$ apart are uncorrelated (or at most very weakly
correlated).  This situation does not specify any particular decay or ordering of correlation
strengths.  It is noteworthy that $K$ itself is unknown, and may depend on $n$ or $p$ or both.

A number of estimators have been proposed for
this setting that outperform the sample covariance matrix,
$\S=n^{-1}\sum_{i=1}^n(\mathbf X_i-\bar{ \mathbf{X}})(\mathbf X_i-\bar{\mathbf{X}})^T$, 
where $\bar{\mathbf{X}} = n^{-1}\sum_{i=1}^n \mathbf{X}_i$. In this paper we will focus on the 
squared Frobenius $\|\ \|_F^{2}$ and operator $\| \ \|_{op}$  norms of the difference between estimators and the population matrix,  as measures of performance. 
It is immediate to show that  $\frac1{p}\| \S - \true\|^2_F = O\left(\frac{p}{n}\right)$, and that $\|\S - \true\|_{op} = O\left(p\sqrt{\frac{\log p}{n}}\right)$, neither of which can be close to zero for $p > n$.  This can be rectified when one uses, instead, estimators that take into account the structure of $\true$.  For instance, 
 \citet{Bickel08band}  introduced   a
class of banding and tapering estimators of the form
\begin{align}
  \hat \Sigma_T=T*\S,\label{eq:tapered}
\end{align}
where $T$ is a Toeplitz matrix and $*$ denotes Schur multiplication. When the 
matrix $T$ is of the form
$T_{jk}=1\{|j-k|>K\}$, for a pre-specified $K$, one obtains what is referred to as the banded estimator. More general forms of $T$ are allowed, see \citet{Bickel08band}. 
The Frobenius and operator norm optimality of such estimators has been studied relative to classes of approximately banded  population matrices discussed in detail in Section 4.2 below. 
Members of these classes are matrices with entries decaying with distance from the main diagonal at rate depending  on the the sample size $n$ and a parameter $\alpha > 0$.
The minimax lower bounds for estimation over these classes have been established, for both norms, in \citet{Cai10}. 
The banding estimator achieves them in both Frobenius norm, see \citet{Bickel08band},   or operator norm, see  \citet{Xiao14}, 
and the same is true for a more general tapering estimator proposed in \citet{Cai10}. The common element of these estimators is that, while being minimax rate optimal, they are 
not minimax adaptive, in that their construction utilizes knowledge of $\alpha$, which is typically not known in practice. Moreover, there  is no guarantee that the banded estimators
are positive definite, while this  can be guaranteed via appropriate
tapering; however, unlike the banded estimators, the tapering estimators are not banded. 

Motivated by the desire to propose a rate-optimal estimator that does not depend on $\alpha$, \citet{Cai12}  propose
an adaptive estimator that partitions $\S$ into blocks of
varying sizes, and then zeros out some of
these blocks.  They show that this estimator is minimax adaptive in operator norm, 
over a certain class of population matrices. The block-structure form of
their  estimator is an artifact of an interesting proof technique,
which relates the operator norm of a matrix to the operator norm of a
smaller matrix whose elements are the operator norms of blocks of the
original matrix.  The construction is tailored to obtaining optimality in operator norm, and  the estimator leans heavily on the
assumption of decaying covariance away from the diagonal. In
particular, it automatically zeros out all elements far from the diagonal without
regard to the data (so a covariance far from the diagonal could be
arbitrarily large and still be set to zero).  In this sense, the method is less data-adaptive than
may be desirable and may suffer in situations in which the
long-range dependence does not decay fast enough. In addition, as in
the case of the banded estimators, this block-thresholded estimator
cannot be guaranteed to be positive definite.  If positive
definiteness is desired, the authors note that one may project the
estimator onto the positive semidefinite cone (however, this
projection step would lose the sparsity of the original solution).

Other estimators with good practical performance, in both operator and Frobenius norm, have been proposed, notably the Nested Lasso of \citet{Rothman10}. Their approach is 
 to regularize the Cholesky factor via solving a series of weighted
 lasso problems.  The resulting estimator is sparse and positive
 definite; however, a theoretical study of this estimator has not been conducted, and its computation 
 may be slow for large matrices.

\subsection{The convex banding estimator}

In this work we aim to bridge some of the gaps in the existing literature and to provide new insights into usages of estimators with a banded structure.  Our contributions are as follows: \\

\noindent 1. We construct a new estimator that is sparsely-banded and
positive definite with high probability. Our estimator is the solution to a convex optimization problem, 
and, by construction, has a data-dependent bandwidth.  We call our
estimation procedure {\em convex banding}.\\

\noindent 2.  We propose an efficient algorithm for constructing this
estimator and show that it amounts to the tapering of the sample
covariance matrix by a data-dependent matrix.  In constrast to previous
tapering estimators, which require a fixed tapering matrix, this
data-dependent tapering allows our estimator to adapt to the unknown bandwidth of the true covariance matrix. \\

\noindent 3.  We show that our estimator is minimax rate adaptive (up to logarithmic factors) with
respect to the Frobenius norm over a  new class of population matrices
that we term {\em semi-banded}. 
This class  generalizes those of banded or approximately banded
matrices. This establishes our estimator as the first with proved
minimax rate adaptivity in Frobenius norm (up to logarithmic factors) 
over the previously studied covariance matrix classes.  Moreover,
members of the newly introduced  class do not require entries to decay with the distance between their indices.   This extends  the scope of banded covariance estimators.   We also show that 
our estimator is minimax rate optimal (again, up to logarithmic factors) and adaptive with respect to the operator norm over a class of matrices with elements close to banded matrices, with bandwidth 
that can grow with $n$, $p$ or both, at an appropriate rate.  Moreover, we show that our estimators recover, with high probability, the sparsity pattern. \\

An unusual (and favorable) aspect of our estimator is
that it is simultaneously sparse and positive definite with high
probability---a property not shared by any other method with comparable theoretical guarantees.

The precise definition  of our convex banding estimator and a
discussion of the algorithm are given in Sections \ref{sec:proposal} and \ref{sec:computation} below. 
Our target $\true$ either has all elements beyond a
certain bandwidth $K$ being zero or it is close (in a sense defined
precisely in Section \ref{sec:theory}) to a $K$-banded matrix.  We therefore
aim at constructing an estimator that will zero out all covariances that are
beyond a certain data-dependent bandwidth.   If we regard the elements we would like to set to zero as a group, it is 
natural to consider a penalized estimator, with a penalty that zeros out groups.  
The  most basic penalty that sets groups of parameters  to zero simultaneously, without any other restrictions 
on these parameters,  is known as the Group Lasso \citep{Yuan06},
a generalization of the Lasso \citep{Tibshirani96}.
\citet{Zhao09} show that by taking a hierarchically-nested group
structure, one can develop penalties that set  one group  of parameters
to zero whenever another group is zero.  Penalties that render various hierarchical sparsity patterns have been proposed and studied in, for instance, 
 \cite{Jenatton10,Radchenko10,Bach12} and \cite{Bien13}.  The most common applications considered in these works are to regression models. 

The convex banding estimator employs a new hierarchical group lasso
penalty that is tailored to covariance matrices with a banded or semi-banded structure; its optimal properties cannot be obtained from 
simple extensions of any of the existing related penalties. 
We discuss this in Sections \ref{sec:weights} and \ref{sec:theory}.  We also provide a connection between our convex 
banded estimator and  tapering  estimators. Section \ref{sec:tapering} shows that our estimator can also be 
written in the form \eqref{eq:tapered},  but where $T$ is a data-driven, sparse tapering matrix, with entries given by a data-dependent recursion formula, 
not a pre-specified, non-random, tapering function.  This representation has both practical and theoretical implications:
 it allows us to compute our estimator efficiently, relate our
 estimator to previous banded or tapered estimators, and establish that it
 is banded and positive definite with high probability.  These issues are treated in Sections \ref{sec:tapering} and \ref{sec:pd}, respectively. 


In Section \ref{sec:recovery}, we prove that, when the population 
covariance matrix is itself a banded matrix, 
 our estimator recovers the sparsity pattern, under minimal signal strength conditions, which are made 
possible by the fact that we employ a hierarchical penalty.  In
Section \ref{sec:minimax} we show that our convex banding estimator is minimax rate adaptive, in both Frobenius and operator norms, up to 
multiplicative logarithmic factors, over appropriately defined classes of population matrices.  In
Section \ref{sec:empirical}, we perform a thorough simulation study that
investigates the sharpness of our bounds and demonstrates that our estimator compares favorably to previous methods.  We
conclude with a real data example in which we show that using our
estimate of the covariance matrix leads to improved classification
accuracy in both quadratic and linear discriminant analysis.

\section{The definition of the convex banding estimator}

\label{sec:proposal}


To describe our estimator, we must begin
by defining a set of groups that will induce the desired
banded-sparsity pattern.  We define
$$\gl=\{jk\in[p]^2:|j-k|\ge p-\ell\}$$
 to be the two right triangles of
  indices formed by taking the $\ell(\ell+1)$ indices that are
  farthest from the main diagonal of a $p\times p$ matrix. See the left
  panel of Figure \ref{fig:gl} for a depiction of $g_3$ when $p=5$.  
  For notational ease we will denote, as above,  $jk= (j, k)$, and $[p]^2 = \{1, \dots, p\} \times \{1, \ldots, p\}$.
  
  We will
  also find it useful to express these groups as a union of
  subdiagonals $s_m$:
$$
\gl=\bigcup_{m=1}^\ell s_m\quad\text{ where }\quad s_m=\{jk:|j-k|=p-m\}.
$$
\noindent For example, $g_1=s_1=\{(1,p), (p,1)\}$ and, at the opposite
extreme, $g_{p-1}^c=s_p$
is the diagonal.  While the
indexing of $g_\ell$ and $s_m$ may at first seem ``backwards'' (in the
sense that we count them from the outside-in rather than from the
diagonal-out), our indexing is natural here because
$|s_m|=2m$ and $g_\ell$ consists of two  equilateral triangles with  side-lengths of $\ell$ elements. The right panel of Figure \ref{fig:gl} depicts the nested group structure: $g_{1}\subset \cdots\subset
  g_{p-1}$, where this largest group contains all off-diagonal
  elements.
  \begin{figure}
    \centering
    \includegraphics[width=0.3\linewidth]{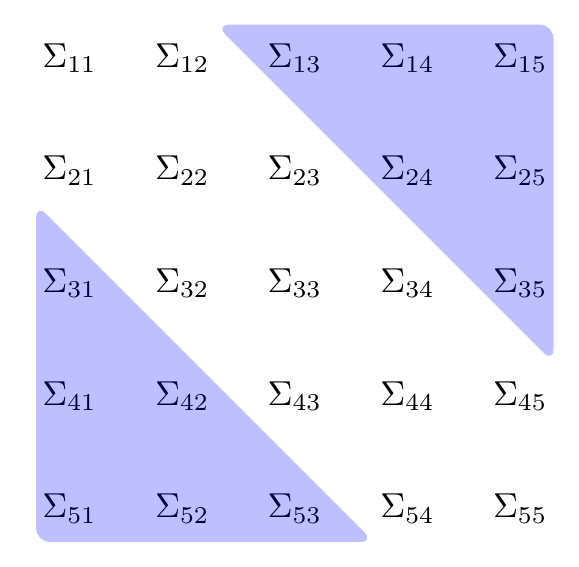}
    \includegraphics[width=0.3\linewidth]{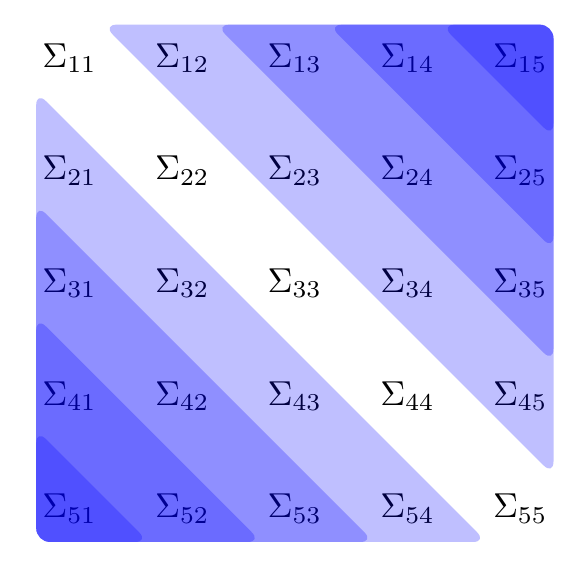}
    \caption{\em (Left) The group $g_3$;  (Right)
      the nested groups in the penalty \eqref{eq:primal}.}
    \label{fig:gl}
  \end{figure}

The following notation is used in the definition of our estimator below. Given a subset of matrix  indices $g\subseteq[p]^2$ of a
  $p\times p$ matrix
  $\Sigma$, let $\Sigma_g\in\real^{|g|}$ be the
  vector with elements $\{\Sigma_{jk}:jk\in g\}$.   For a given non-negative sequence of weights $w_{\ell m}$
  with $1\le m\le\ell\le p$, discussed in the following sub-section,
  and for a given $\lambda\ge0$, let $\hatP$ be defined as the solution to the following convex optimization
  problem:
  \begin{align}
    \label{eq:primal}
    \hatP = \arg\min_{\Sigma}\left\{ \half\|\Sigma-\S\|_F^2+\lambda\|\Sigma\|_{2,1}^* \right \},
  \end{align}
where $\|\cdot\|_F$ is the Frobenius norm and 
\begin{align}
  \label{eq:penalty}
  \|\Sigma\|_{2,1}^*=\sum_{\ell=1}^{p-1}\sqrt{\sum_{m=1}^\ell w_{\ell m}^2\|\Sigma_{s_m}\|_2^2}=\sum_{\ell=1}^{p-1}\|(W^{(\ell)}*\Sigma)_{g_\ell}\|_2.
\end{align}
Our penalty term is a weighted group lasso, using
$\{g_\ell:1\le\ell\le p-1\}$ as
the group structure.  For the second equality, we express the penalty
as the elementwise product, denoted by $*$, of $\Sigma$ with a sequence of
weight matrices, $W^{(\ell)}$,  that are Toeplitz with
$
W^{(\ell)}_{s_m}=
  w_{\ell m}1_{\{m\le\ell\}}\cdot1_{2m}.
$
\begin{rem}
\textnormal{Problem \ref{eq:primal} is strictly convex, so  $\hatP$ is the unique solution.}
\end{rem}
\begin{rem}
\textnormal{As the tuning parameter $\lambda$ is increased, subdiagonals of
$\hatP$ become exactly zero.  As long as $w_{\ell m}>0$ for all
$1\le m\le\ell\le p-1$, the hierarchical group structure ensures that
a subdiagonal will only be set to zero if all elements farther from
the main diagonal are already zero.  Thus, we get an estimated
bandwidth, $\hat K$, which satisfies $\hatP_{g_{p-\hat
    K-1}}=0$ and $\hatP_{s_{p-\hat K}}\neq 0$ (see Theorem
\ref{thm:tapering} and Corollary \ref{cor:banded} for details).}
\end{rem}
We refer to $\hatP$ as the {\em convex banding} of
the matrix $\S$.  We show in Section \ref{sec:pd} that $\hatP$ is positive definite with high probability.  Empirically, we find
that positive definiteness holds except in very extreme cases.
Moreover, in Section \ref{sec:pd} we propose a different version of
convex banding that guarantees positive definiteness:
\begin{align}\label{eq:tilde}
  \tildeP= \arg\min_{\Sigma\succeq\delta I_p}\left\{ \half\|\Sigma-\S\|_F^2+\lambda\|\Sigma\|_{2,1}^*\right\}.
\end{align}
Of course when $\lambda_{\min}(\hatP)\ge\delta$, the two estimators
coincide.

\subsection{The weight sequence $w_{\ell m}$}
\label{sec:weights}

The behavior of our estimator is dependent on the choice
of weights $w_{\ell m}$.  Since for each $\ell$, the weight $w_{\ell m}$ penalizes $s_{m}$, $ 1 \leq m \leq \ell$, and the subdiagonals $s_{m}$ increase in size with $m$, 
we want to give the largest subdiagonals the largest weight.  We will
therefore choose $w_{\ell m}$ with the following property: \\

\begin{property}\label{prop:weights-1}
  For each $1 \leq \ell \leq p - 1$, $w_{\ell\ell} = \max_{1 \leq m
    \leq \ell}w_{\ell m}$. We let $w_{\ell} := w_{\ell\ell}$.
\end{property}

\noindent  We consider three choices of weights satisfying Property
\ref{prop:weights-1} that yield estimators with differing behaviors:\\
\noindent {\em 1. (Non-hierarchical) Group lasso penalty.}
Let \begin{equation}\label{eq:groupw}
w_{\ell m } =1_{\{\ell = m\}} \sqrt{2\ell}, \ \mbox{for} \   1\leq
m\leq \ell, 1\leq \ell \leq p-1.
\end{equation}
Notice that  $w_{\ell} = \sqrt{2\ell}$ and 
$$
\|\Sigma\|_{2,1}^*=  \sum_{\ell=1}^{p-1}\sqrt{2\ell}\|\Sigma_{s_{\ell}}\|_2.
$$
This is a traditional group lasso penalty that acts on
subdiagonals. The size of $w_{\ell}$ is in concordance with the size
of ${s_{\ell}}$. 
Note that
this penalty is not hierarchical, as each subdiagonal may be set to
zero independently of the others. In Section \ref{sec:recovery} we
show how failing to use a hierarchical penalty requires more stringent
minimum signal conditions in order to correctly recover the true
bandwidth of $\true$.  However, if the interest is in accurate estimation in
Frobenius or operator norm of population  matrices that are
close to banded, we show in Section \ref{sec:F-norm} that this
estimator is minimax rate optimal, up to logarithmic factors,  although in finite samples it may fail to have the correct sparsity pattern.\\
\noindent {\em 2. Basic hierarchical penalty.}
If $w_{\ell m} = \sqrt{2\ell}$, for all $m$ and $\ell$, the penalty term becomes
\begin{equation}\label{eq:basicw}
\|\Sigma\|_{2,1}^*=  \sum_{\ell=1}^{p-1}\sqrt{2\ell}\|\Sigma_{g_{\ell}}\|_2,
\end{equation}
which employs the same weight, $\sqrt{2\ell}$, as above, but now for a
triangle $g_{\ell}$. Recalling that $|g_{\ell}|=\ell(\ell+1)$, 
we note that this does not follow the common principle guiding weight
choices in the (non-overlapping) group lasso literature of using $\sqrt{|g_\ell|}$. In
Section \ref{sec:recovery}, 
we show however that $\sqrt{2\ell}$ is indeed the appropriate weight
choice for consistent bandwidth selection under minimal conditions on
the strength of the signal.  It turns out, however, that  this choice of
weights is not refined enough for rate optimal estimation of $\true$
with respect to either Frobenius or operator norm.  In particular,
consider the fact that the subdiagonal $s_m$ is included in $p-m$
terms of the penalty in \eqref{eq:basicw}.  Subdiagonals far from the
main diagonal (small $m$) are thus excessively penalized.  To
balance this overaggressive enforcement of hierarchy, one desires
weights that decay with $m$ within a fixed group $g_\ell$ and yet
still exhibit a similar $\sqrt{2\ell}$ growth on $w_{\ell\ell}$.\\
\noindent {\em 3. General hierarchical penalty}. Based on the considerations
  above, we take the following choice of weights:
\begin{equation}\label{eq:genw}
w_{\ell,m} = \frac{\sqrt {2\ell}}{\ell - m +1},  \ \mbox{for}\  1 \leq m \leq \ell, \  1 \leq \ell \leq p - 1.
\end{equation}

Once again, $w_{\ell} =: w_{\ell \ell}= \sqrt{2\ell}$. We will show in Sections
\ref{sec:F-norm}  and \ref{sec:op-norm} that the corresponding estimator is  minimax rate adaptive, in Frobenius  and operator norm, 
 up  to logarithmic factors, over appropriately defined classes of population covariance matrices.  
 
 \begin{rem}
\textnormal{We re-emphasize the difference between the weighting schemes considered above. 
 The estimators corresponding to \eqref{eq:basicw} and \eqref{eq:genw}
 will always impose the sparsity structure of a banded matrix, a fact
 that is apparent from Theorem \ref{thm:tapering} below. 
In contrast, the group lasso estimator (\ref{eq:groupw}), which is performed
on each sub-diagonal separately,  may fail to recover this pattern.}
\end{rem}

\section{Computation and properties}

\label{sec:computation}
The most common approach to solving the standard group lasso is
blockwise coordinate descent (BCD).  \citet{Tseng01} proves that when the
nondifferentiable part of a convex objective is separable, BCD is
guaranteed to solve the problem.  Unfortunately, this
separable structure is not present in \eqref{eq:primal}.  As in
\citet{Jenatton10}, we consider instead the dual problem, which does
possess this separability property, meaning that BCD on the dual will
work.
\begin{theorem}
A dual of \eqref{eq:primal} is given by
  \begin{align}
  &\minimize_{\Al\in\real^{p\times p}}\half
    \left\|\S-\lambda\sum_{\ell=1}^{p-1}
      W^{(\ell)}*\Al\right\|_F^2\quad\st\quad\|\Al_{\gl}\|_2\le 1,~\Al_{\gl^c}=0\mathrm{~for~}1\le\ell\le p-1.
\label{eq:dual}
\end{align}
In particular, given a solution to the dual, $(\hat A^{(1)},\ldots,\hat
A^{(p-1)})$, the solution to \eqref{eq:primal} is given by
\begin{align}
  \hatP=\S-\lambda\sum_{\ell=1}^{p-1}
  W^{(\ell)}*\hAl.\label{eq:primal-dual}
\end{align}
\label{thm:dual}
\end{theorem}
\begin{proof}
  See Appendix \ref{sec:proof-dual}.
\end{proof}
\noindent Algorithm \ref{alg:psd} gives a BCD algorithm for solving
\eqref{eq:dual}, which by the primal-dual relation in
\eqref{eq:primal-dual} in turn gives a solution to \eqref{eq:primal}.
The blocks correspond to each dual variable matrix.
The update over each $\Al$ involves projection onto an
ellipsoid, which amounts to finding a root of the univariate function,
\begin{align}
  h_\ell(\nu)=\sum_{m=1}^\ell\frac{w_{\ell m}^2}{(w_{\ell
      m}^2+\nu)^2}\|\hat R^{(\ell)}_{s_m}\|^2.\label{eq:nu}
\end{align}
We explain in Appendix \ref{sec:ellipsoid-projection} the details of
ellipsoid projection and observe that we can get $\hat\nu_\ell$ in closed
form for all but the last $\hat K$ values of $\ell$.  A remarkable
feature of our algorithm is that only a single pass of BCD is required
to reach convergence.  This property is proved in \citet{Jenatton11}
and is a direct consequence of the nested structure of the problem.

\begin{algorithm}[t]
\caption{BCD on dual of Problem \eqref{eq:primal}.}
\emph{Inputs:} $\S,~\lambda$, and weights matrices,
$W^{(\ell)}$. Initialize $\hat A^{(\ell)}=0$ for all $\ell$.\\

For $\ell=1,\ldots, p-1$:
  \begin{itemize}
  \item Compute $\hat R^{(\ell)}\leftarrow
    \S-\lambda\sum_{\ell'=1}^{p-1} W^{(\ell')}*\hat A^{(\ell')}$
  \item For $m\le\ell$, set $\hat A_{s_m}^{(\ell)}\leftarrow \frac{w_{\ell m}}{\lambda(w_{\ell m}^2+\max\{\hat\nu_\ell,0\})}\hat
      R^{(\ell)}_{s_m}$ where $\hat\nu_\ell$ satisfies
      $\lambda^2=h_\ell(\hat\nu_\ell)$, as in \eqref{eq:nu}.
  \end{itemize}
$\{\hat A^{(\ell)}\}$ is a solution to \eqref{eq:dual} and $\hat
R^{(p)}=\S-\lambda\sum_{\ell=1}^{p-1} W^{(\ell)}*\hAl$ is the solution to \eqref{eq:primal}. \label{alg:hat}
\end{algorithm}
When we use the simple weights \eqref{eq:basicw}, in which $w_{\ell
  m}=w_\ell$, Algorithm \ref{alg:hat} becomes extraordinarly simple
and transparent:
  \begin{enumerate}
  \item Initialize $\hatP\leftarrow\S$
  \item For $\ell=1,\ldots,p-1$: $\hatP_{\gl}\leftarrow(1-\lambda w_\ell/\|\hatP_{\gl}\|_2)_+\hatP_{\gl}$.
  \end{enumerate}
In words, we start with the sample covariance matrix, and then work
from the corners of the matrix inward
toward the main diagonal.  At each step, the next largest
triangle-pair is group-soft-thresholded.  If a triangle is ever set to
zero, it will remain zero for all future steps.  We will show that
this simple weighting scheme admits exact bandwidth and pattern recovery in
Section \ref{sec:recovery}.

\subsection{Convex banding as a tapering estimator}
\label{sec:tapering}
The next result shows that $\hatP$ can be regarded as a tapering
estimator with a data-dependent tapering matrix.  In Section
\ref{sec:theory}, we will see that, in contrast to estimators that use
a fixed tapering matrix, our estimator adapts to the unknown bandwidth
of the true matrix $\true$.


  \begin{theorem}
    The convex banding estimator, $\hatP$, can be written as a tapering
    estimator with a Toeplitz, data-dependent tapering matrix, $\hatP=\hat
    T*\S$:
$$
\hat T_{s_m}=\begin{cases} 
  1_m&\text{ for }m=p \mathrm{~(diagonal)}\\
\prod_{\ell=m}^{p-1}\frac{[\hat\nu_\ell]_+}{w_{\ell m}^2+[\hat\nu_\ell]_+}1_m&\text{ for
  }1\le m\le p-1
\end{cases}
$$
where $\hat\nu_\ell$ satisfies $\lambda^2=\sum_{m=1}^\ell\frac{w_{\ell m}^2}{(w_{\ell
      m}^2+\hat\nu_\ell)^2}\|\hat R^{(\ell)}_{s_m}\|^2$ and
  $1_m\in\real^m$ is the vector of ones.
\label{thm:tapering}
\end{theorem}
\begin{proof}
By Proposition 5 in \citet{Jenatton11}, we can get $\hatP$ by a
single pass as in Algorithm \ref{alg:hat}.  We begin with $\hat
R^{(1)}=\S$ and then for $\ell=1,\ldots,p-1$, (and for each
$m\le\ell$), we have
\begin{equation}\label{recur}
  \hat R^{(\ell+1)}_{s_m}=\hat R^{(\ell)}_{s_m} - \lambda w_{\ell m}\hat
  A^{(\ell)}_{s_m}=\frac{[\hat\nu_\ell]_+}{w_{\ell m}^2+[\hat\nu_\ell]_+}\hat
  R^{(\ell)}_{s_m}.
\end{equation}
The optimality conditions give $\hatP = \hat R^{(p)}$, so that we have
\begin{equation}\label{esthat}
\hatP_{s_m}=\prod_{\ell=m}^{p-1}\frac{[\hat\nu_\ell]_+}{w_{\ell m}^2+[\hat\nu_\ell]_+}\cdot\S_{s_m}
\end{equation}
which establishes this as an adaptively tapered estimator. 

\end{proof}
\noindent The following result shows that, as desired, our estimator is banded. 

\begin{cor}
  $\hatP$ is a banded matrix with bandwidth $\hat
  K=p-1-\max\{\ell:\hat\nu_\ell\le0\}$.
\label{cor:banded}
\end{cor}
\begin{proof}
  By definition, $\hat\nu_{p-1-\hat K}\le0$ and $\hat\nu_{p-\hat
    K},\ldots,\hat\nu_{p-1}>0$.  It follows from the theorem that
  $\hat T_{s_m}=0\cdot1_m$ for all $m\le p-1-\hat K$ and $[\hat
  T_{s_m}]_i>0$ for $p-\hat K\le m\le p-1$.
\end{proof}



\section{Statistical properties of the convex banding estimator}

  \label{sec:theory}
In this section, we study the statistical properties of the convex
banding estimator.  We begin by stating two assumptions that
specify the general setting of our theoretical analysis.\\


\noindent {\it Assumption 1.} \
Let $\mathbf{X} = (X_1,\dots, X_p)^T$. Assume $\mathbb{E} \mathbf{X} = \mathbf{0}$ 
 and denote $\mathbb{E}\mathbf{X}\mathbf{X}^T = \true$.   We assume that each $X_j$ is 
marginally sub-Gaussian:
$$\mathbb{E}\exp(tX_j/\sqrt{\true_{jj}})\leq \exp(Ct^2)$$
for all $t\geq 0$ and for some  constant $C > 0$ that is independent of  $j$. Moreover, $\max_{ij}|\true_{ij}|\le M$,  for some constant $M > 0$. \\

\noindent {\it Assumption 2. } The dimension $p$  can grow with $n$ at most exponentially in $n$:  $\gamma_0 \log n \leq \log p \leq \gamma n$,  for some constants $\gamma_0>0, \gamma>0$.\\

In Section \ref{sec:recovery}, we prove
that our estimator recovers the true
bandwidth of $\true$ with high probability assuming the nonzero
values of $\true$ are large enough.  We will demonstrate that our 
estimator can detect lower signals than what could be recovered by  an
estimator that does not enforce hierarchy,  re-emphasizing  the need for  a  hierarchical penalty.  In Section
\ref{sec:minimax}, we  show that our estimator is minimax adaptive, up to multiplicative logarithmic terms, 
with respect to both the operator and Frobenius norms. Our results hold  either with high probability or in expectation,
and are established over classes of population matrices defined in Sections \ref{sec:F-norm}  and \ref{sec:op-norm}, respectively. 


We begin by introducing a random set, on which all our results hold. Fixing $x>0$, let
\begin{align}
  \mathcal A_x=\left\{\max_{1\leq i, j\leq p} \left|\S_{ij}-\true_{ij}\right|\leq
  x\sqrt{\log p/n}\right\}.\label{eq:whp}
\end{align}
The following lemma shows that this set has high
probability. 
\begin{lem}
Under Assumptions 1 and 2,  there exists a constant $c>0$ such that for sufficiently large $x>0$, 
$$
\mathbb{P}(\mathcal A_x) \geq 1-\frac{c}{p}.
$$
\label{lem:whp}
\end{lem}
\begin{proof}
See Appendix \ref{sec:lem-whp},  proof of (\ref{lemma1}) of Theorem \ref{thm:whp}.
\end{proof}
\begin{rem}
\textnormal{ (i) Similar results exist in the literature.  Lemma 3 in \cite{Bickel08band} is proved under a Gaussianity assumption coupled with the assumption that $\|\true\|_{op}$ is bounded.
Whereas inspection of the proof shows that the latter is not needed, we  cannot  quote this result  directly for other types of design. The 
commonly employed assumptions  of sub-Gaussianity are placed on the entire vector ${\bf X} = (X_1, \ldots, X_p)^T$  and postulate that there exists $\tau > 0$ 
such that 
\[ \mathbb{P}\{ |v^T({\bf X} - \mathbb{E}({\bf X}))| > t \} \leq e^{-t^2/(2\tau)}, \ \mbox{for all} \ t > 0 \ \mbox{and} \  \|v\|_2 = 1,\]
see, for instance, \cite{Cai12breg}.  However, if such a $\tau$ exists, then $\|\true\|_{op} \leq \tau$. Lemma \ref{lem:whp}  above shows that a bound on  $\|\true\|_{op}$
 can be avoided in the probability bounds regarding $\max_{ij}\left|\S_{ij}-\true_{ij}\right|$, and the distributional assumption can be weakened to marginals. }
 
\textnormal{ (ii)  In the classical case in which $p$  does not depend on   $n$, one can modify the definition of the set $\mathcal{A}_x$ by replacing the factor $\log \ p$ by $\log \max({n,p})$,  and the result of Lemma \ref{lem:whp} will continue to hold with probability $1 - c/\max({n, p})$,  for a possibly different constant $c$. }
\end{rem}


\subsection{Exact bandwidth recovery}
\label{sec:recovery}
Suppose $\true$ has bandwidth $K$, that is, for
$L=p-K-1$, we have
$\true_{g_{L}}=0$ and $\true_{s_{L+1}}\neq0$ (see Figure
\ref{fig:L-and-K}).
\begin{figure}
  \centering
  \includegraphics[width=0.4\linewidth]{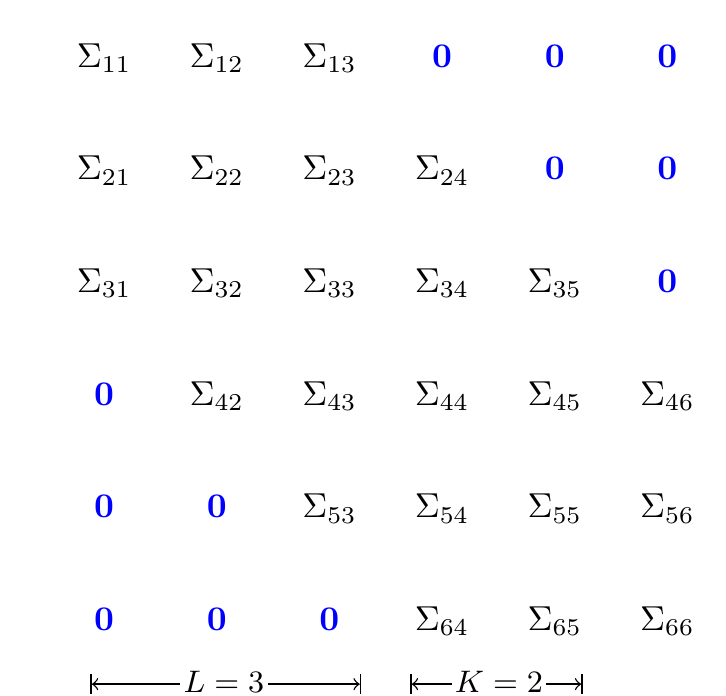}
  \caption{A $K$-banded matrix where $K=2$ and $p=6$.  Note that $L=p-K-1=2$ counts the number of subdiagonals that are zero.  This can be expressed as $\true_{g_L}=0$ or $\true_{s_\ell}=0$ for $1\le\ell\le L$. }
  \label{fig:L-and-K}
\end{figure}
We prove in this
section that under mild conditions our estimator $\hatP$ correctly recovers $K$ with high
probability.  The next theorem expresses the intuitive result that if
$\lambda$ is chosen sufficiently large, we will not over-estimate the
true bandwidth.
\begin{theorem}
If $\lambda\ge x\sqrt{\log p/n}$ and $w_\ell=\sqrt{2\ell}$, then
$
\hat K\le K, 
$
on the set $\mathcal{A}_{x}$. 
\label{thm:recovery1}
\end{theorem}
\begin{proof}
  See Appendix \ref{sec:bandwidth-recovery-proof}.
\end{proof}
For our estimator to be able to detect the full band, we must require
that the ``signal'' be sufficiently large relative to the noise.  In
the next theorem we measure the size of the signal by the $\ell_2$
norm of each sub-diagonal (scaled in proportion to the square root of its size, $w_\ell=\sqrt{2\ell}$) and the size of the noise by $\lambda$. 
\begin{theorem}
If 
\begin{equation}
\label{signal}
\min_{\ell\ge L+1} \|\true_{s_{\ell}}\|_2/w_{\ell} >2\lambda,
\end{equation}
where $\lambda \geq x\sqrt{\log p/n}$ and $w_{\ell}=\sqrt{2\ell}$, then
$
K\le \hat K
$
on the set $\mathcal A_x$.
\label{thm:recovery2}
\end{theorem}

\begin{proof}
  See Appendix \ref{sec:bandwidth-recovery-proof}.
\end{proof}
In typical high-dimensional statistical problems, the support set can
generically be any subset and therefore one must require
that each element in the support set be sufficiently large on its
own to be detected.  By contrast, in our current context we know that the
support set is of the specific form $\{L+1,\ldots, p-1\}$, for some
unknown $L$.  Thus, as long as the signal is sufficiently large at $L+1$, one
might expect that the signal could be weaker at subsequent elements
of the support.  In the next theorem we demonstrate this phenomenon by
showing that when $\true_{s_{L+1}}$ exceeds the threshold given in
the previous theorem, it may ``share the wealth'' of this excess,
relaxing the requirement on the size of $\true_{s_{L+2}}$.
\begin{theorem}[``Share the wealth'']
Suppose (for some $\gamma>0$)
\begin{align*}
  \|\true_{s_{L+1}}\|_2&=(2+\gamma)\lambda w_{L+1}\\
\min_{\ell\ge L+3} &\|\true_{s_{\ell}}\|_2/w_{\ell} >2\lambda\\
\|\true_{s_{L+2}}\|_2&>
\begin{cases}
\lambda w_{L+2}(1+\sqrt{1-\gamma^2})  &\text{ for }0<\gamma<1\\
0 & \text{ for }\gamma\ge1,
\end{cases}
\end{align*}
 where $\lambda \geq x\sqrt{\log p/n}$ and $w_{\ell} = \sqrt{2\ell}$ and
$w_{\ell,\ell}\ge w_{\ell+1,\ell}>0$,
then
$
K\le\hat K
$
on the set $\mathcal A_x$.
\label{thm:recovery3}
\end{theorem}

\begin{proof}
  See Appendix \ref{sec:bandwidth-recovery-proof}.
\end{proof}
\noindent When $\gamma=0$, Theorem \ref{thm:recovery3} reduces to Theorem
\ref{thm:recovery2}.  However, as $\gamma$ increases,
the required size of $\|\true_{s_{L+2}}\|_2$ decreases without preventing
the bandwidth from being misestimated.  In fact, for $\gamma\ge1$,
there is no requirement on $\true_{s_{L+2}}$ for bandwidth recovery.  This robustness to individual
subdiagonals being small is a direct result of our method being
hierarchical.
\begin{rem}
\textnormal{(i)  By Lemma \ref{lem:whp}, $\mathbb{P}(\mathcal{A}_{x}) > 1 - cp^{-1}$, under Assumption 1 and 2. Thus, all the results 
of this section hold with this probability. \\
(ii)  Theorems  \ref{thm:recovery1}  and   \ref{thm:recovery2}  apply to
  all three weighting schemes considered in Section \ref{sec:weights}.  Theorem \ref{thm:recovery3}'s additional requirement of
  positivity excludes the ``Group lasso'' weights \eqref{eq:groupw}, which is non-hierarchical
  and therefore is unable to ``share the wealth.''  }
\end{rem}

 \subsection{Minimax adaptive estimation}
\label{sec:minimax}

\subsubsection{Frobenius norm rate optimality}
\label{sec:F-norm}
In this section we show that the estimator $\hatP$, 
with weights given by either \eqref{eq:groupw} or  \eqref{eq:genw}, is minimax
adaptive,  up to multiplicative logarithmic factors,   with respect to the Frobenius norm, over a class of
population covariance matrices that generalizes both the class of
$K$-banded matrices and previously studied classes of approximately banded matrices. We begin by stating Theorem \ref{thm:oracle}, which is a general oracle inequality from which  adaptivity to the optimal minimax rate will follow as immediate corollaries. \\

\noindent  For any $B\in\mathbb{R}^{p\times p}$, 
let $L(B)$ be such that $B_{g_{L(B)}} = 0$ and
$B_{s_{L(B)+1}}\neq 0$, that is $B$ has bandwidth $K(B) =p-1-L(B)$.
  Let $\mathcal{S}_p$ be
the class of all $p\times p$ positive definite covariance matrices. 
In Theorem \ref{thm:deter} of  Appendix \ref{sec:Fnorm-deter}, we provide a deterministic upper
bound on  $\| \hatP -\true \|_F^2$ that indicates what terms need to
be controlled in order to obtain optimal stochastic performance of our
estimator.  The following theorem, which is the main result of this
section, is a consequence of  this theorem, and shows that for appropriate choices of the weights and tuning parameter, 
the proposed estimator achieves the best bias-variance trade-off both in probability and in expectation, up to small additive terms, which are the price paid for adaptivity. 

\begin{theorem}\label{thm:oracle}
Suppose $\true\in \mathcal{S}_p$ and $\max_{i, j} |\true_{ij}| \leq
M$ for some constant $M$. If the weights are given by either
\eqref{eq:groupw} or \eqref{eq:genw} and $\lambda = x\sqrt{\log p/n}$,
then on the set $\mathcal A_x$ the convex banding estimator of
 \eqref{eq:primal} satisfies
\begin{equation*}
\| \hatP -\true \|_F^2  \leq \inf_{B\in\mathbb{R}^{p\times p}} \left\{\|\true - B\|_F^2 +16x^2  \frac{K(B)p\log p}{n}\right\} + x^2\frac{p\log p}{n}.
\end{equation*}
Moreover, for $x$
sufficiently large
, 
\begin{equation*}
\mathbb{E}\| \hatP -\true \|_F^2 \lesssim \inf_{B\in\mathbb{R}^{p\times p}} \left\{ \|\true-B\|_F^2 + \frac{K(B)p\log p}{n} \right\} +\frac{p}{n},
\end{equation*}
where the symbol $\lesssim$ is used to denote an inequality that holds up to positive multiplicative constants that are independent of $n$ or $p$. 
\end{theorem}

\noindent  For  any given $K \geq 1$,  the class of
{\em $K$-banded} matrices is
\begin{align*}
\mathcal{B}(K) = \mathcal{B}(K, M) =: \left\{ \Sigma\in \mathcal{S}_p: \Sigma_{g_{p-K-1}}=0,\,  \Sigma_{s_{p-K}}\neq0  \mbox{ and } \max_{ij}|\Sigma_{ij}| \leq M  \right\}.
\end{align*} 
Motivated by Theorem \ref{thm:oracle} above, we define the following  general class of covariance matrices, henceforth referred to as {\em $K$-semi-banded}:
\begin{align}\label{eq:genbanded}
\mathcal{G}(K)&= \mathcal{G}(K, M)\\
&:= \left\{ \Sigma \in \mathcal{S}_p: \max_{ij}|\Sigma_{ij}| \leq M  \mbox{ and }  \|\Sigma - B\|_{F}^{2} \lesssim 
pK\log p/n \mbox{ for some } B\in\mathcal{B}(K, M) \right\},\nonumber
\end{align}
for any $K \geq 1$.   Trivially, 
\[ \mathcal{B}(K) \subset \mathcal{G}(K), \ \mbox{for any} \ K \geq 1.\]
Notice that a  semi-banded covariance matrix  does not have any
explicit restrictions on the order of magnitude of $\|\Sigma\|_{op}$;
nor does it require exact zero entries.\\

Theorem \ref{thm:Flower}  below  gives the minimax lower bound, in
probability, for estimating a covariance matrix $\true \in
\mathcal{B}(K)$, in Frobenius norm, from a sample $\mathbf X_1,  \ldots, \mathbf X_n$.  Let $\mathbb{P}_{K}$ be the class of probability distributions of $p$-dimensional vectors satisfying Assumption 1 and with covariance matrix in $\mathcal{B}(K)$. 
\begin{theorem}\label{thm:Flower}
If $K < \sqrt{n}$, there exist absolute constants $c > 0$ and $\beta \in (0, 1)$ such that 
\begin{equation}
\inf_{\hat{\Sigma}}\sup_{\mathbb{P}_{K}}P\left( \| \hat \Sigma  -\true \|_F^2 > cp K/n\right) \geq \beta.
\end{equation}
\end{theorem}
To the best of our knowledge, the  minimax lower bound  for Frobenius norm estimation over  $\mathcal{B}(K)$ is  new.  The  proof of the existing related result in  \cite{Cai10}, established for approximately banded matrices, or that in \citet{Cai12breg}, for {\it approximately sparse}  matrices, could be adapted to this situation, but the resulting proof would be involved, as it would  require re-doing all their calculations. We provide in Appendix \ref{sec:proof-theor-Flower}  a different and much shorter proof, tailored to our class of  covariance matrices.

From Theorems \ref{thm:oracle}  and \ref{thm:Flower} above we conclude
that convex banding with either the (non-hierarchical) group lasso penalty, \eqref{eq:groupw}, or the
general hierarchical penalty, \eqref{eq:genw}, achieves, adaptively,
the minimax rate, up to log terms, not only over $\mathcal{B}(K)$, but over the larger class $\mathcal{G}(K)$ of semi-banded matrices.  We summarize this below. 

\begin{cor}\label{cor:Fnorm}
 Suppose $\true\in\mathcal G(K)$.  The convex banding estimator of
 \eqref{eq:primal}, with weights given by either \eqref{eq:groupw} or
 \eqref{eq:genw}, and with $\lambda = x\sqrt{\log p/n}$ satisfies 
 \begin{equation*}
\| \hatP -\true \|_F^2  \lesssim pK\log p/n
 \end{equation*}
on the set $\mathcal A_x$
and also
$$ \E\| \hatP -\true \|_F^2 \lesssim pK\log p/n.$$
\end{cor}
\noindent  To the best of our knowledge, this is the first estimator which,  in particular, 
has been proven to be minimax adaptive over the class of exactly banded matrices, up to logarithmic factors.

It should be emphasized that, despite certain similarities, the problem we are studying is different
from that of estimating the autocovariance matrix of a stationary
time series based on the observation of a single realization of the series.  For a
discussion of this distinction and a treatment of this latter problem,
see \citet{Xiao12}.



\begin{rem}
\textnormal{An immediate consequence of Theorem \ref{thm:deter} given in Appendix
\ref{sec:Fnorm-proof} is that for the basic hierarchical penalty, \eqref{eq:basicw}, we have,  on the set $\mathcal{A}_x$,
\begin{equation*}
\frac{1}{p}\| \hatP -\true \|_F^2  \lesssim K^2\log p/n,
 \end{equation*}
for $\true\in\mathcal B(K)$, which is a sub-optimal rate. Therefore,
if Frobenius norm rate optimality is of interest and one desires a
banded estimator, then the simple weights \eqref{eq:basicw} do not suffice, and one must resort to the more general ones given by  \eqref{eq:genw}.}
\end{rem}

\noindent The class of  semi-banded matrices is more general
than the previously studied class of approximately banded matrices
considered in \citet{Cai10}:
\begin{equation}
\begin{split}
\label{C}
 \mathcal{C}_{\alpha} = \Big \{\Sigma\in\mathcal{S}_p: \,\, & |\Sigma_{ij}| \leq M_1|i-j|^{-(\alpha+1)} \mbox{ for all }  i\neq j
   \text{ and }  \lambda_{\max}(\Sigma) \leq M_0\Big\}.
 \end{split}
\end{equation}
Note that \citet{Bickel08band} and \citet{Cai12} consider a closely related class,
\begin{equation}
\label{D}
\begin{split}
\mathcal{D}_{\alpha} = \Big\{\Sigma\in\mathcal{S}_p:\,\,& \max_j \sum_{i:\,\, |i-j|>k} |\Sigma_{ij}| \leq M_1 |i-j|^{-\alpha}, \text{ for all }  k>0,\\
&\text{ and } 0 < M_0^{-1}\leq \lambda_{\min}(\Sigma) \leq \lambda_{\max}(\Sigma) \leq M_0 \Big\}.
\end{split}
\end{equation}

Interestingly, the class of approximately banded matrices $\mathcal{C}_{\alpha}$  does not include, in general, the class of exactly banded matrices $\mathcal{B}(K)$, as 
membership in these classes forces neighboring entries to decay
rapidly, a restriction that is not imposed on elements of
$\mathcal{B}(K)$.  Moreover,  the largest eigenvalue of a $K$-banded matrix can be of order $K$, and hence cannot be  bounded by a constant when $K$ is allowed to grow with $n$ and $p$. Therefore, the class 
of approximately banded matrices is different from, but not
necessarily more general than, $\mathcal{B}(K)$, for arbitrary values
of $K$.  



In contrast, the class $\mathcal{G}(K)$ introduced 
in \eqref{eq:genbanded} above is more general than $\mathcal{B}(K)$, for all $K$, and also contains $\mathcal{C}_{\alpha}$ for an appropriate value of $K$. Let 
$K_\alpha\lesssim \left(\frac{n}{\log
      p}\right)^{\frac{1}{2\alpha + 2}}$.
Then, it is immediate to see that, for any $\alpha > 0$,  
\[ \mathcal{C}_{\alpha} \subset   \mathcal{G}(K_{\alpha}).\] 
To gain further insight into particular types of matrices in  $
\mathcal{G}(K_{\alpha})$, notice that it contains, for instance,
matrices in which neighboring elements do not have to start decaying
immediately, but only when they are far enough from the main diagonal. Specifically, for some positive constants $M_1 < M_0$ and any $\alpha > 0$, let
\label{abanded}
\begin{equation}
\begin{split}
\mathcal{G}_{\alpha} =  \Big\{\Sigma\in \mathcal{S}_p: &\,\,  |\Sigma_{ij}| \leq M_0, \ \mbox{if} \  |i - j| \leq K_{\alpha} \text{ and }\\
&\,\,  |\Sigma_{ij}| \leq M_1|i-j-K_{\alpha}|^{-(\alpha+1)},  \text{ if  } |i - j| > K_{\alpha} \Big\}.
\end{split}
\end{equation}
Then we  have, for any fixed $\alpha > 0$ and appropriately adjusted constants, that
\[ \mathcal{C}_{\alpha} \subset \mathcal{G}_{\alpha} \subset  \mathcal{G}(K_{\alpha}) \  \ \ \ \  \mbox{and} \ \ \ \ \  \mathcal{B}(K_{\alpha}) \subset  \mathcal{G}_{\alpha} \subset   \mathcal{G}(2K_{\alpha}). \]


\noindent The following corollary gives the rate of estimation over $\mathcal{G}(K_{\alpha})$. 
\begin{cor}\label{semib}
Suppose $\true \in \mathcal{G}(K_{\alpha})$. The convex banding estimator of
 \eqref{eq:primal},  with weights given by either \eqref{eq:groupw} or
 \eqref{eq:genw}, and with $\lambda = x\sqrt{\log p/n}$, where $x$ is
 as in Theorem \ref{thm:oracle}, satisfies 
 \begin{align*}
(i)& \ \  \frac1{p}\| \hatP -\true \|_F^2 \lesssim\left( \frac{\log
    p}{n}\right)^{\frac{2\alpha + 1}{2\alpha + 2}} \ \mbox{on the set
} \mathcal{A}_x \\
 \text{ and } (ii)& \ \ \mathbb{E}\| \hatP -\true \|_F^2 \lesssim  p \left( \frac{\log p}{n}\right)^{\frac{2\alpha + 1}{2\alpha + 2}}.
 \end{align*}
\end{cor} 
 The upper bound in (ii), up to logarithmic factors,  is the minimax lower bound derived in \citet{Cai10} over $\mathcal{C}_{\alpha}$ and $\mathcal{D}_{\alpha}$. Therefore, Corollary \ref{semib} shows that 
the minimax rate, with respect to the Frobenius norm, can be achieved, {\it adaptively} over $\alpha >0$, not only over the classes of approximately banded matrices, but also over the larger class of semi-banded matrices  $\mathcal{G}(K_{\alpha})$. In contrast, the two existing estimators  with established minimax rates of convergence in Frobenius norm require 
knowledge of $\alpha$ prior to estimation. This is the case for the
banding estimator of \citet{Bickel08band},  
and the tapering estimator of \citet{Cai10}, over the class  $\mathcal{D}_{\alpha}$. 
 Adaptivity to the minimax rate,  over $\mathcal{D}_\alpha$,  with respect to the Frobenius norm,  is alluded to in Section 5 of \citet{Cai12}, who proposed a block-thresholding estimator, but a  formal statement and proof are left open.


\subsubsection{Operator norm rate optimality}
\label{sec:op-norm}
In this section we study the operator norm optimality of our estimator over classes of matrices that are close, in operator norm, to banded matrices, with 
bandwidths that might increase with $n$ or $p$. 
\[
\mathcal{G}_{op}(K)  =\left \{\Sigma\in \mathcal{S}_p: \,\, \text{there exists} \ B\in \mathcal{B}(K) \mbox{ such that }
\|\Sigma - B\|_{op}\lesssim K\sqrt{\log p/n} \right\}.
\]

The following theorem gives the minimax lower bound, in probability, for estimating a covariance matrix $\true \in \mathcal{G}_{op}(K)$, in operator norm, from a sample $\mathbf X_1,  \ldots, \mathbf X_n$.  Let $\mathbb{P}_{K}$ be the class of probability distributions of $p$-dimensional vectors satisfying Assumption 1 and with covariance matrix in $\mathcal{B}(K)$.

\begin{theorem}\label{Olower}
If $K < \sqrt{n}$, there exist absolute constants $c > 0$ and $\beta \in (0, 1)$ such that 
\begin{equation}
\inf_{\hat{\Sigma}}\sup_{\mathbb{P}_{K}}P\left( \| \hat \Sigma  -\true \|_{op} > cK/\sqrt{n}\right) \geq \beta.
\end{equation}
\end{theorem}
Minimax lower bounds in operator norm over the classes of approximately sparse matrices \citep{Cai12breg} and approximately banded matrices  \citep{Cai10}  are provided in 
the cited works. Moreover, an  inspection of the 
proof  of Theorem 3 in \cite{Cai10} shows that the proof of the bound stated  in Theorem \ref{Olower} above follows immediately from it, by the same arguments, 
by letting their $k$ be  our $K$. We therefore state the bound for clarity and completeness, and omit  the proof.  The following theorem shows that our proposed estimator achieves adaptively, up to logarithmic factors, the target rate. 
\begin{theorem}\label{thm:opnormexact2}
The convex banding estimator of
 \eqref{eq:primal} with weights given by either \eqref{eq:groupw} or
 \eqref{eq:genw}, and with $\lambda = 2x\sqrt{\log p/n}$, where $x$ is
 as in Theorem \ref{thm:oracle}, satisfies 
 \begin{equation*}
\| \hatP -\true \|_{op}  \lesssim K\sqrt{\log p/n},\ \mbox{on the set } \mathcal{A}_x,
 \end{equation*}
for any  $\true \in \mathcal{G}_{op}(K)$ that meets the signal strength condition
\[
\min_{\ell\in S} \|\true_{s_{\ell}}\|_2/\sqrt{2\ell} \geq c\,\,\mbox{ for some constant } c>0.
\]
\end{theorem}
\begin{proof}
  See Appendix \ref{sec:proof-onorm}.
\end{proof}
\begin{rem} {\bf A counter example}\label{counter_example}
\textnormal{If the signal strength condition fails and $p > n$, then
  the estimator $\hatP$  may not  even be consistent in operator
  norm.  The following example illustrates this.
Let $\true$ be such that $\true_{ii}=1$ for all $i$,
$\true_{12}=\true_{21}=0.5$, and $\true_{ij}=0$ otherwise.
Then $\|\true_{s_{p-1}}\|_2/\sqrt{2(p-1)} = o(1)$
and the signal
strength condition cannot hold when $p$ grows.  For the above $\true$, $S = \{p-1, p\}$. For the estimator in Theorem \ref{thm:opnormexact2} with $\lambda = 2x\sqrt{\log p/n}$, 
$\hat S \subseteq \{p-1, p\}$  by Theorem \ref{thm:recovery1}.  Then, by (\ref{recur}) and (\ref{esthat}) of Theorem \ref{thm:recovery1} above, and recalling that $w_{p-1} = \sqrt{2(p-1)}$,  we have 
\[ \hatP_{s_{p-1}} = \frac{ [\hat{\nu}_{p-1}]_{+} }  {   [\hat{\nu}_{p-1}]_{+} +2 (p-1)} \S_{s_{p-1}},\]
where $\hat{\nu}_{p-1}$ satisfies
\begin{equation} \label{1} \lambda^2 = \frac{2(p-1) }{(2(p-1) + \hat{\nu}_{p-1})^2}  \| \S_{s_{p-1}} \|^2_2.\end{equation}
Notice  that  
we have 
\[
\|\S_{s_{p-1}}\|_2 \leq \|\true_{s_{p-1}}\|_2 +
\|\S_{s_{p-1}} -\true_{s_{p-1}}\|_2\leq \sqrt{2\times 0.5^2} +
\|\S_{s_{p-1}} -\true_{s_{p-1}}\|_2 < \lambda \sqrt{2(p-1)}
\]
with high probability, for large $p$.  Then,  (\ref{1}) implies that  $\hat{\nu}_{p-1} \leq 0$, in which case 
$
\hatP_{s_{p-1}} = 0
$,  so  $\|\hatP-\true\|_{op} \geq |\true_{12}|=0.5$, and the estimator cannot be consistent. 
}
\end{rem}


Our focus in this section was on the behavior of our estimator over a class that generalizes that of banded matrices, such as   $\mathcal{G}_{op}$,  since no minimax adaptive estimators in operator norm have been constructed, to the best of our knowledge, for this class. Our estimator is however slightly sub-optimal, in operator norm, over 
different classes, for instance over the class  of approximately banded matrices.  It is immediate to show that it achieves,  adaptively, the slightly suboptimal  rate $O\{(\log p/n)^{\frac{\alpha}{2\alpha + 2}}\}$ over the class $\mathcal{D}_{\alpha}$,  when a  signal strength condition similar to that given in  Theorem \ref{thm:opnormexact2} is satisfied.
However, the minimax optimal rate over this class is  $O\{(\log p/n)^{\frac{\alpha}{2\alpha + 1}}\}$,
and has been shown to be achieved, non-adaptively, 
by the tapering estimator of  \citet{Cai10}
and in  \citet{Xiao14}, by the banding estimator. For the latter case,
a SURE-type bandwidth selection is proposed,  as an alternative to the
usual cross validation-type criteria, for the practical bandwidth choice.  Moreover, the optimal rate of convergence
in operator norm can be achieved, adaptively, 
by  a block-thresholding estimator proposed by \citet{Cai12breg}. For
these reasons, we do not further pursue the issue of operator norm minimax adaptive estimation over $\mathcal{D}_{\alpha}$ in this work. 

\subsection{Positive definite, banded covariance estimation}
\label{sec:pd}
In this section, we study the positive definiteness of our estimator.
Empirically, we have observed $\hatP$ to be positive definite except in
quite extreme circumstances (e.g., very low $n$ settings).  We begin
by proving that $\hatP$ is positive definite with high probability
assuming certain conditions.

\begin{theorem}
Under Assumptions 1 and 2, the convex banding estimator with weights given by either \eqref{eq:groupw} or
 \eqref{eq:genw}  and with $\lambda = 2x\sqrt{\log p/n}$ (for appropriately large $x$) has
minimum eigenvalue at least $C'K\sqrt{\log p/n}>0$ with probability larger than $1 - cp^{-1}$, for some $c > 0$, provided that  
\[ (i) \  \lambda_{\min}(\true)\ge2C' K\sqrt{\log p/n}, \]
and \[ (ii) \ 
\min_{\ell\in S} \|\true_{s_{\ell}}\|_2/\sqrt{2\ell} \geq c'\,\,\mbox{ for some constant } c' >0.
\]
\label{thm:psd}
\end{theorem}
\begin{proof}
  See Appendix \ref{sec:pd-proofs}.
\end{proof}
Although we find empirically that $\hatP$ is reliably positive
definite, it is reasonable to desire an estimator that is {\em always}
positive definite.  In such a case, we suggest using the estimator
$\tildeP$ defined in \eqref{eq:tilde}.  Algorithm \ref{alg:psd} of
Appendix \ref{sec:pd-proofs} provides a BCD approach similar to
Algorithm \ref{alg:hat}.
\begin{cor}
Under the assumptions of Theorem \ref{thm:psd}, $\tildeP=\hatP$ with
high probability as long as we choose $\delta \le C'K\sqrt{\frac{\log p}{n}}$.
\end{cor}

We conclude this section by noting that, in light of the corollary,
$\tildeP$ has all the same properties as $\hatP$ under the
additional assumptions of Theorem \ref{thm:psd}. 

\section{Empirical study}

\label{sec:empirical}
\subsection{Simulations}
\label{sec:simulations}

\subsubsection{Convergence rate dependence on $K$}
\label{sec:conv-rate-depend-K}

We fix $n=100$, and for $p\in\{500,1000,2000\}$ and ten equally
spaced values of $K$ between 10 and 500 we simulate data according to
an instance of a moving-average($K$) process.  In particular, we draw $n$ multivariate normal
vectors with covariance matrix
$$
\true_{ij}=
\begin{cases}
  1-\frac{1}{K}|i-j|&\text{ for }|i-j|\le K\\
  0&\text{ for }|i-j|> K.
\end{cases}
$$
We draw 20 samples of size $n$ to compute the average scaled squared
Frobenius norm, $\|\hatP-\true\|_F^2/p$, and operator norm, $\|\hatP-\true\|_{op}$.  Figures \ref{fig:rates-F} and \ref{fig:rates-op} show how these two
quantities vary with $K$ when we take $\lambda=2\sqrt{\log(p)/n}$.  Both are seen to grow approximately linearly in $K$ in agreement with our bounds for these
quantities given in Corollary \ref{cor:Fnorm} and Theorem \ref{thm:opnormexact2}.  In the right panel
of each figure, we scale the Frobenius and operator norms by the
functions of $p$ appearing on the right-hand side of Corollary \ref{cor:Fnorm} and Theorem \ref{thm:opnormexact2}.  The fact
that with this scaling the curves become more aligned shows that the
$p$ dependence of the bounds holds (approximately).  Also, we see that
the behavior matches that of the bounds most closely for large $p$ and small $K$, 
\begin{figure}
  \centering
  \includegraphics[width=0.45\linewidth]{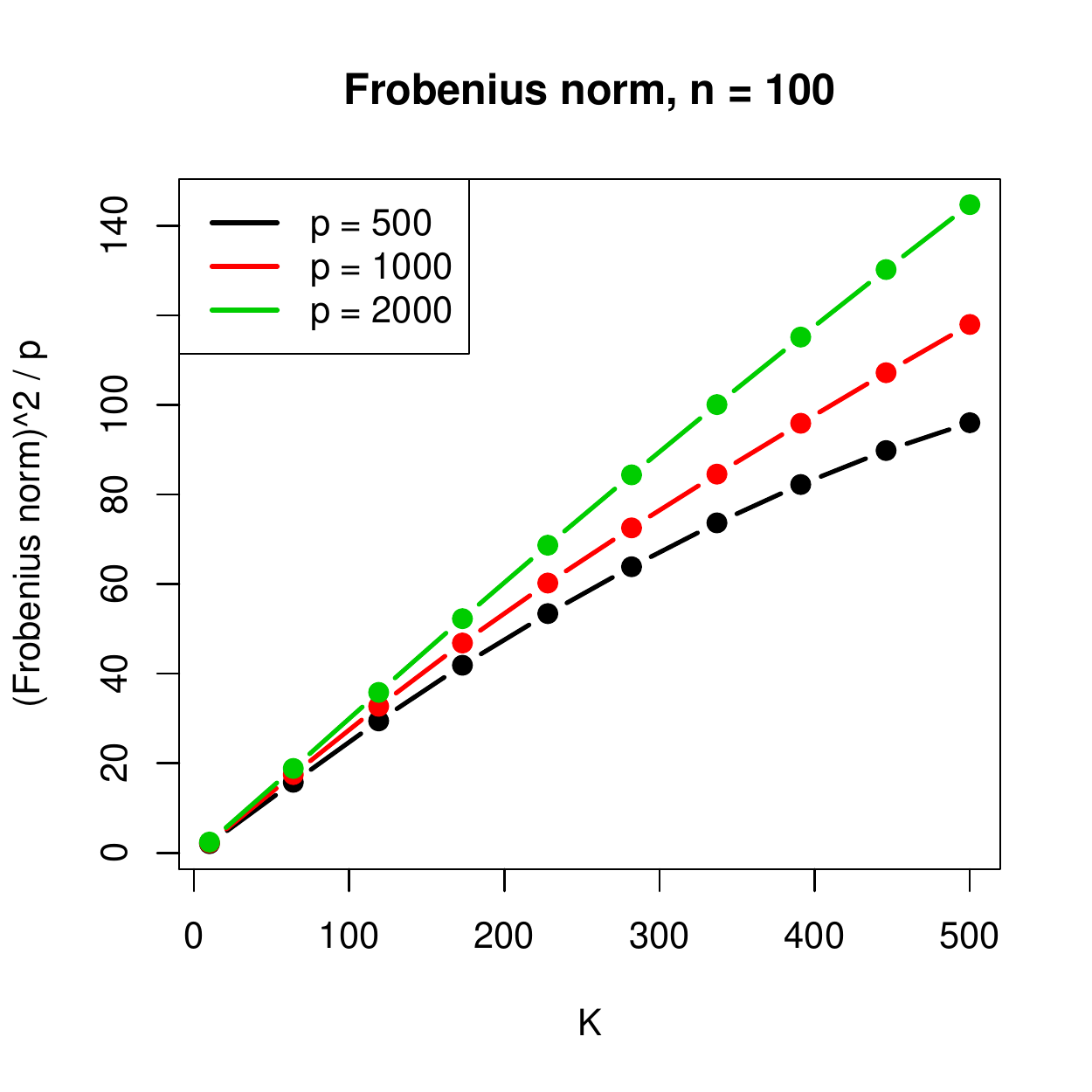}
  \includegraphics[width=0.45\linewidth]{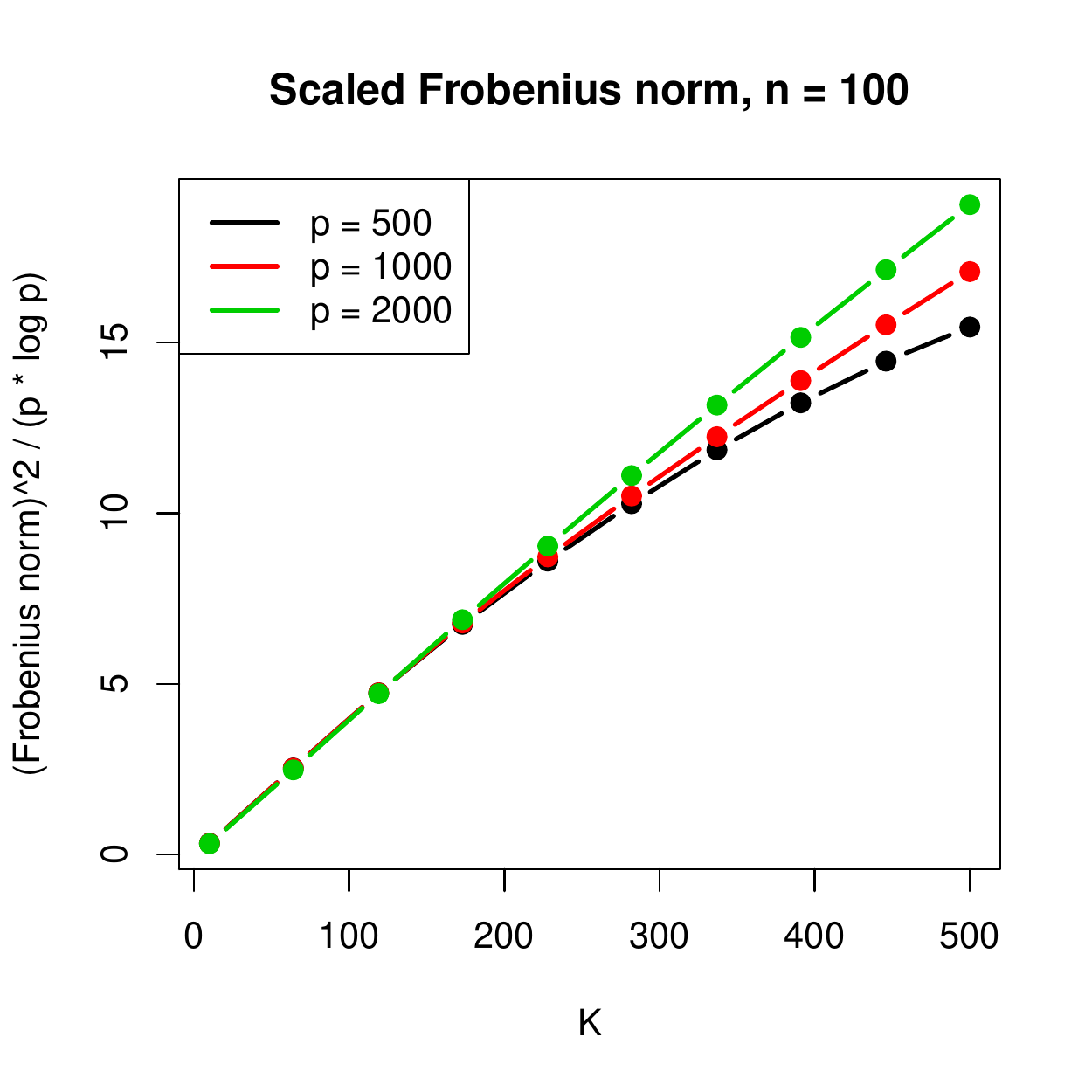}
  \caption{\em Convergence in Frobenius norm. Monte Carlo estimate of
    (Left) $E\left[\|\hatP-\true\|_F^2\right]/p$ and (Right) $E\left[\|\hatP-\true\|_F^2\right]/(p\log p)$ as a function of $K$, the bandwidth of $\true$.}
  \label{fig:rates-F}
\end{figure}
\begin{figure}
  \centering
  \includegraphics[width=0.45\linewidth]{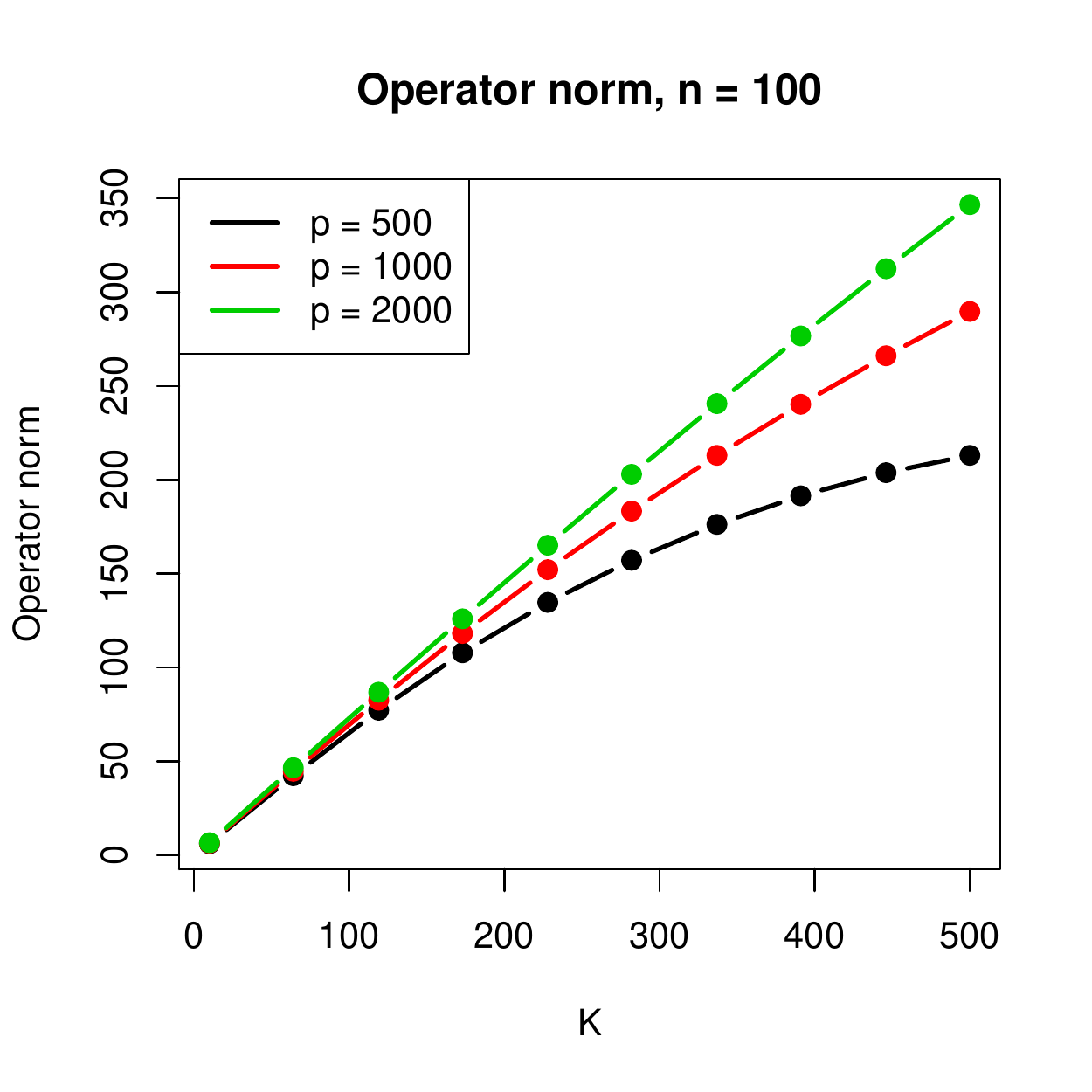}
  \includegraphics[width=0.45\linewidth]{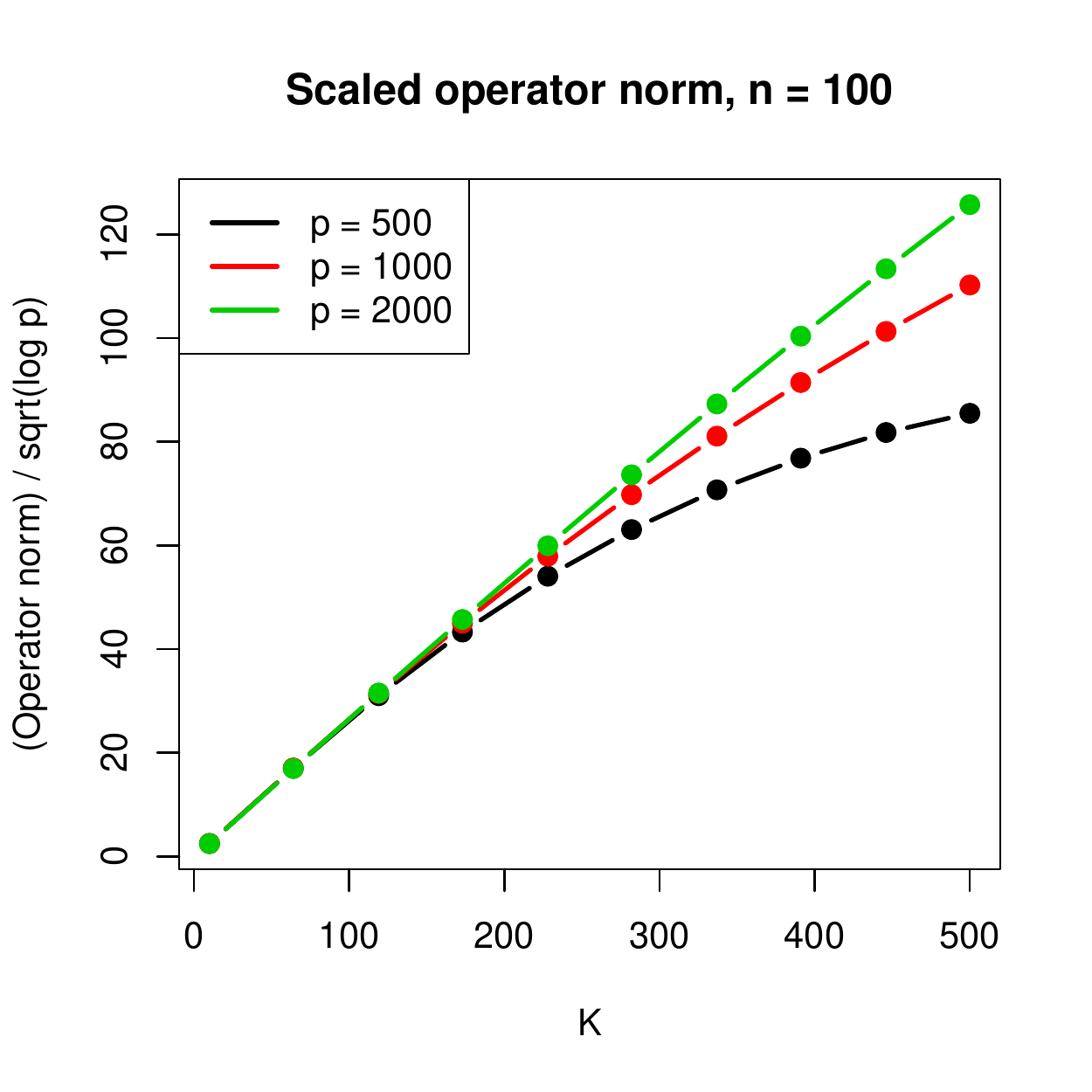}
  \caption{\em Convergence in operator norm. Monte Carlo estimate of
    (Left) $E\left[\|\hatP-\true\|_{op}\right]$ and (Right) $E\left[\|\hatP-\true\|_{op}\right]/\sqrt{\log p}$ as a function of $K$, the bandwidth of $\true$.}
  \label{fig:rates-op}
\end{figure}
In Figure \ref{fig:rates_simple} we repeat the same simulation for $\hatP$
with the simple weighting scheme, $w_{\ell m}=\sqrt{\ell}$.  We find
that this weighting scheme performs empirically much worse than the
general weighting scheme of \eqref{eq:genw}.
\begin{figure}
  \centering
  \includegraphics[width=0.45\linewidth]{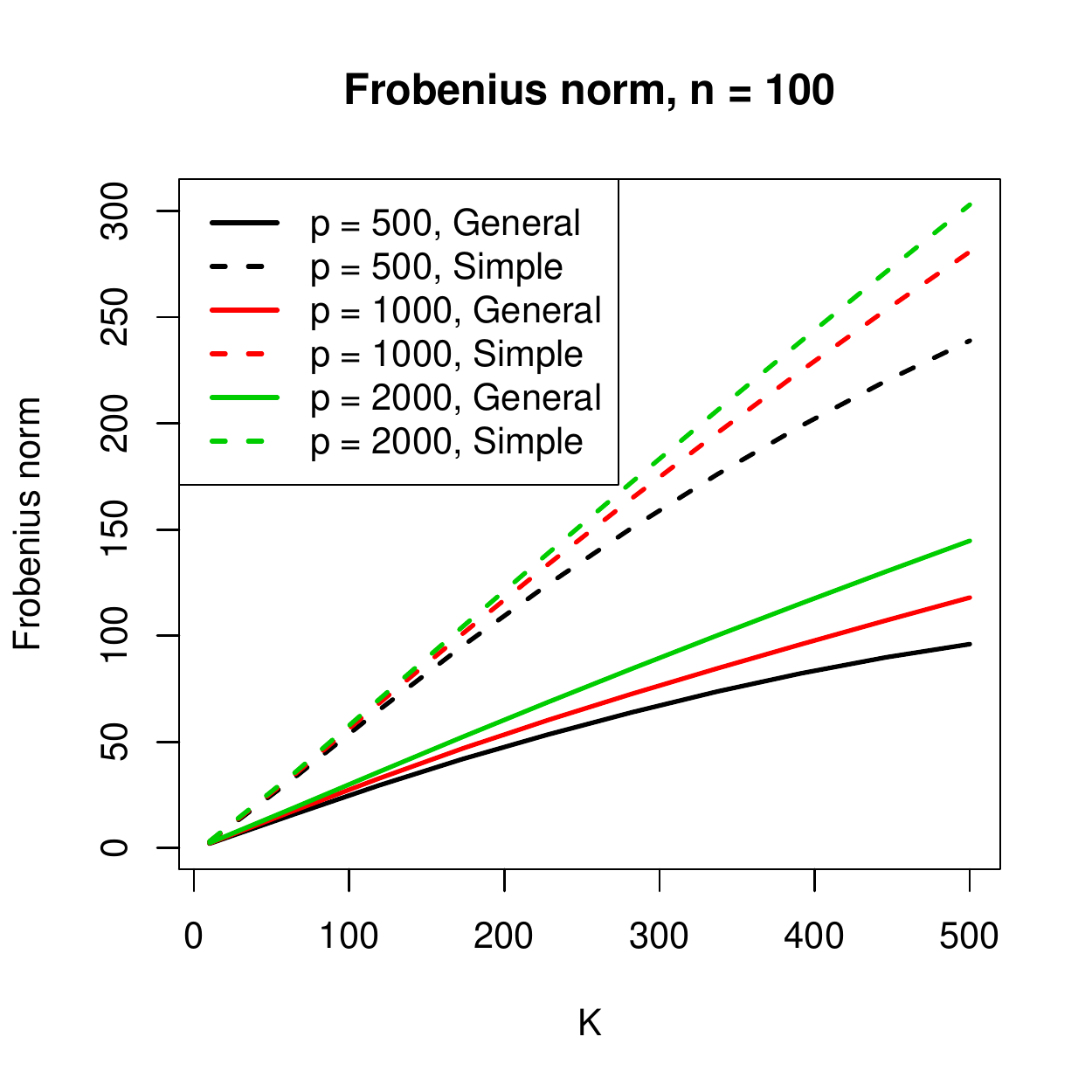}
  \includegraphics[width=0.45\linewidth]{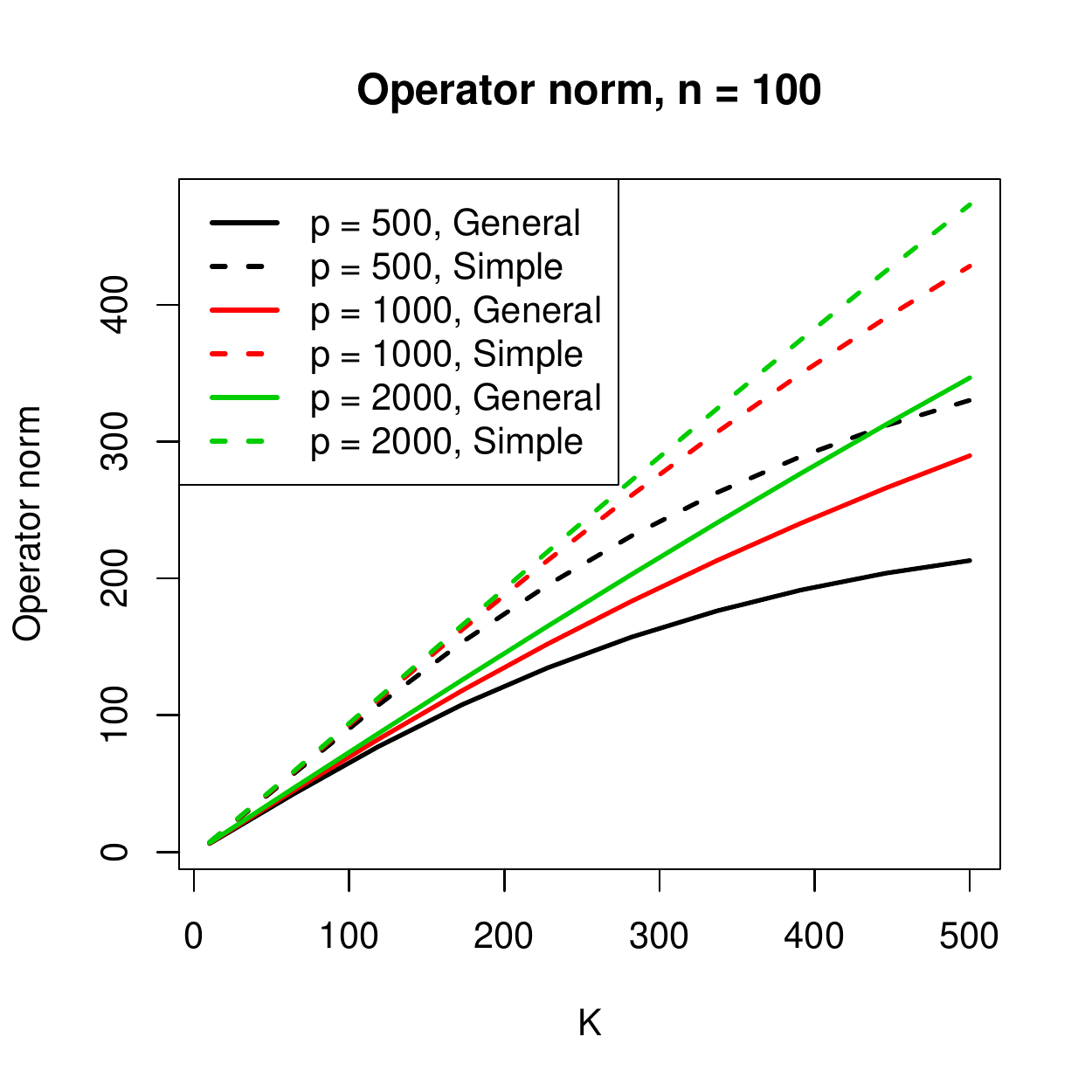}
  \caption{\em Comparison of our procedure with simple and general weights.}
  \label{fig:rates_simple}
\end{figure}
\subsubsection{Convergence rate dependence on $n$}
\label{sec:conv-rate-depend-n}

To investigate the dependence on sample size, we vary $n$ along an equally-spaced grid
from 10 to 500, fixing $K=50$ and taking, as before,
$p\in\{500,1000,2000\}$.  Simulating as before and taking, again
$\lambda=2\sqrt{\log(p)/n}$, Figures \ref{fig:rates-F-varyn} and
\ref{fig:rates-op-varyn} exhibit the $1/n$ and $1/\sqrt{n}$ dependence
suggested by the bounds in Corollary \ref{cor:Fnorm} and Theorem \ref{thm:opnormexact2}.

\begin{figure}
  \centering
  \includegraphics[width=0.45\linewidth]{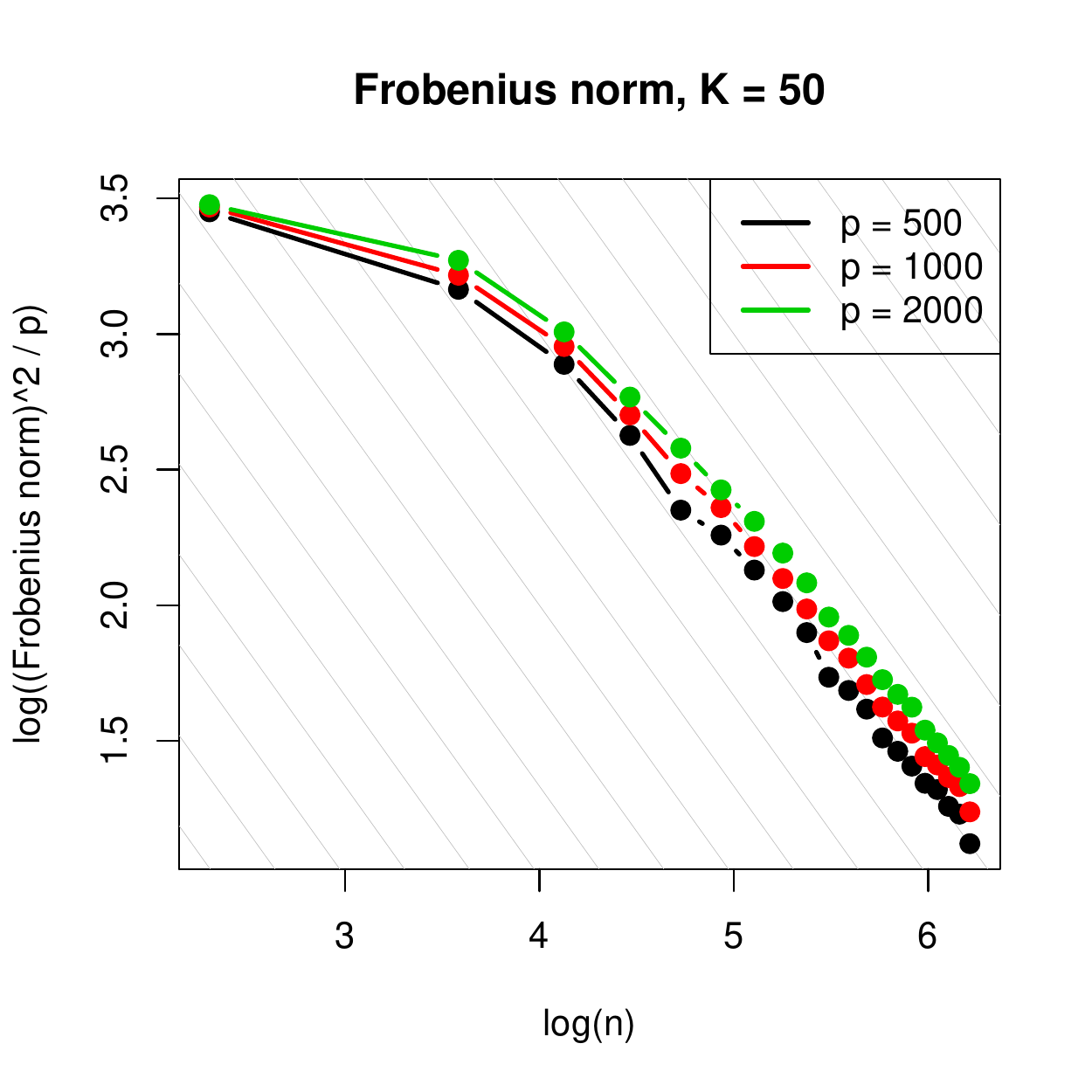}
  \includegraphics[width=0.45\linewidth]{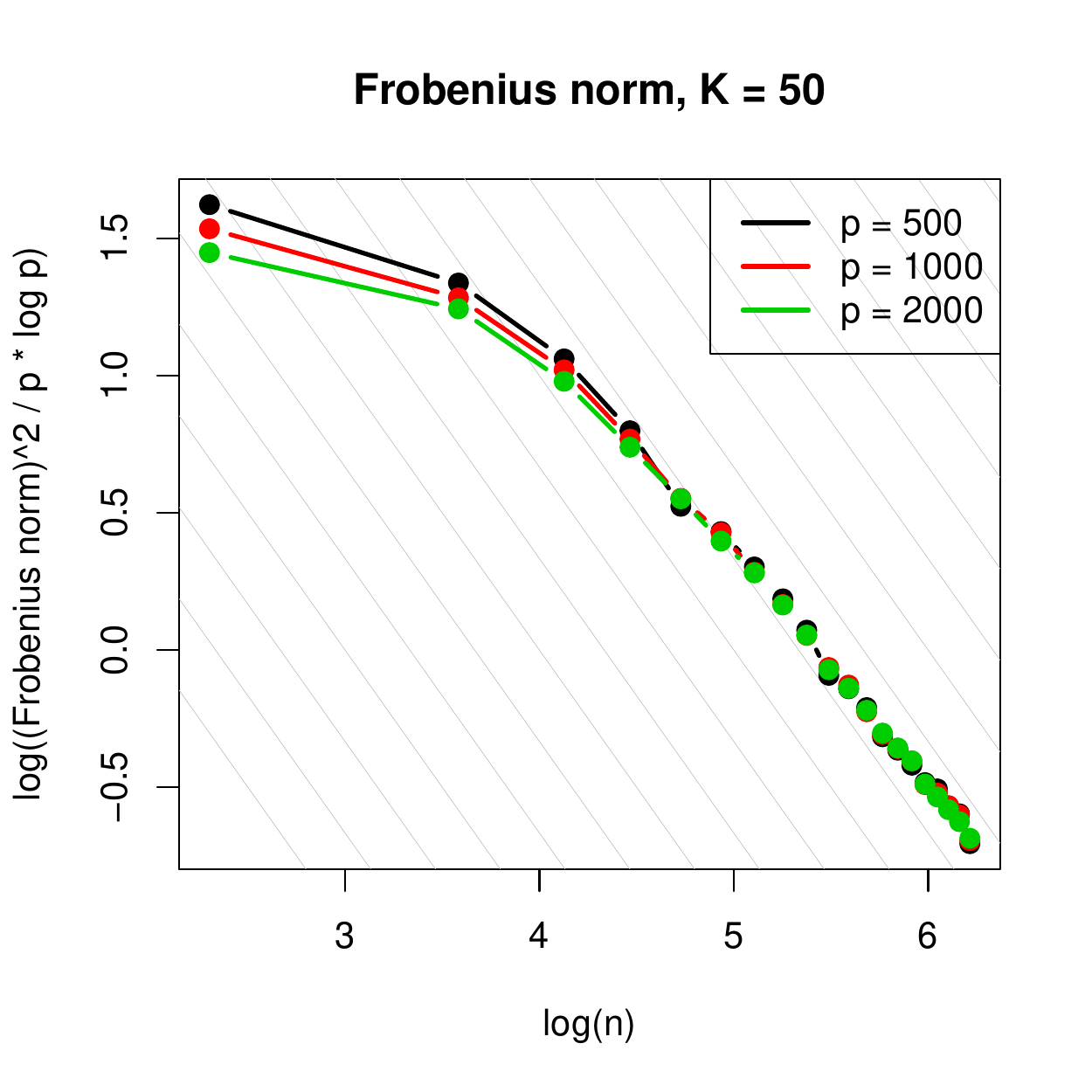}
  \caption{\em (Left:) For large $n$, expected $E\left[\|\hatP-\true\|_F^2\right]/p$ is
    seen to decay like
    $n^{-1}$ as can be seen from $-1$ slope of $\log$-$\log$ plot
    (gray lines are of slope $-1$).
    (Right:) Scaling this quantity by $\log p$ aligns the curves for
    large $n$.  Both of these phenomena are suggested by Corollary \ref{cor:Fnorm}.}
  \label{fig:rates-F-varyn}
\end{figure}
\begin{figure}
  \centering
  \includegraphics[width=0.45\linewidth]{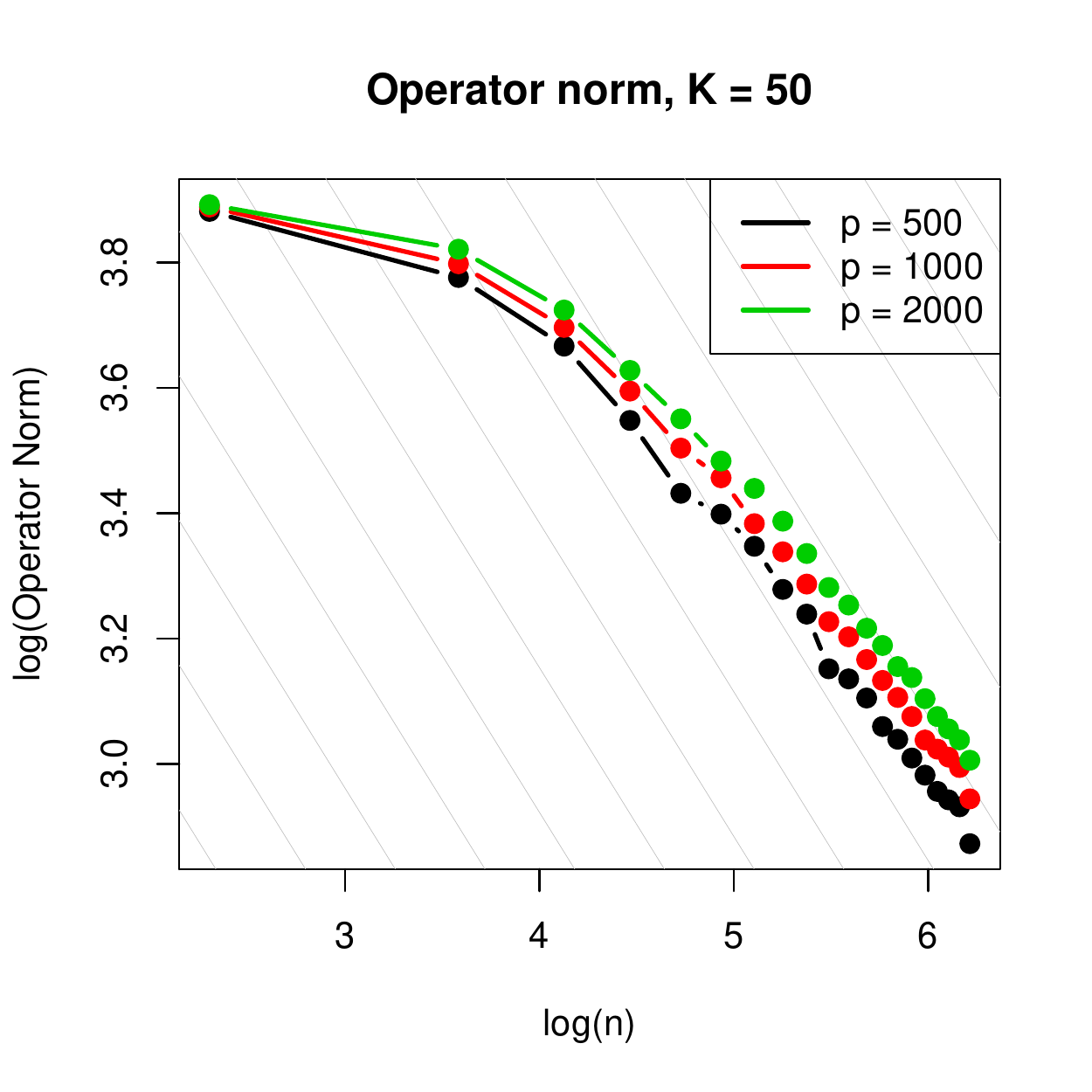}
  \includegraphics[width=0.45\linewidth]{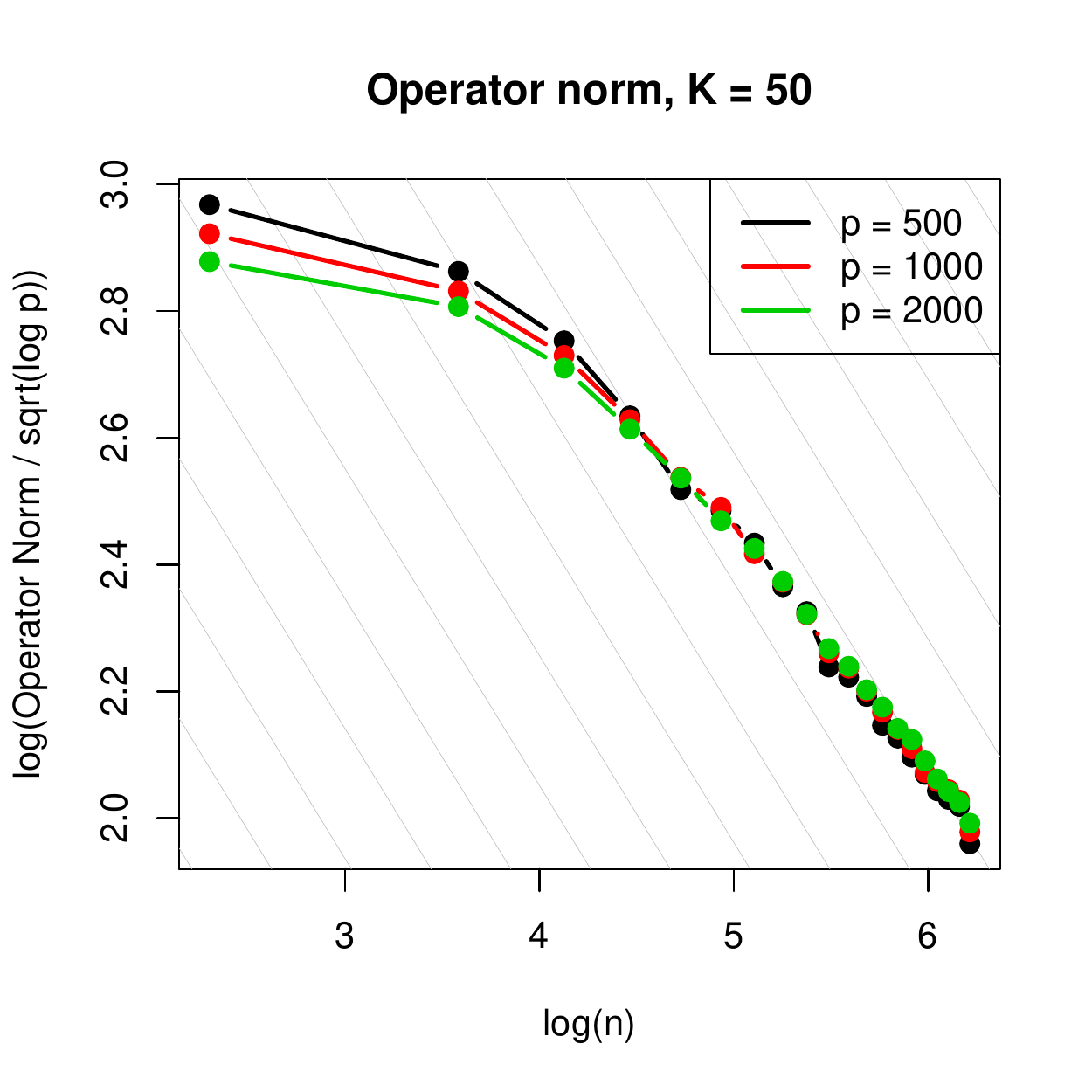}
   \caption{\em (Left:) For large $n$, expected $E\|\hatP-\true\|_{op}$ is
    seen to decay like
    $n^{-1/2}$ as can be seen from $-1/2$ slope of $\log$-$\log$ plot
    (gray lines are of slope $-1/2$).
    (Right:) Scaling this quantity by $\sqrt{\log p}$ aligns the curves for
    large $n$.  Both of these phenomena are suggested by Theorem \ref{thm:opnormexact2}.}
 \label{fig:rates-op-varyn}
\end{figure}

\subsubsection{Comparison to other banded estimators}
\label{sec:comp-other-band}
In this section, we compare the performance of our estimator to
that of the following methods that perform banded covariance
estimation:
\begin{itemize}
\item {\tt Hier Band}: our method with weights given in
  \eqref{eq:genw} and no eigenvalue constraint.
\item {\tt Simple Hier Band}: our method with weights given in
  \eqref{eq:basicw} and no eigenvalue constraint.
\item {\tt GP Band}: group lasso without hierarchy (weighting scheme
  given in \eqref{eq:groupw}) and no eigenvalue constraint.
\item {\tt Block Thresh}: \citet{Cai12}'s block-thresholding
  approach (our implementation).
\item {\tt Nested Lasso}: \citet{Rothman10}'s approach that
  regularizes the Cholesky factor via solving a series of weighted
  lasso problems.
\item {\tt Banding}: \citet{Bickel08band}'s $T*\S$ estimator
  with $T_{jk}=1_{\{|j-k|\le K\}}$.
\end{itemize}
\noindent Each method has a single tuning parameter that is varied to
its best-case value.  In the case of {\tt Banding}, this makes it
equivalent to an oracle-like procedure in which all non-signal
elements are set to zero and all signal elements are left unshrunken.

\noindent While we focus on estimation error as a means of comparison, it
should be noted that these methods could be evaluated in other ways as well.  For example,
{\tt Banding} and {\tt Block Thresh} frequently lead to estimators
that are not positive semidefinite.  This means that these covariance estimates cannot be
directly used in other procedures requiring a covariance matrix (see,
e.g., Section \ref{sec:data} for an example where not being positive
definite would be problematic).  Another consideration is computation
time, in which regard {\tt Nested Lasso} suffers.

We simulate three scenarios for the basis of our comparison:
  \begin{itemize}
  \item {\tt MA(5)}: with $p=200$, $n=100$.  See Section
    \ref{sec:conv-rate-depend-K} for description.
  \item {\tt CY}: A setting used in \citet{Cai12} in which $\true$ is only
    approximately banded:
    $\true_{ij}=0.6|i-j|^{-2}U_{ij}$, where $U_{ij}=U_{ji}\sim
    \text{Uniform}(0,1)$.  We take $p=200$, $n=100$.
  \item {\tt Increasing operator norm}: We produce a sequence of
    $\true$ having increasing operator norm.  In particular, we take
    $\true_{ij}=0.81\{|i-j|\le K\}+c1\{i=j\}$, where $c$ is chosen so
    that $\lambda_{\min}(\true)=0.1$. We take $p=400, n=100$ and vary
    $K$ from 2 to 200.
  \end{itemize}
We can make several observations from this comparison.  We find that
the hierarchical group lasso methods outperform the regular group
  lasso.  This is to be expected since in our simulation scenarios,
  the true covariance matrix is banded (or approximately banded).
As noted before, since we report the best performance of each method over all
  values of the tuning parameter, {\tt Banding} is effectively an
  ``oracle'' method in that the banding is being performed exactly
  where it should be for {\tt MA(5)}.  So it is no
  surpise that this method does so well in this scenario.  However,
  in the approximately banded scenario it suffers, likely because it
  does not do as well as methods that do shrink the nonzero values.
We find that the {\tt Block Thresh} does not do exceptionally well in
any of these scenarios (though it should be noted that it does still
substantially improve on the sample covariance matrix).
Finally, we find that the {\tt Nested Lasso} performs very well
(although we find it to be computationally much slower than the rest).

Figure \ref{fig:divop} shows the results for the third scenario, in
which we consider a sequence of $\true$'s having increasing $K$ and
therefore increasing operator norm.  In terms of Frobenius norm, we
find that {\tt Hier Band} does best.  {\tt Banding} performs very well in
terms of operator norm, which makes sense since $\true$ is a banded
matrix and so simple banding with the correct bandwidth is essentially
an oracle procedure.  It is worth noting that {\tt Block Thesh}
does not do too poorly in terms of operator norm for a wide range of $K$
in contrast to its poor performance in the first two scenarios.

\begin{figure}
  \centering
  \includegraphics[width=0.325\linewidth]{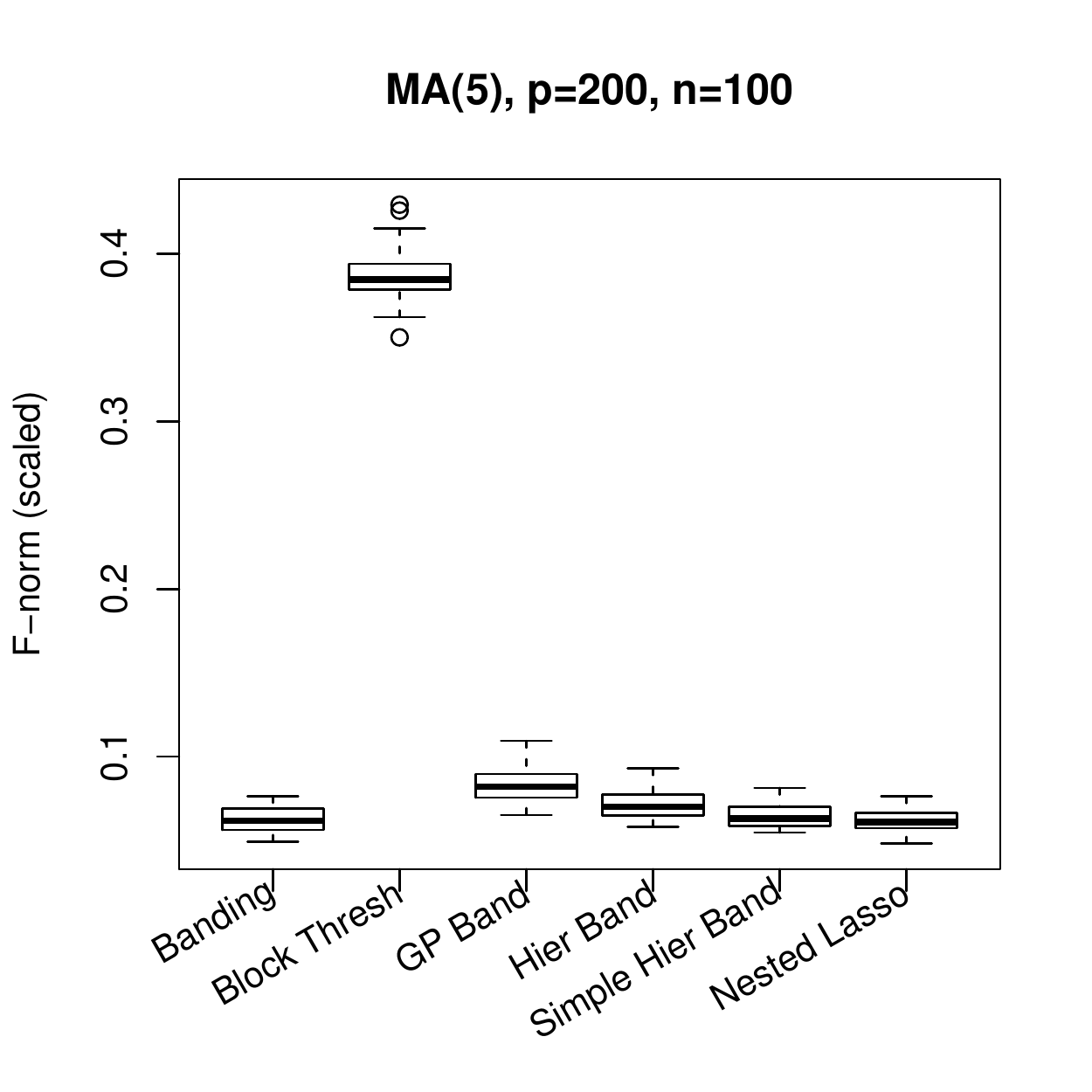}
  \includegraphics[width=0.325\linewidth]{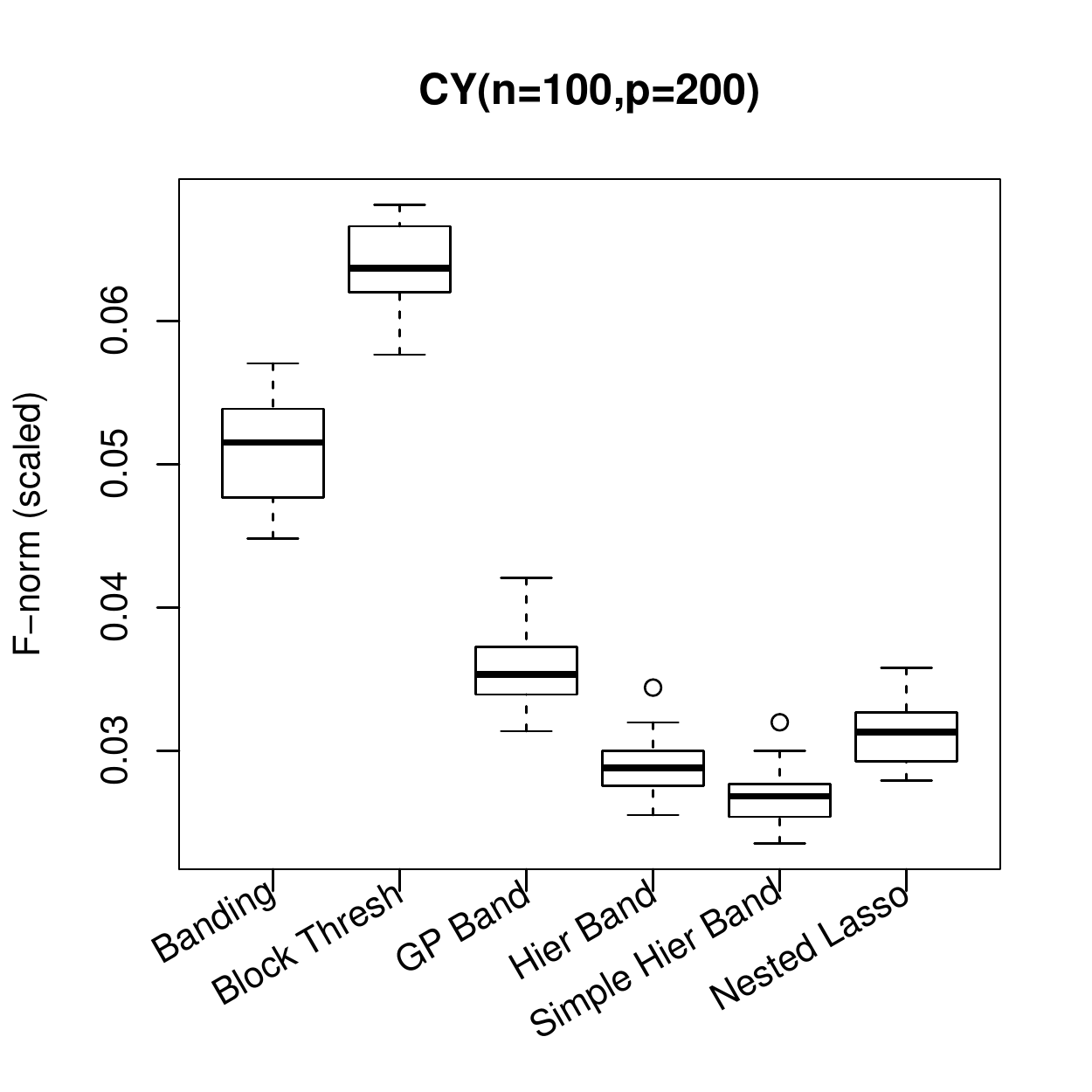}\\
  \includegraphics[width=0.325\linewidth]{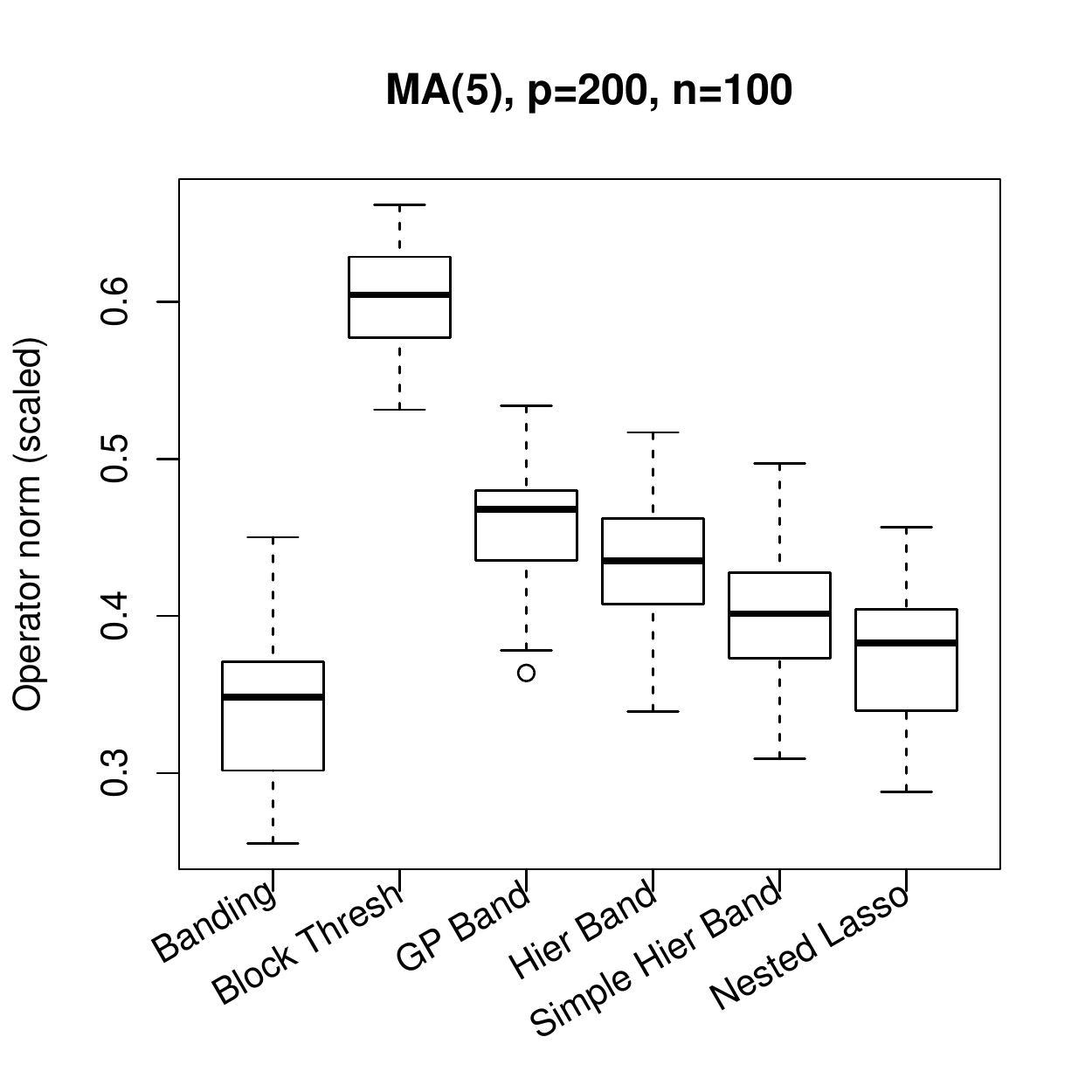}
  \includegraphics[width=0.325\linewidth]{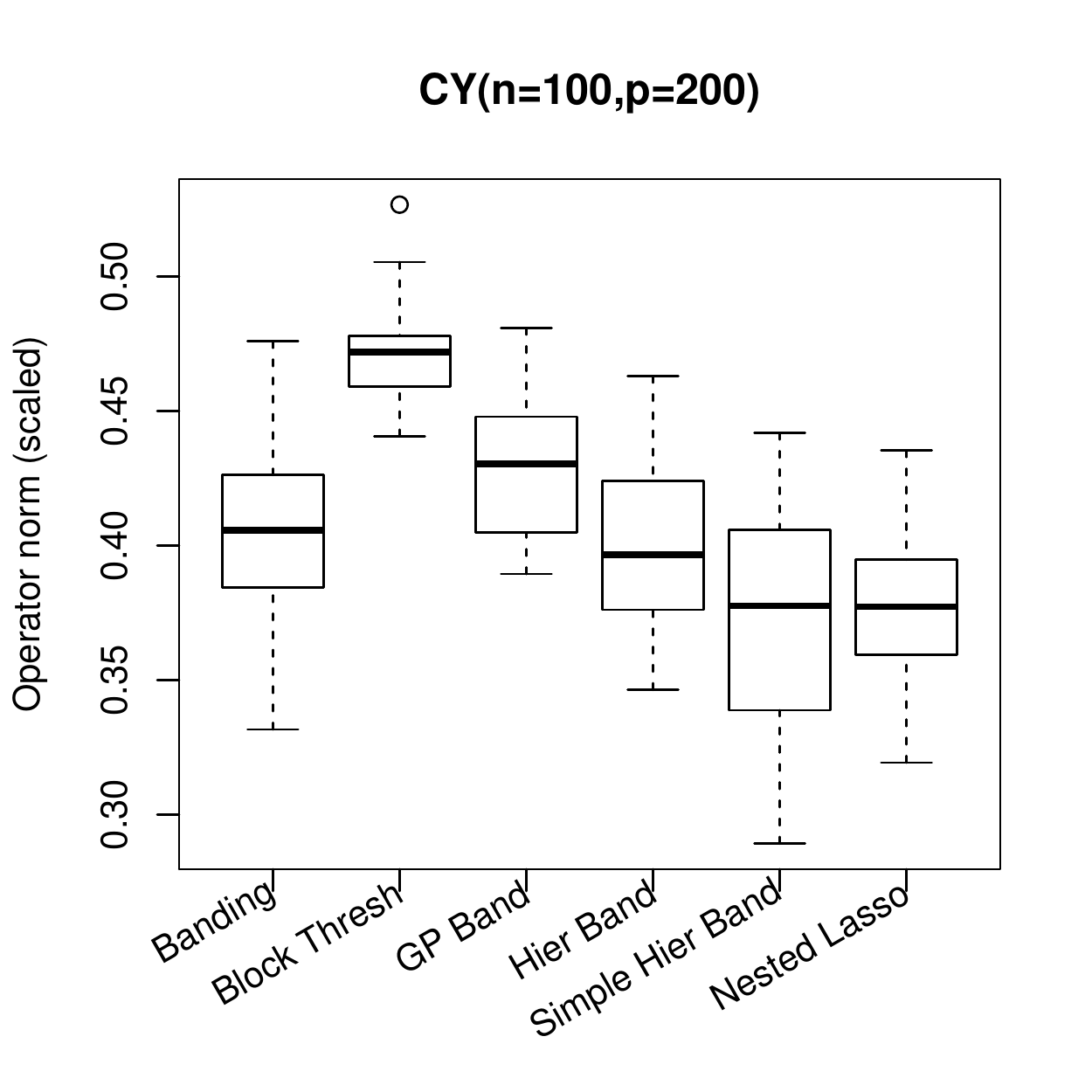}
  \caption{\em Comparison of methods in two settings in terms of
    (Top) Frobenius norm, relative to $\|\S-\true\|_F^2$ and
    (Bottom) operator norm relative to $\|S-\true\|_{op}$.}
\label{fig:compare}
\end{figure}
\begin{figure}
  \centering
  \includegraphics[width=0.45\linewidth]{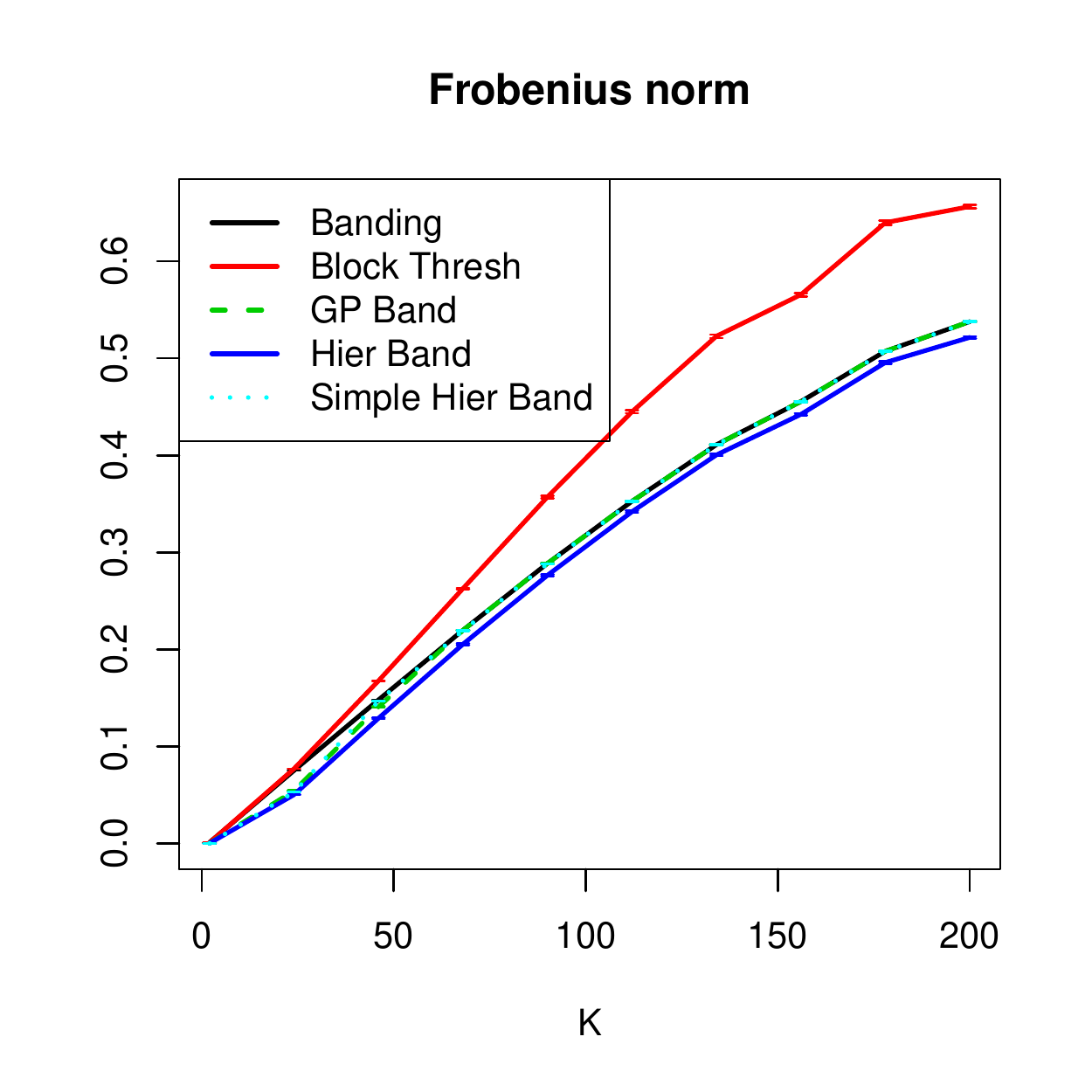}
  \includegraphics[width=0.45\linewidth]{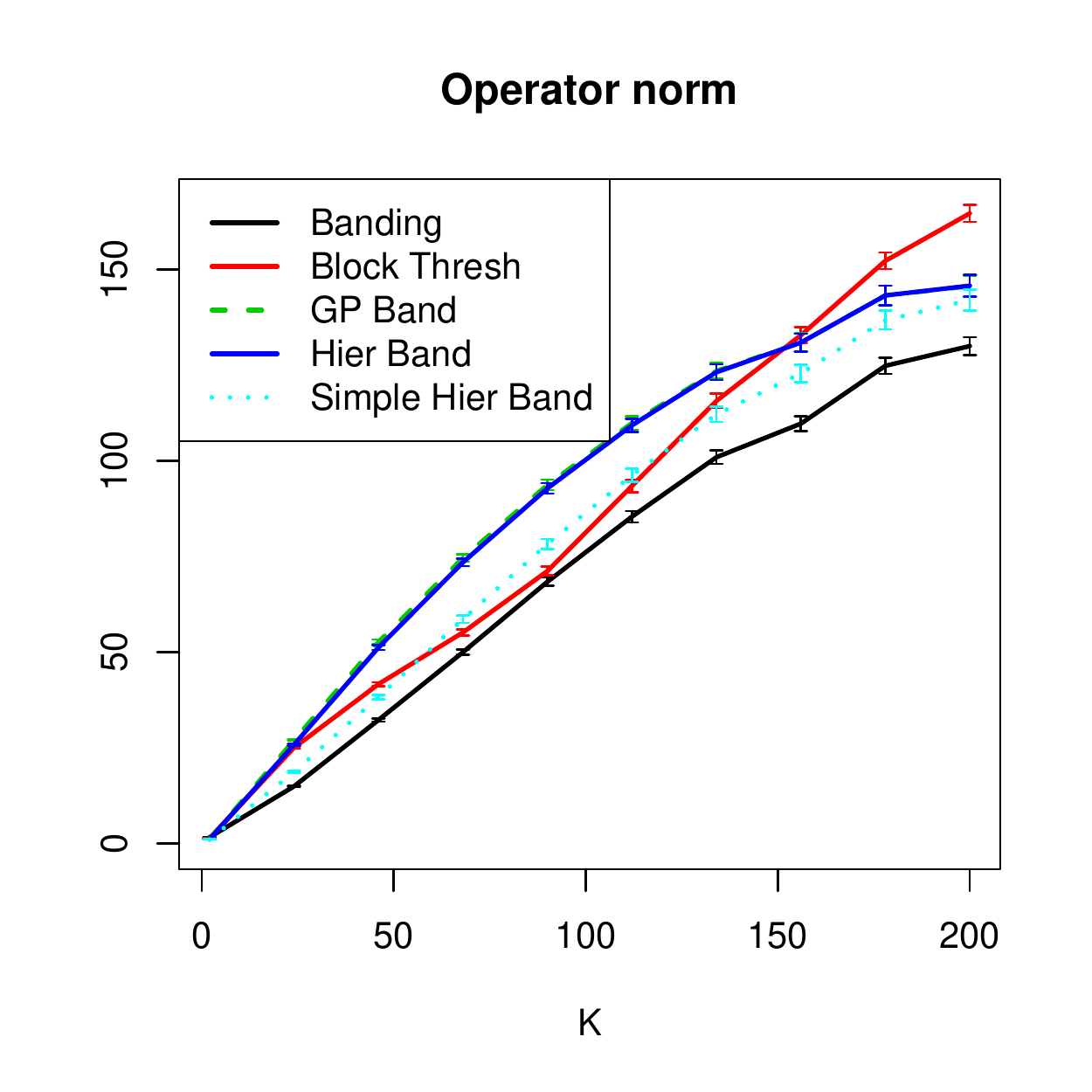}
  \caption{\em Comparison of methods in a setting where
    $\|\true\|_{op}$ is increasing. (Left) Frobenius norm
    and (Right) operator norm.}
\label{fig:divop}
\end{figure}

\subsection{An application to discriminant analysis of phoneme data}
\label{sec:data}
In this section, we develop an application of our banded estimator.  In
general, one would expect any procedure that
requires an estimate of the covariance matrix to benefit from
convex banding when the true covariance matrix is banded (or
approximately banded).  Consider the classification setting in which
we observe $n$ i.i.d. pairs, $(x_i,y_i)$, where $x_i\in\real^p$ is a vector
of predictors and $y_i$ labels the class of the $i$th point.  In quadratic discriminant analysis (QDA), one
assumes that $x_i|y_i=k\sim N_p(\mu^{*(k)},{\true}^{(k)})$.  The QDA
classification rule is to assign $x\in\real^p$ to the class $k$
maximizing $\hat{\mathbb P}(y=k|x)$.  Here $\hat{\mathbb P}$ denotes the estimate of
$\mathbb P(y=k|x)$ given by replacing the parameters $\mu^{*(k)}$,
${\true}^{(k)}$, and $\mathbb P(y_i=k)$ with their maximum likelihood
estimates.

To demonstrate the use of our covariance estimate in QDA, we consider
a binary classification problem described in \citet{ESL}.  The dataset
consists of short sound recordings of male voices saying one of two
similar sounding phonemes, and the goal is to build a classifier for
automatic labeling of the sounds.  The predictor vectors, $x$, are
log-periodograms, representing the (log) intensity of the recordings
across $p=256$ frequencies.  Because the predictors are naturally
ordered, one might expect a regularized estimate of the within-class
covariance matrix to be appropriate, and inspection of the sample
covariance matrices (Figure \ref{fig:phoneme-covs}) supports this.
\begin{figure}
  \centering
 \includegraphics[width=0.35\linewidth]{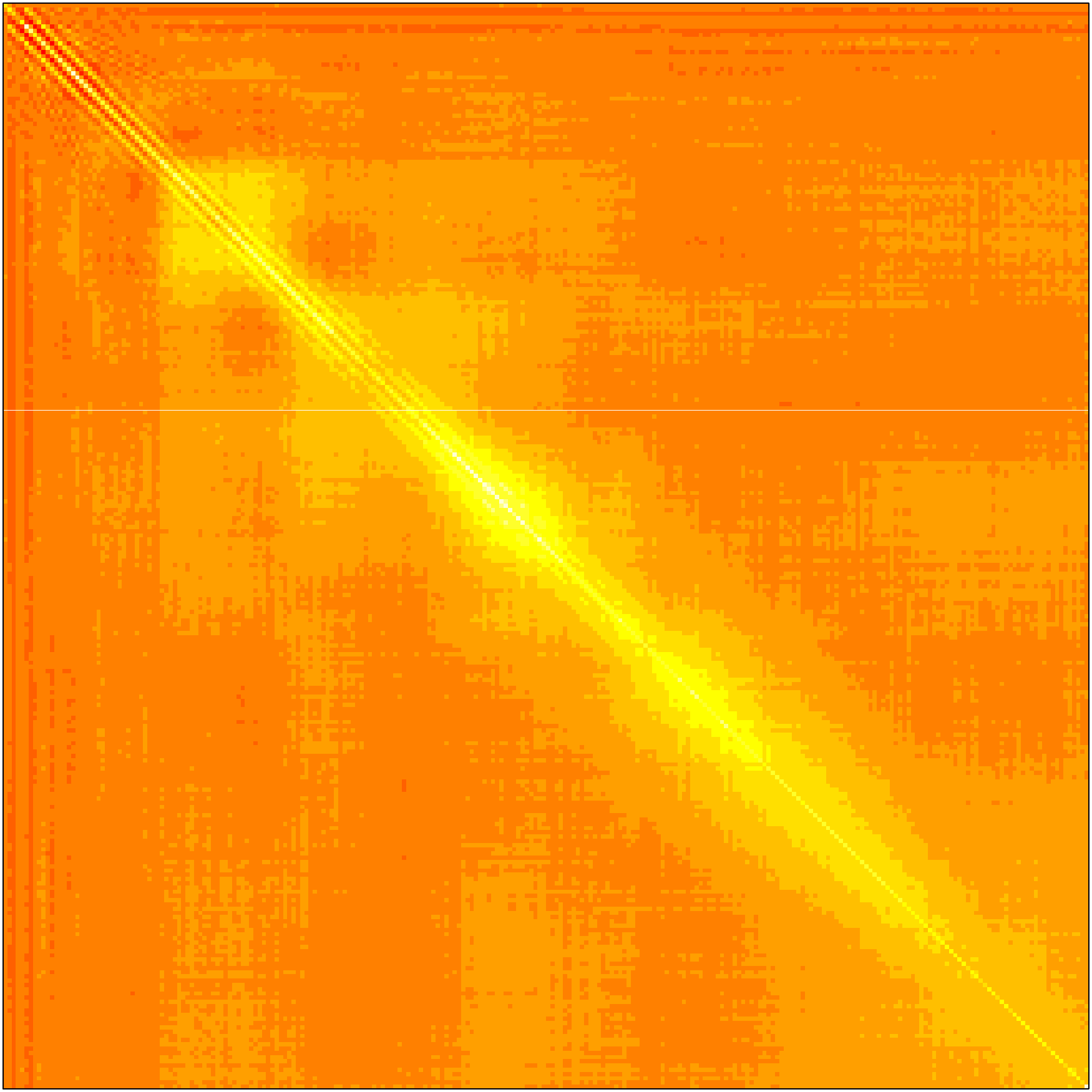}
 \includegraphics[width=0.35\linewidth]{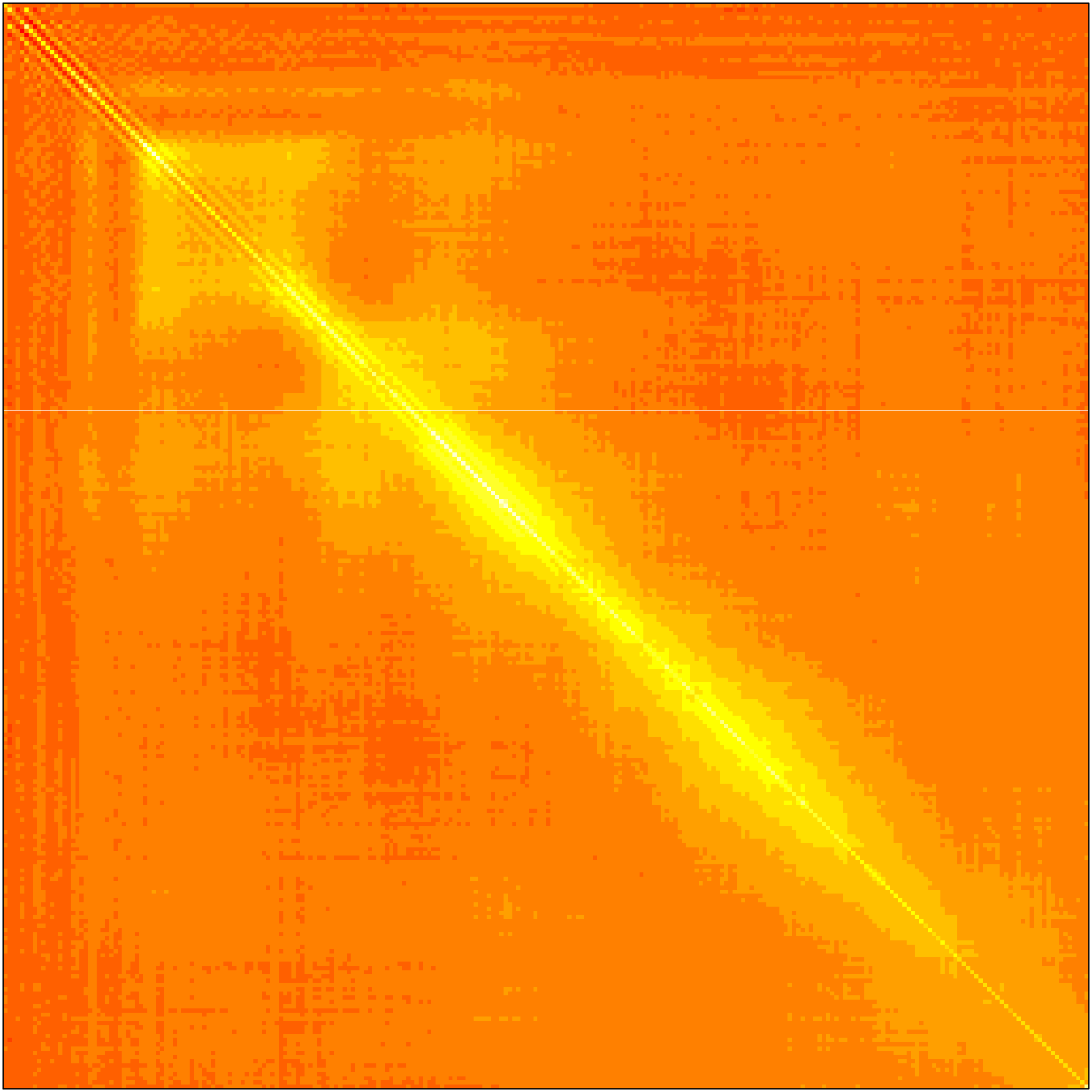}
  \caption{\em Within-class sample covariance matrices for phoneme data}
  \label{fig:phoneme-covs}
\end{figure}
There are $n_1=695$ and $n_2=1022$
phonemes in the two classes.  We randomly split the data into equal
sized training and test sets.  On the training set, five-fold cross
validation is performed to select tuning parameters
$(\lambda_1,\lambda_2)$ and then the convex banding estimates,
$\hatP^{(1)}_{\lambda_1}$ and $\hatP^{(2)}_{\lambda_2}$, are
used in place of the sample covariances to form the basis of the prediction rule,
$\hat{\mathbb P}(y=k|x)$.  The first two boxplots of Figure \ref{fig:phonemes}
show that QDA can be substantially improved by using the regularized
estimate.

Linear discriminant analysis (LDA) is like QDA but assumes a common
covariance matrix between classes, ${\true}^{(k)}=\true$, which is typically estimated as
$$
\S_w=\frac1{n_1+n_2-2}\left[(n_1-1)\S^{(1)}+(n_2-1)\S^{(2)}\right],
$$
where $\S^{(k)}=(n_k-1)^{-1}\sum_{i=1}^{n_k}(x_i-\bar x^{(k)})
(x_i-\bar x^{(k)})^T$.  Figure \ref{fig:phonemes}
shows that LDA does better than the QDA methods, as expected in this
regime of $p$ and $n$ by, e.g., \citet{Cheng04}, and can itself be improved by
a regularized version of $\S_w$, in which we replace
$\S^{(k)}$ in the above expression with $\hatP^{(k)}_{\lambda_k}$ (again, cross
validation is performed to select the tuning parameters).
\begin{figure}
  \centering
  \includegraphics[width=0.5\linewidth]{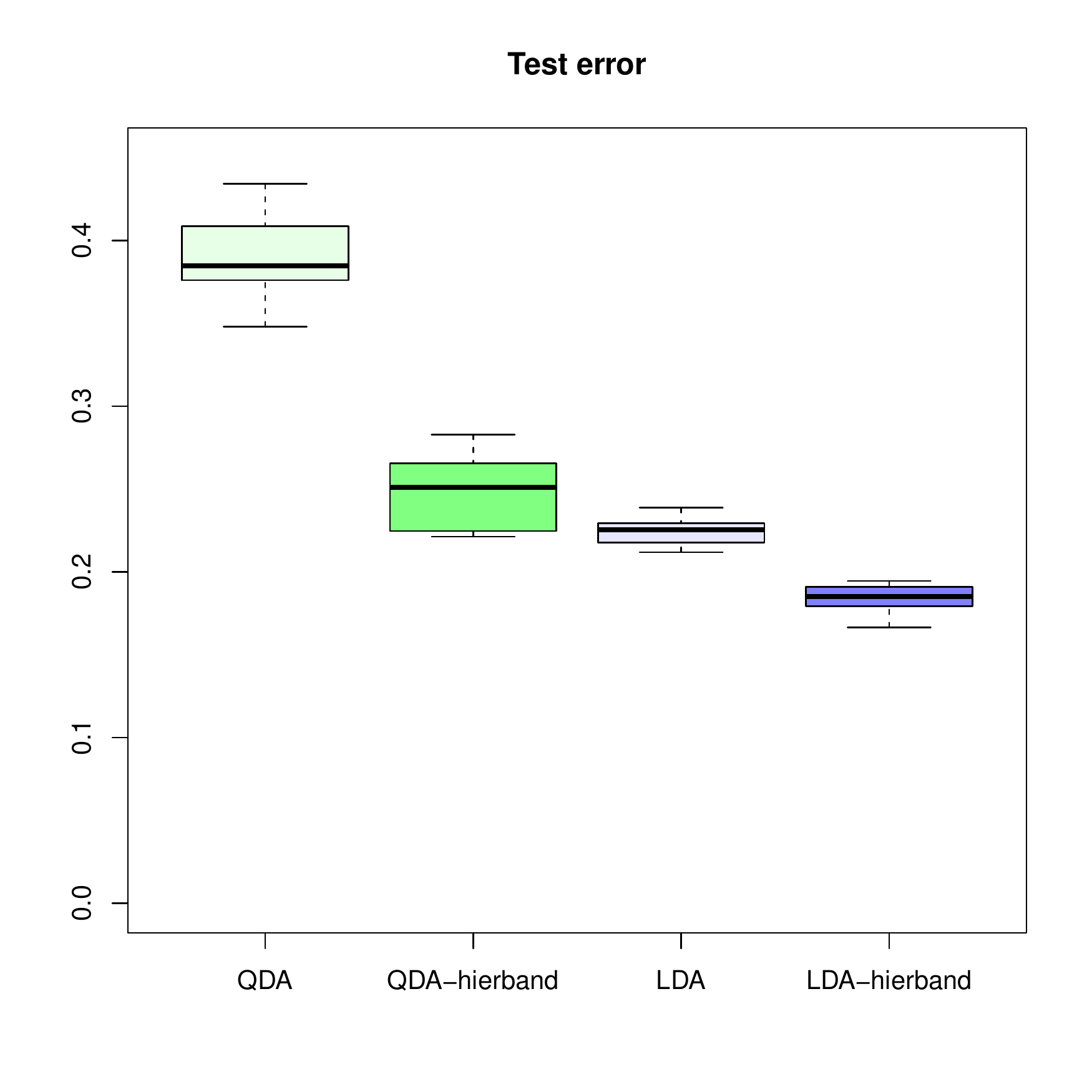}
  \caption{\em Test set misclassification error rates of discriminant analysis of phoneme data.  Convex banding of the within-class covariance matrices improves both QDA and LDA.}
  \label{fig:phonemes}
\end{figure}

\section{Conclusion}
\label{sec:conclusion}

We have introduced a new kind of banding estimator of the covariance matrix that has
strong practical and theoretical performance.  Formulating our
estimator through a convex optimization problem admits precise
theoretical results and efficient computational algorithms.  We prove
that our estimator is minimax rate adaptive  in Frobenius norm over the class of semi-banded matrices
introduced in Section \ref{sec:F-norm}  and minimax rate adaptive  in operator norm over 
the class of banded matrices, up to multiplicative logarithmic factors.  Both classes of matrices over which optimality is established are 
allowed to have bandwidth growing with $n$, $p$ or both. The construction of adaptive estimators over these classes, with established 
rate optimality is, to the best of our knowledge, new.  Moreover,  the proposed estimators recover  the true bandwidth with high probability, for truly banded matrices, 
under minimal signal strength conditions.    In contrast to existing
estimators, our proposed estimator, which enjoys all the  above theoretical properties,  is guaranteed to be both exactly banded,  and positive
definite, with high probability.  We also propose a version of the estimator, that is guaranteed to be positive definite in finite samples. 
Simulation studies show that the hierarchically banded estimator is a strong and fast competitor of  the existing estimators. 
Through a data example,  we demonstrate how our estimator can
be effectively used to improve classification error of both quadratic
and linear discriminant analysis.  Indeed, we expect that many common
statistical procedures that require a covariance matrix may be
improved by using our estimator when the assumption of semi-bandedness
is appropriate.  An {\tt R} \citep{R} package, named {\tt hierband}, will be
made available, implementing our estimator.

\section{Acknowledgments}
\label{sec:acknowledgments}

We thank Adam Rothman for providing {\tt R} code for the Nested Lasso method.

\appendix
\section{Dual problem}
\label{sec:proof-dual}
Define
$$
L(\Sigma,A^{(1)},\ldots, A^{(p-1)})=\half\|\S-\Sigma\|_F^2+\lambda\langle\sum_{\ell=1}^{p-1}
W_\ell*\Al,\Sigma\rangle.
$$
Observe that
$$
\|(W^{(\ell)}*\Sigma)_{g_\ell}\|_2=\max_{A^{(\ell)}}\left\{\langle W^{(\ell)}*A^{(\ell)},\Sigma\rangle\text{ s.t. }\|A^{(\ell)}_{g_\ell}\|_2\le1,~A^{(\ell)}_{g_\ell^c}=0\right\}.
$$
It follows that
\eqref{eq:primal} is equivalent to 
$$
\min_\Sigma\left\{\max_{\Al}\left[L(\Sigma,A^{(1)},\ldots, A^{(p-1)})\text{ s.t. }\|\Al_{\gl}\|_2\le 1,\Al_{\gl^c}=0\right]\right\}.
$$
We get the dual problem by interchanging the $\min$ and $\max$.  The
inner minimization gives the primal-dual relation given in the theorem
(by strong duality) 
and the following dual function:
$$
\min_\Sigma L(\Sigma,A^{(1)},\ldots, A^{(p-1)})=-\half
    \|\S-\lambda\sum_{\ell=1}^{p-1}
      W^{(\ell)}*\Al\|_F^2+\half\|\S\|_F^2.
$$

\section{Ellipsoid projection}
\label{sec:ellipsoid-projection}
To update $\hAl$ in Algorithm \ref{alg:psd}, we must solve a problem
of the form
$$
\hAl_{\gl}=\arg\min_{a\in\real^{|\gl|}}\|\hat R^{(\ell)}_{\gl}-\lambda
D^{(\ell)} a\|_2^2\text{ s.t. }\|a\|_2^2\le 1,
$$
which (in a change of coordinates) is the projection of a point onto
an ellipsoid.  Clearly, if $\|{D^{(\ell)}}^{-1}\hat R^{(\ell)}_{\gl}\|\le \lambda$, then $\hAl_{\gl}={D^{(\ell)}}^{-1}\hat
R^{(\ell)}_{\gl}/\lambda$.  Otherwise, we use the method of Lagrange
multipliers (to solve the problem with an equality constraint):
$$
\mathcal L(a,\nu)=\|\hat R^{(\ell)}_{\gl}-\lambda
D^{(\ell)} a\|_2^2+\nu \lambda^2 (\|a\|_2^2-1)
$$
whence
$$
0=2\lambda D^{(\ell)}(\lambda
D^{(\ell)}\hat a-\hat R^{(\ell)}_{g_\ell})+2\nu \lambda^2 a\implies \hat
a=(D^{(\ell)2}+\nu I)^{-1}\lambda^{-1} D^{(\ell)}\hat R^{(\ell)}_{\gl}.
$$
That is,
$$
\hat A_{s_m}^{(\ell)}=\frac{w_{\ell m}}{\lambda(w_{\ell m}^2+\hat\nu_\ell)}\hat
  R^{(\ell)}_{s_m}
$$
where $\hat\nu_\ell$ is such that $\hAl_\gl$ has unit norm, i.e. such
that $h_\ell(\hat\nu_\ell)=\lambda^2$ where $h_\ell$ is defined in \eqref{eq:nu}.
We compute this root numerically.  We can get limits within which
$\hat\nu_\ell$ must lie.  For example, replacing $w_{\ell m}$ by
$w_\ell=\max_m w_{\ell m}$ makes the RHS smaller whereas
replacing it by 0 makes the RHS larger:

$$
\frac{\sum_{m=1}^\ell w_{\ell m}^2 \|\hat R^{(\ell)}_{s_m}\|^2}{(w_{\ell
      }^2+\nu)^2} =:h_\ell^L(\nu)\le h_\ell(\nu) \le
    h_\ell^U(\nu):=\frac{\sum_{m=1}^\ell w_{\ell m}^2 \|\hat R^{(\ell)}_{s_m}\|^2}{(0+\nu)^2}.
$$
Now, since $h_\ell(\nu)$ is a decreasing function, we know that 
$\hat\nu_\ell^L\le\hat\nu_\ell\le\hat\nu_\ell^U$, where
$h_\ell^L(\hat\nu_\ell^L)=\lambda^2=h_\ell^U(\hat\nu_\ell^U)$.  From
this, it follows that
$$
\left[\sqrt{\sum_{m=1}^\ell w_{\ell m}^2 \|\hat
R^{(\ell)}_{s_m}\|^2}-\lambda w_\ell^2\right]_+\le \lambda\hat\nu_\ell\le \sqrt{\sum_{m=1}^\ell w_{\ell m}^2 \|\hat
R^{(\ell)}_{s_m}\|^2}.
$$
Noting that $\sum_{m=1}^\ell w_{\ell m}^2 \|\hat
R^{(\ell)}_{s_m}\|^2=\|D^{(\ell)}\hat R_\gl^{(\ell)}\|^2_2$, this
simplifies to
$$
\|D^{(\ell)}\hat R_\gl^{(\ell)}\|_2-\lambda w_\ell^2\le \lambda\hat\nu_\ell\le \|D^{(\ell)}\hat R_\gl^{(\ell)}\|_2.
$$
To summarize, we have for $1\le m\le\ell\le p-1$,
$$
\hat A_{s_m}^{(\ell)}=
\begin{cases}
  \hat R_{s_m}^{(\ell)}/(\lambda w_{\ell m}) & \text{ if
  }\|{D^{(\ell)}}^{-1}\hat R_{g_\ell}^{(\ell)}\|_2\le\lambda\\
\frac{w_{\ell m}}{\lambda(w_{\ell m}^2+\hat\nu_\ell)}\hat
  R^{(\ell)}_{s_m}&\text{ otherwise}
\end{cases}
$$
and $\hat A_{g_\ell^c}^{(\ell)}=0$.  This can be written more simply
as $\frac{w_{\ell m}}{\lambda(w_{\ell m}^2+[\hat\nu_\ell]_+)}\hat
  R^{(\ell)}_{s_m}$ if we note that $\|{D^{(\ell)}}^{-1}\hat
R_{g_\ell}^{(\ell)}\|_2\le\lambda$ is equivalent to
$h_\ell(0)\le\lambda^2$ and that in this case $\hat\nu_\ell\le0$, since $h_\ell$
is nonincreasing.  It turns out that often we will be able to get $\hat\nu_\ell$
in closed form.  First, $h_1(\nu)=w_{1}^2\|\hat
R^{(1)}_{s_1}\|^2/(w_1^2+\nu)^2$, so $\hat\nu_1=w_1(\|\hat
R^{(1)}_{s_1}\|/\lambda-w_1)$.  Furthermore, for $\ell\ge1$, if
$\hat\nu_\ell\le0$ then for all $m\le\ell$, we have $\hat
R^{(\ell+1)}_{s_m}=0$.  This means that
$h_{\ell+1}(\nu)=\frac{w_{\ell}^2}{(w_{\ell}^2+\nu)^2}\|\hat
R^{(\ell+1)}_{s_{\ell+1}}\|^2$ whence
$$
\hat\nu_{\ell+1}=w_{\ell+1}\left(\|\hat
R^{(\ell+1)}_{s_{\ell+1}}\|/\lambda-w_{\ell+1}\right)
$$

Thus, we only need to perform numerical root-finding when $\hat\nu_\ell$ when $\ell=p-1-\hat K$.

\section{Determining $\lambda_{\max}$}
\label{sec:lambda-max}

We solve our estimator along a decreasing sequence of values of
the tuning parameter:
$\lambda_{\max}=\lambda_1\ge\cdots\ge\lambda_r$.  To get
a full range of sparsity levels, we wish to choose $\lambda_{\max}$ to
be the smallest $\lambda$ for which $\hatP$ is diagonal.  In light
of the discussion at the end of Appendix \ref{sec:ellipsoid-projection}, $\hatP$ is
diagonal precisely when $\hat\nu_\ell\le0$ for $\ell=1,\ldots,p-1$,
which in this case is equivalent to 
$\|\hat R^{(\ell)}_{s_{\ell}}\|_2\le \lambda w_{\ell}$ for
$\ell=1,\ldots,p-1$.  Now, if we start with $\hAl=0$ for all $\ell$ (and recall that $\hat C=0$
throughout since we are solving for $\hatP$), we have $\hat
R^{(\ell)}_{s_{\ell}}=\S_{s_{\ell+1}}$ for all $\ell$.
Therefore, we take
$$
\lambda_{\max}=\max_{\ell}\{\|\S_{s_\ell}\|_2/w_\ell\}.
$$



\section{Positive definiteness}
\label{sec:pd-proofs}

\subsection{Proof of Theorem \ref{thm:psd}}
\begin{proof}
    Let $u$ be the eigenvector of $\hatP$ such that $u^T\hatP u=\lambda_{\min}(\hatP)$.  Then
  \begin{align*}
    \lambda_{\min}(\hatP)=u^T\true u - u^T(\true-\hatP)u\ge \lambda_{\min}(\true)-\|\hatP-\true\|_{op}.
  \end{align*}
Now, by Theorem \ref{thm:opnormexact2}, $\|\hatP-\true\|_{op}\le C'K\sqrt{\frac{\log p}{n}}$, so the assumption on
$\lambda_{\min}(\true)$ ensures that, whp, $\hatP\succeq\delta
I_p$.  Thus, the constraint in \eqref{eq:primal} may be dropped
without changing the solution, meaning $\hatP=\tildeP$.
\end{proof}


\subsection{Algorithm for $\tildeP$ and its derivation}
\label{sec:deriv-algor-tilde}

\begin{theorem}
A dual of \eqref{eq:tilde} is given by
  \begin{align}
  &\mathrm{Minimize}_{\Al,C\in\real^{p\times p}}\half
    \left\|\S-\lambda\sum_{\ell=1}^{p-1}
      W^{(\ell)}*\Al+\lambda C\right\|_F^2-\lambda\delta\cdot\tr(C) \label{eq:tilde-dual}\\
&\qquad\st C\succeq 0,~\|\Al_{\gl}\|_2\le 1,~\Al_{\gl^c}=0\mathrm{~for~}1\le\ell\le p-1.\nonumber
\end{align}
In particular, given a solution to the dual, $(\hat A^{(1)},\ldots,\hat
A^{(p-1)},\hat C)$, the solution to \eqref{eq:tilde} is given by
\begin{align}
  \tildeP=\S-\lambda\sum_{\ell=1}^{p-1}
  W^{(\ell)}*\hAl+\lambda\hat C.\label{eq:tilde-primal-dual}
\end{align}
\label{thm:tilde-dual}
\end{theorem}
\begin{proof}
Define
$$
L(\Sigma,A^{(1)},\ldots, A^{(p-1)},C)=\half\|\S-\Sigma\|_F^2+\lambda\langle\sum_{\ell=1}^{p-1}
W_\ell*\Al,\Sigma\rangle-\lambda\langle C,\Sigma-\delta I_p\rangle.
$$
Observe that
$$
1_\infty\{\Sigma\succeq\delta I_p\}=\max_{C\succeq0}-\langle \Sigma-\delta I_p,C\rangle
$$
and (as before) that
$$
\|(W^{(\ell)}*\Sigma)_{g_\ell}\|_2=\max_{A^{(\ell)}}\left\{\langle W^{(\ell)}*A^{(\ell)},\Sigma\rangle\text{ s.t. }\|A^{(\ell)}_{g_\ell}\|_2\le1,~A^{(\ell)}_{g_\ell^c}=0\right\}.
$$
It follows that
\eqref{eq:primal} is equivalent to 
$$
\min_\Sigma\left\{\max_{\Al,C}\left[L(\Sigma,A^{(1)},\ldots, A^{(p-1)},C)\text{ s.t. }\|\Al_{\gl}\|_2\le 1,\Al_{\gl^c}=0, C\succeq 0\right]\right\}.
$$
We get the dual problem by interchanging the $\min$ and $\max$.  The
inner minimization gives the primal-dual relation given in the theorem
and the following dual function:
$$
\min_\Sigma L(\Sigma,A^{(1)},\ldots, A^{(p-1)},C)=-\half
    \|\S-\lambda\sum_{\ell=1}^{p-1}
      W^{(\ell)}*\Al+\lambda C\|_F^2+\half\|\S\|_F^2+\lambda\delta\cdot\tr(C).
$$
\end{proof}

A BCD algorithm for solving \eqref{eq:tilde-dual} is given in Algorithm \ref{alg:psd}.
The update over $C$ involves projecting a matrix onto the positive
semidefinite cone.  The other details are similar to those explained
in Section \ref{sec:computation}.

\begin{algorithm}[t]
\caption{BCD on dual of Problem \eqref{eq:tilde}.}
\emph{Inputs:} $\S,~\delta,~\lambda$, and weights matrices,
$W^{(\ell)}$. Initialize $A^{(\ell)},C$.\\
Repeat until convergence:

\begin{itemize}
\item $\{\hat A^{(\ell)}\}\leftarrow${\tt
    threshold\_subdiagonals}$\left(\S+\lambda \hat C,\{\hat A^{(\ell)}\},\lambda,\{w_{\ell m}\}\right)$
\item Let $UDU^T=\lambda\sum_{\ell=1}^{p-1}
  W^{(\ell)}*\hAl-\S$ be the eigenvalue decomposition.  Then,
$$
\lambda\hat C\leftarrow U[D+\delta I_p]_+U^T
$$
where the positive part, $[\cdot]_+$, is applied to each diagonal
element.
\end{itemize}

{\bf Subroutine} {\tt threshold\_subdiagonals}$\left(R,\{\hat A^{(\ell)}\},\lambda,\{w_{\ell m}\}\right)$
For $\ell=1,\ldots, p-1$:

  \begin{itemize}
  \item Compute $\hat R^{(\ell)}\leftarrow
    R-\lambda\sum_{\ell=1}^{p-1} W^{(\ell)}*\hAl$
  \item For $m\le\ell$, set $\hat A_{s_m}^{(\ell)}\leftarrow \frac{w_{\ell m}}{\lambda(w_{\ell m}^2+\max\{\hat\nu_\ell,0\})}\hat
      R^{(\ell)}_{s_m}$ where $\hat\nu_\ell$ satisfies
      $\lambda^2=h_\ell(\hat\nu_\ell)$, as in \eqref{eq:nu}.
  \end{itemize}
Return $\{\hat A^{(\ell)}\}$. \label{alg:psd}
\end{algorithm}

\section{Bounds on $\max_{ij}\left|\S_{ij}-\true_{ij}\right|$ }
\label{sec:lem-whp}
\begin{rem}
\textnormal{Lemma 3 in \cite{Bickel08band} is proved under a Gaussianity assumption coupled with the assumption that $\|\true\|_{op}$ is bounded.
Whereas inspection of the proof shows that the latter is not needed, we  cannot  quote this result  directly for other types of design. The 
commonly employed assumptions  of sub-Gaussianity are placed on the entire vector ${\bf X} = (X_1, \ldots, X_p)^T$  and postulate that there exists $\tau > 0$ 
such that 
\[ \mathbb{P}\{ |v^T({\bf X} - \mathbb{E}({\bf X}))| > t \} \leq e^{-t^2/(2\tau)}, \ \mbox{for all} \ t > 0 \ \mbox{and} \  \|v\|_2 = 1,\]
see, for instance, \cite{Cai12breg}.  However, if such a $\tau$ exists, then $\|\true\|_{op} \leq \tau$. Lemma \ref{lem:whp} shows that a bound on  $\|\true\|_{op}$ can be avoided 
in the probability bounds regarding $\max_{ij}\left|\S_{ij}-\true_{ij}\right|$, and the distributional assumption can be weakened to marginals. }
\end{rem} 
\begin{theorem}
\label{thm:whp}Assume $\log p \leq \gamma n$ for some constant $\gamma>0$ and $\max_{ij}|\true_{ij}|\leq M$ for some constant $M$.
Let $D = \max_{i, j}\left |\S_{ij}-\true_{ij}\right|$. 
There exists some constant $c>0$ such that for sufficiently large $x>0$, 
\begin{equation}\label{lemma1}
\mathbb{P}\left(D> x\sqrt{\log (p\vee n)/n}\right) \leq\frac{c}{p\vee n}
\end{equation}
and
\begin{equation}
\mathbb{E} D^2\cdot 1_{\left\{D> x\sqrt{\log (p\vee n)/n}\right\}}\leq \frac{c}{n(n\vee p)}
\end{equation}
\end{theorem}

\begin{proof}
By Lemma \ref{lem:whp-0} below, there exist constants $c_i, i=1, 2, 3, 4$ such that, for any $0 <t <2M$,
\[
\mathbb{P}\left( D> t\right) \leq p^2 c_1 e^{-c_2nt^2} + p c_3 e^{-c_4 nt}.
\]
Hence,
\begin{align*}
\mathbb{P}\left(D> x\sqrt{\log (p\vee n)/n}\right) &\leq p^2 c_1 e^{-c_2 x^2 \log(p\vee n)} + pc_3 e^{-c_4 x\log (p\vee n) \sqrt{n/\log  (p\vee n)}}\\
&\leq c_1  (p\vee n)^{2-c_2x^2} + c_3  (p\vee n)^{1-c_4x\sqrt{n/\log  (p\vee n)}}.
\end{align*}
Next we derive that
\begin{align*}
\mathbb{E} &D^2\cdot 1_{\left\{D> x\sqrt{\log  (p\vee n)/n}\right\}}= \int_{x^2\log  (p\vee n)/n}^{\infty} \mathbb{P}(D^2\geq y) \mathrm{d}y\\
&\leq \int_{x^2\log  (p\vee n)/n}^{\infty} \left(p^2 c_1 e^{-c_2nt} + p c_3 e^{-c_4 n\sqrt{t}}\right)\mathrm{d}t\\
&=\frac{p^2 c_1e^{-c_2 nx^2\log  (p\vee n)/n}}{c_2 n} + \frac{pc_3e^{-c_4nx\sqrt{\log  (p\vee n)/n}}}{c_4n}\left(x\sqrt{\log  (p\vee n)/n} + \frac{1}{c_4n}\right)\\
&\leq \frac{c_1  (p\vee n)^{2-c_2x^2}}{c_2n} + \frac{c_3  (p\vee n)^{1-c_4x\sqrt{n/\log  (p\vee n)}}}{c_4n}\left(x\sqrt{\log  (p\vee n)/n} + \frac{1}{c_4n}\right).
\end{align*}
Hence if $x$ is sufficiently large and by the assumption that $\log p \leq \gamma n$, the  inequalities in the  lemma holds for some
constant $c$.
\end{proof}

\begin{lem}
\label{lem:whp-0}
There exist two constants $c_1$ and $c_2$ such that
\[
\mathbb{P}\left(\max_{ij}\left|\S_{ij}-\true_{ij}\right|> t\right)\leq 
2p^2  \exp\left(-\frac{c_2nt^2}{\max_j \true_{jj}}\right) + 8p \exp\left (- \frac{c_1 nt}{\max_j \true_{jj}}\right) .
\]
for any $0<t < 2 \max_j \true_{jj}$.
\end{lem}

\noindent{\it Proof}: Note that 
\[
\S_{ij} =  n^{-1}\sum_{k=1}^n X_{ki}X_{kj} - \bar X_i \bar X_j.
\]
and
\[
\left|\S_{ij}-\true_{ij}\right| \leq \left|n^{-1}\sum_{k=1}^n X_{ki}X_{kj}-\true_{ij}\right| + \left|\bar X_i \bar X_j\right|.
\]
Hence
\begin{align*}
&\mathbb{P}\left(\max_{ij}\left|\S_{ij}-\true_{ij}\right|> t\right)\\
\leq&\mathbb{P}\left(\max_{ij}\left|n^{-1}\sum_{k=1}^n X_{ki}X_{kj}-\true_{ij}\right|> t/2\right) + \mathbb{P}\left(\max_{ij}\left|\bar X_i \bar X_j\right| > t/2\right)\\
\leq &p^2 \max_{ij} \mathbb{P}\left(\left|n^{-1}\sum_{k=1}^n X_{ki}X_{kj}-\true_{ij}\right|> t/2\right) + 2p \max_{j} \mathbb{P}\left(\left|\bar X_j\right| > \sqrt{t/2}\right).
\end{align*}
Let $I_{ij} = \mathbb{P}\left(\left|n^{-1}\sum_{k=1}^n X_{ki}X_{kj}-\true_{ij}\right|> t/2\right)$ 
and $I_j =  \mathbb{P}\left(\left|\bar X_j\right| > \sqrt{t/2}\right)$. Then
\begin{equation}
\label{eq1}
\mathbb{P}\left(\max_{ij}\left|\S_{ij}-\true_{ij}\right|> t\right)\leq p^2 \max_{ij} I_{ij} + 2p\max_j I_j.
\end{equation}
We first consider $I_j$.  $\bar X_j$ is sub-Gaussian with variance $\true_{jj}/n$ and 
\begin{align*}
\mathbb{E}\exp\left(t\bar X_j/\sqrt{\true_{jj}/n}\right)= &\prod_{k=1}^n \mathbb{E} \exp\left(t X_{kj}/ \sqrt{n \true_{jj}}\right)\\
\leq & \left\{\exp(Ct^2/n)\right\}^n\\
=& \exp(Ct^2).
\end{align*}
By Lemma 5.5 in \cite{Vershynin11}, 
\[
\mathbb{P}\left\{|\bar X_j|/\sqrt{\true_{jj}/n} >t \right\} \leq \exp(1-t^2/K_1^2)
\]
for some constant $K_1$ that does not depend on $j$. It follows that 
\begin{align*}
I_j &=  \mathbb{P}\left(\left|\bar X_j\right| > \sqrt{t/2}\right)=\mathbb{P}\left\{|\bar X_j|/\sqrt{\true_{jj}/n} >\sqrt{tn/(2\true_{jj})} \right\} 
\leq \exp\left(1 - \frac{nt}{2K_1^2\true_{jj}} \right).
\end{align*}
Therefore,
\begin{equation}
\label{eq2}
I_j \leq 4 \exp\left (- \frac{c_1 nt}{\max_j \true_{jj}}\right)
\end{equation}
for some constant $c_1$.

Now we consider $I_{ij}$. We shall find  $\nu_{ij}$ and $c_{ij}$ such that 
\begin{equation}
\label{eq3-1}
\sum_{k=1}^n \mathbb{E} (X_{ki}^2 X_{kj}^2) \leq \nu_{ij}
\end{equation}
and 
\begin{equation}
\label{eq3-2}
\sum_{k=1}^n \mathbb{E} \left\{(X_{ki}X_{kj})^q_{+}\right\} \leq \frac{q!}{2}\cdot \nu_{ij} \cdot c_{ij}^{q-2}
\end{equation}
for all integers $q\geq 3$. Then by Theorem 2.10 and Corollary 2.11 in \cite{Boucheron13}, for any $t> 0$,
\begin{equation}
\label{eq3-3}
\mathbb{P}\left(\left|\sum_{k=1}^n \left(X_{ki}X_{kj}-\true_{ij}\right)\right|> t\right)\leq 2\exp\left\{-\frac{t^2}{2(\nu_{ij}+c_{ij}t)}\right\}.
\end{equation}

To find $\nu_{ij}$ and $c_{ij}$,  note that by Lemma 5.5 in \cite{Vershynin11}, $\mathbb{E}|X_{ij}/\sqrt{\sigma_{jj}}|^q \leq K_2^q q^{q/2}$ 
for all $q\geq 1$ and some constant $K_2$ that does not depend on $j$. Hence
\begin{align*}
\sum_{k=1}^n \mathbb{E} (X_{ki}^2 X_{kj}^2) \leq \sum_{k=1}^n \sqrt{\mathbb{E} X_{ki}^4\cdot \mathbb{E}X_{kj}^4}\leq 16 n\true_{ii}\true_{jj}K_2^4.
\end{align*}
Similarly,
\begin{align*}
\sum_{k=1}^n \mathbb{E} \left\{(X_{ki}X_{kj})^q_{+}\right\}\leq \sum_{k=1}^n \sqrt{\mathbb{E} X_{ki}^{2q}\cdot \mathbb{E}X_{kj}^{2q}}\leq n\Sigma^{*q/2}_{ii}\Sigma^{*q/2}_{jj} K_2^{2q}(2q)^{q}.
\end{align*}
It is easy to show that  (\ref{eq3-1}) and (\ref{eq3-2}) hold with $\nu_{ij} =K_3n\true_{ii}\true_{jj} $ and $c_{ij} = K_3\sqrt{\true_{ii}\true_{jj}}$ for some  constant $K_3$ that  is sufficiently large and does not depend on $i$ or $j$.

By (\ref{eq3-3}),  it follows that 
\begin{align*}
I_{ij} &\leq 2\exp\left\{-\frac{n^2 t^2}{4(2\nu_{ij} + c_{ij}nt)}\right\}\\
&= 2\exp\left\{-\frac{n^2 t^2}{4(2K_3 n\true_{ii}\true_{jj} +K_3\sqrt{\true_{ii}\true_{jj}} nt)}\right\}\\
&\leq 2\exp\left\{-\frac{n t^2}{4(2K_3\true_{ii}\true_{jj} +K_3\sqrt{\true_{ii}\true_{jj}} t)}\right\}.
\end{align*}
If $t < 2 \max_j \true_{jj}$, we have
\begin{equation}
\label{eq3}
I_{ij} \leq 2 \exp\left(-\frac{c_2nt^2}{\max_j \true_{jj}}\right),
\end{equation}
where $c_2 = (16K_3)^{-1}$.

\section{Proof of bandwidth recovery}
\label{sec:bandwidth-recovery-proof}
\begin{proof}[Proof of Theorem \ref{thm:recovery1}]
Referring to the proof of Theorem \ref{thm:tapering}, we have 
$
\hat R^{(\ell+1)}_{g_{\ell}} = 0
$
if $\hat\nu_\ell\le0$ or equivalently if $h_\ell(0)\le\lambda^2$,
where $h_\ell$ is defined in \eqref{eq:nu}.  If $L=0$, then $K=p-1$
and $\hat K\le K$ holds automatically.  Thus, assume that $L>0$.
We prove that $\hat R^{(L+1)}_{g_L}=0$ by induction on $\ell$. For $\ell = 1$, 
$$
h_\ell(0) = 2\S_{1p}^2/w_1^2 \leq
\max_{ij}|\S_{ij}-\true_{ij}|^2\leq \lambda^2
$$
on the set $\mathcal A_x$ defined in \eqref{eq:whp}.
Assume $\hat{R}_{g_{\ell}}^{(\ell+1)}=0$ for $\ell < L$. Then, since
$\hat R^{(\ell+1)}_{s_{\ell+1}}=\S_{s_{\ell+1}}$,
$$
h_{\ell+1}(0)=\| \S_{s_{\ell+1}}\|^2/w_{\ell+1,\ell+1}^2 \leq
\lambda^2
$$
on $\mathcal A_x$.  Therefore, $\hat R_{g_L}^{(L+1)} = 0$ and so $\hatP_{g_L}=0$, i.e.,
$\hat K\le K$.
\end{proof}

\begin{proof}[Proof of Theorems \ref{thm:recovery2} and
  \ref{thm:recovery3}]
In both theorems, we wish to show that $\hatP_{s_{L+1}}\neq0$ or equivalently that
$h_\ell(0)>\lambda^2$ for each $\ell\ge L+1$, whence we get the condition
\begin{align}
  \min_{\ell\geq L+1}h_\ell(0)>\lambda^2.\label{eq:condition}
\end{align}
Recalling that $h_\ell(0)=\sum_{m=1}^\ell\|\hat R^{(\ell)}_{s_m}\|^2/w_{\ell
  m}^2$, we have
$$h_\ell(0)\ge\|\hat
R^{(\ell)}_{s_\ell}\|^2/w_{\ell}^2=\|\S_{s_\ell}\|^2/w_{\ell}^2.
$$
Now, being on the set $\mathcal A_x$ implies that for any $\ell$,
\begin{align}
  \|\S_{s_\ell}\|_2 \geq \|\true_{s_\ell}\|_2 -\|\S_{s_\ell}-\true_{s_\ell}\|_2\geq \|\true_{s_\ell}\|_2 -
  \sqrt{2\ell}\lambda=\|\true_{s_\ell}\|_2 -
  \lambda w_\ell. \label{eq:bound}
\end{align}
We consider the two theorems separately:
\begin{enumerate}
\item (Theorem \ref{thm:recovery2}) By assumption, for $\ell\ge L+1$,
  \eqref{eq:bound} gives us $h_\ell(0)>\lambda^2$.  Thus,
  \eqref{eq:condition} is satisfied, proving the first theorem.  
\item (Theorem \ref{thm:recovery3}) By the same argument as above, we
  have $h_\ell(0)>\lambda^2$ for $\ell=L+1$ and for $\ell\ge L+3$.
  It remains to show that $h_{L+2}(0)>\lambda^2$.  Since
  $h_{L+1}(0)>\lambda^2$, we have that $\hat\nu_{L+1}>0$ and since
  $\hat\nu_L\le0$ (see appendix on ellipsoidal projection), $\hat\nu_{L+1}=w_{L+1}(\|\S_{s_{L+1}}\|/\lambda-w_{L+1})>0$.  Thus, 
$$
\hat
R^{(L+2)}_{s_{L+1}}=\frac{\hat\nu_{L+1}\S_{s_{L+1}}}{w_{L+1}^2+\hat\nu_{L+1}}=\left(\frac{\|\S_{s_{L+1}}\|_2-\lambda
w_{L+1}}{\|\S_{s_{L+1}}\|_2}\right)\S_{s_{L+1}}
$$
and
\begin{align*}
  h_{L+2}(0) &=\|\hat R^{(L+2)}_{s_{L+2}}\|_2^2/w_{L+2}^2+\|\hat
  R^{(L+2)}_{s_{L+1}}\|_2^2/w_{L+2,L+1}^2\\
  &=\|\S_{s_{L+2}}\|_2^2/w_{L+2}^2+(\|\S_{s_{L+1}}\|_2-\lambda
  w_{L+1})^2/w_{L+2,L+1}^2\\
&\ge (\|\true_{s_{L+2}}\|_2-\lambda w_{L+2})_+^2/w_{L+2}^2+(\|\true_{s_{L+1}}\|_2-2\lambda
  w_{L+1})^2/w_{L+2,L+1}^2\\
&= (\|\true_{s_{L+2}}\|_2-\lambda w_{L+2})_+^2/w_{L+2}^2+\lambda^2\gamma^2
  w_{L+1}^2/w_{L+2,L+1}^2\\
&\ge(\|\true_{s_{L+2}}\|_2-\lambda w_{L+2})_+^2/w_{L+2}^2+\lambda^2\gamma^2
\end{align*}
again applying \eqref{eq:bound}, and using that
$w_{L+1}=w_{L+1,L+1}\ge w_{L+2,L+1}$.  Now, for this to exceed
$\lambda$ we have the following: If $\gamma\ge1$, there is
no requirement on $\true_{s_{L+2}}$; if $0<\gamma<1$, then
$$
\|\true_{s_{L+2}}\|_2>\lambda w_{L+2}\left(1+\sqrt{1-\gamma^2}\right)
$$
This establishes that
$h_\ell(0)>\lambda^2$ for $\ell\ge L+1$, completing the proof of the
second theorem. 
\end{enumerate}
\end{proof}

\section{Deterministic upper bound in Frobenius norm}
\label{sec:Fnorm-deter}
Define
\begin{equation}\label{eq:norms}
\|\Sigma\|_{2,1} = \sum_{\ell=1}^{p-1}w_{\ell}\|\Sigma_{s_{\ell}}\|_2\,\,\text{ 
  and }\,\,\|\Sigma\|_{2,\infty}=\max_{1\le\ell\le p-1}w_\ell^{-1}\|\Sigma_{s_\ell}\|_2.
\end{equation}
The results of Section \ref{sec:F-norm} are consequences of the
following theorem.  Recall that for any $B\in\mathbb{R}^{p\times p}$,
we define $L(B)$ to be such that $B_{g_{L(B)}} = 0$ and
$B_{s_{L(B)}+1}\neq 0$,  and $S(B) = \{L(B)+1,\cdots, p-1\}$ with
$K(B) = |S(B)|$.  
 Note then that $K(B) =p-1-L(B)$. 
 \begin{theorem}
Suppose $\true\in \mathcal{S}_p$ and $\max_{i, j} |\true_{ij}| \leq
M$ for some constant $M$. If the weights are given by either
\eqref{eq:groupw} or \eqref{eq:genw} and $\lambda = x\sqrt{\log p/n}$
for some constant $x>0$.
   Then for any $B\in\mathbb{R}^{p\times p}$,
\begin{equation*}
\begin{split}
\| \hatP -\true \|_F^2  \leq &\|\S_{s_p} - \true_{s_p}\|_2^2 +\|\true - B\|_F^2 +4\lambda^2  K(B) w_0(L(B)) \\
&+ 2\left(\|\S - \true\|_{2,\infty}- \lambda\right)\cdot1_{\left\{\|\S - \true\|_{2,\infty}\geq \lambda\right\}} \cdot \sqrt{\sum_{\ell=1}^p w_{\ell}^2}\cdot \|\hatP - B \|_F,
\end{split}
\end{equation*}
where $w_0(\ell) =  \max_{\ell  + 1\leq m\leq p-1} \sum_{s=m}^{p-1} w_{s m}^2$.
\label{thm:deter}
\end{theorem}
\noindent Recalling that the subdiagonal $s_m$ is included in $g_\ell$
for $ m\le\ell\le
p-1$, we see that  $\sum_{\ell=m}^{p-1}w_{\ell m}^2$ measures the
amount of ``net weight'' applied to the subdiagonal $s_m$ and
$w_0(L(B))$ measures the largest amount of ``net weight'' applied to
any subdiagonal in $m\in S(B)$.
 
\subsection{Proof of Theorem \ref{thm:deter}}\label{sec:Fnorm-proof}
\noindent We begin by stating and proving two lemmas  and two propositions that will be instrumental in the proof of Theorem \ref{thm:deter}.  The first lemma  provides bounds on the inner product of two
matrices in terms of the newly introduced norms in \eqref{eq:norms}.  We directly bound the inner product $\langle A,B\rangle^-$ in which we leave out the contribution of the diagonals, $$
\langle A,B\rangle^-=\langle A,B\rangle -
\langle A_{s_p},B_{s_p}\rangle=\sum_{j\neq k}A_{jk}B_{jk}.
$$
We treat the main diagonal differently from the rest because it does
not appear in the penalty term $\|\Sigma\|_{2,1}^*$ of \eqref{eq:penalty}.

\begin{lem}
Let $A$ and $B$ be two arbitrary $p\times p$ matrices, then
\[
\langle A, B\rangle^-  \leq \|A\|_{2,1}\cdot \|B\|_{2,\infty} \leq \|A\|_{2,1}^* \cdot \|B\|_{2,\infty}.
\]
\label{lem:innerprod}
\end{lem}
\begin{proof}
\begin{align*}
\langle A, B\rangle^- &=  \sum_{\ell=1}^{p-1} \langle A_{s_{\ell}}, B_{s_{\ell}}\rangle\\
& \leq  \sum_{\ell=1}^{p-1} \|A_{s_{\ell}}\|_2 \cdot\|B_{s_{\ell}}\|_2\\
&\leq \|A\|_{2,1} \cdot \|B\|_{2,\infty}.
\end{align*}
The second inequality follows from the fact that $\|A\|_{2,1}\leq \|A\|_{2,1}^*$.
\end{proof}
\noindent For any matrix $\Sigma\in\real^{p\times p}$ and set
$S\subseteq\{1,\ldots,p-1\}$, let $\Sigma_S$ denote the $p\times p$ matrix
such that $[\Sigma_S]_{ij}=\Sigma_{ij}1\{|i-j|\in S\}$ and let $\Sigma_{S^c}=\Sigma-\Sigma_S$.
\begin{lem}\label{lem2}
Let $S =\{L+1,\dots, p-1\}$ for some $L$.
For any $p\times p$ matrix $\Sigma$, 
\begin{eqnarray*}
(i) \ \  \|\Sigma\|_{2,1}&=& \|\Sigma_S\|_{2,1} + \|\Sigma_{S^c}\|_{2,1},\\
(ii) \ \   \|\Sigma\|_{2,1}^{\ast}&\leq& \|\Sigma_S\|_{2,1}^{\ast} + \|\Sigma_{S^c}\|_{2,1}^{\ast},\\
(iii) \ \  \|\Sigma\|_{2,1}^{\ast}&\geq&\|\Sigma_S\|_{2,1}^{\ast} + \|\Sigma_{S^c}\|_{2,1}.
\end{eqnarray*}
\end{lem}
\begin{proof}
We have
\begin{align*}
\|\Sigma\|_{2,1} &= \sum_{\ell=1}^{p-1} w_{\ell}\|\Sigma_{s_{\ell}}\|_2\\
&= \sum_{\ell=1}^{p-1}w_{\ell}\left(\|1_{\{\ell \in S\}}\Sigma_{s_{\ell}}\|_2 +  \|1_{\{\ell \notin S\}}\Sigma_{s_{\ell}}\|_2\right)\\
&=\|\Sigma_S\|_{2,1} + \|\Sigma_{S^c}\|_{2,1}.
\end{align*}
Similarly,
\begin{align*}
\|\Sigma\|_{2,1}^{\ast} &= \sum_{\ell=1}^{p-1}\sqrt{\sum_{m=1}^{\ell} w_{\ell m}^2\|\Sigma_{s_m}\|_2^2}\\
&\leq \sum_{\ell=1}^{p-1}\left\{\sqrt{\sum_{m=1}^{\ell}w_{\ell m}^2 \|1_{\{m\in S\}}\Sigma_{s_{m}}\|_2^2} +\sqrt{\sum_{m=1}^{\ell}w_{\ell m}^2 \|1_{\{m \notin S\}}\Sigma_{s_{m}}\|_2^2}\right\}\\
&=\|\Sigma_S\|_{2,1}^{\ast} + \|\Sigma_{S^c}\|_{2,1}^{\ast}.
\end{align*}
Finally,
\begin{align*}
\|\Sigma\|_{2,1}^{\ast} &= \sum_{\ell=1}^{p-1}\sqrt{\sum_{m=1}^{\ell} w_{\ell m}^2\|\Sigma_{s_m}\|_2^2}\\
&\geq \sum_{\ell=1}^{p-1}\left\{\sqrt{\sum_{m=1}^{\ell}w_{\ell m}^2 \|1_{\{m\in S\}}\Sigma_{s_{m}}\|_2^2} +w_{\ell \ell} \|1_{\{\ell \notin S\}}\Sigma_{s_{\ell}}\|_2\right\}\\
&= \|\Sigma_S\|_{2,1}^{\ast} + \|\Sigma_{S^c}\|_{2,1}.
\end{align*}
\end{proof}

\noindent  Let $w_{\ell\cdot}\in\real^\ell$ denote the weights on the
  $\ell$th triangle and let  the weight matrix $W^{(\ell)}\in \mathbb{R}^{p\times p}$ be  defined as: $W^{(\ell)}_{s_m} = w_{\ell m}1_{2m}$ for $1\leq m\leq \ell$
  and $W^{(\ell)}_{s_m} = 0$ if $m>\ell$. Here $1_{2m}$ is a length-$2m$ vector of $1$'s.
 Observe that the penalty term \eqref{eq:penalty} can be equivalently written as 
$$
\|\Sigma\|_{2,1}^*=\sum_{\ell=1}^{p-1}\|(W^{(\ell)}*\Sigma)_{g_\ell}\|_2,
$$
where $*$ denotes elementwise multiplication.  \\
  
\noindent Define $f_{\ell}(B): = \|(W^{(\ell)}\ast B)_{g_{\ell}}\|_2$. Recall the definitions of the new norms in \eqref{eq:norms}. 
\begin{prop}
For any $B\in\mathbb{R}^{p\times p}$ and $W^{(\ell)}*A^{(\ell)}\in \partial f_{\ell}(B), 1\leq \ell\leq p-1$,
 \begin{equation*}
  \begin{split}
  \| \hatP -\true \|_F^2 \leq&  \| \true - B\|_F^2 -  \|\hatP - B\|_F^2  +2\langle \S_{s_p} - \true_{s_p}, \hatP_{s_p} - B_{s_p}\rangle \\
  & +  2\|\S - \true\|_{2,\infty} \cdot \|\hatP - B \|_{2,1}- 2\lambda
\left \langle \sum_{\ell=1}^{p-1} W^{(\ell)}\ast  A^{(\ell)}, \hatP - B\right\rangle.
\end{split}
 \end{equation*}
\label{prop:subgrad}
\end{prop}
\begin{proof}  $f_{\ell}(B)$ is convex and its sub-differential is
 \begin{equation}
 \begin{split}
 \partial f_{\ell} (B) =&\Bigl\{W^{(\ell)}*A^{(\ell)}\in\mathbb{R}^{p\times p}:
 \| A_{g_{\ell}}^{(\ell)}\|_2 \leq 1,  A_{g_{\ell}^c}^{(\ell)} = 0 \text{ and }\\
 \quad &\left\langle (W^{(\ell)}*B)_{\gl},A^{(\ell)}_{\gl}\right\rangle=\|
  (W^{(\ell)}* B)_{\gl}\|_2\cdot \|A^{(\ell)}_{\gl}\|_2 \Bigr\}.
  \end{split}
  \label{eq:subdiff}
  \end{equation}
Let $W^{(\ell)}*\hat A^{(\ell)} \in \partial f_{\ell}(\hatP)$. For an arbitrary $B\in\mathbb{R}^{p\times p}$, let $W^{(\ell)} *A^{(\ell)} \in \partial f_{\ell}(B)$.
Since the sub-gradient of a convex function is monotone,  we have
  \begin{align*}
  \langle W^{(\ell)}*\hat A^{(\ell)}, \hatP - B \rangle &= \left \langle (W^{(\ell)}*\hat A^{(\ell)})_{g_{\ell}}, (\hatP - B)_{g_{\ell}}\right \rangle\\
  &\geq  \left \langle (W^{(\ell)}* A^{(\ell)})_{g_{\ell}}, (\hatP - B)_{g_{\ell}}\right \rangle\\
  &=   \langle W^{(\ell)}* A^{(\ell)}, \hatP - B \rangle. 
  \end{align*}
  It follows that 
\begin{equation}
  \left \langle \sum_{\ell=1}^{p-1} W^{(\ell)}\ast \hat A^{(\ell)}, \hatP - B\right\rangle \geq \left \langle \sum_{\ell=1}^{p-1} W^{(\ell)}\ast  A^{(\ell)}, \hatP - B\right\rangle.
\label{eq:subgrad-monotone}
\end{equation}
Using the primal-dual relation \eqref{eq:primal-dual}, 
$$
\hatP -\S = - \lambda\left(\sum_{\ell=1}^{p-1} W^{(\ell)}\ast \hat A^{(\ell)}\right),
$$
and the fact that $ \hatP-\true =  \hatP- \S + \S - \true$, we have 
$$
 \langle \hatP-\true, \hatP - B\rangle = \langle \S - \true, \hatP - B\rangle - \lambda 
\left \langle \sum_{\ell=1}^{p-1} W^{(\ell)}\ast \hat A^{(\ell)}, \hatP - B\right\rangle.
$$
Combining this with \eqref{eq:subgrad-monotone}, we derive that
\begin{equation}
\label{eq2}
 \langle \hatP-\true, \hatP - B\rangle
\leq  \langle \S - \true, \hatP - B\rangle - \lambda 
\left \langle \sum_{\ell=1}^{p-1} W^{(\ell)}\ast  A^{(\ell)}, \hatP - B\right\rangle.
 \end{equation}
 By the cosine formula, 
 $$
2 \langle \hatP-\true, \hatP - B\rangle =  \| \hatP -\true \|_F^2 + \|\hatP - B\|_F^2 - \| \true - B\|_F^2.
 $$
 Therefore, we can rewrite~\eqref{eq2} as
$$
  \| \hatP -\true \|_F^2 + \|\hatP - B\|_F^2\leq  \| \true - B\|_F^2 + 2\langle \S - \true, \hatP - B\rangle - 2\lambda 
\left \langle \sum_{\ell=1}^{p-1} W^{(\ell)}\ast  A^{(\ell)}, \hatP - B\right\rangle
$$
and the proposition follows since, by Lemma \ref{lem:innerprod},  $$\langle \S - \true, \hatP - B\rangle \leq \langle \S_{s_p} - \true_{s_p}, \hatP_{s_p} - B_{s_p}\rangle +   \| \S - \true\|_{2,\infty}\cdot \|\hatP - B \|_{2,1}.$$
\end{proof}

\begin{prop}
For any $B\in\mathbb{R}^{p\times p}$,
 \begin{equation*}
  \begin{split}
  \| \hatP -\true \|_F^2 \leq&  \| \true - B\|_F^2 -  \|\hatP - B\|_F^2  +2\langle \S_{s_p} - \true_{s_p}, \hatP_{s_p} - B_{s_p}\rangle \\
  & +  2\left(\|\S - \true\|_{2,\infty}-\lambda\right) \cdot \|\hatP - B \|_{2,1} + 4\lambda \|\hatP_{S(B)}-B\|_{2,1}^*.
\end{split}
 \end{equation*}
\label{prop:deter}
\end{prop}
\begin{proof}
For simplicity, let $S = S(B)$ and $L=L(B)$. 
The focus of this proof is on the term 
\[
\left \langle \sum_{\ell=1}^{p-1} W^{(\ell)}\ast  A^{(\ell)}, \hatP - B\right\rangle\]
in Proposition~\ref{prop:subgrad}.
For $1\leq \ell \leq  L$, the constraints on $A^{(\ell)}$  are
$\|A^{(\ell)}_{g_{\ell}^c}\|_2 = 0$ and $\|A^{(\ell)}_{g_{\ell}}\|_2 \leq
1$ (the third constraint holds automatically since $B_{g_\ell}=0$ for
$\ell\le L$).  We  let $A_{s_m}^{(\ell)} = w_{\ell m}\hatP_{s_m}/f_{\ell}(\hatP)$
if $f_{\ell}(\hatP) \neq 0$ and $0$ otherwise, for $1\leq m\leq \ell,
1\leq \ell\leq L$. 
Then for $\ell \leq L$,
 \begin{align*}
\left \langle  W^{(\ell)}\ast  A^{(\ell)}, \hatP - B\right\rangle 
=&\sum_{m=1}^{\ell} \langle w_{\ell m} A^{(\ell)}_{s_m}, \hatP_{s_m} - B_{s_m}\rangle\\
=&\sum_{m=1}^{\ell} \langle w_{\ell m}^2 \hatP_{s_m}/f_{\ell}(\hatP), \hatP_{s_m}\rangle\\
=&\sum_{m=1}^{\ell} w_{\ell m}^2\|\hatP_{s_m}\|_2^2/f_{\ell}(\hatP)\\
= &f_{\ell}(\hatP).
 \end{align*}
It follows that
 $$
 \left \langle \sum_{\ell=1}^{L} W^{(\ell)}\ast  A^{(\ell)}, \hatP - B\right\rangle = \sum_{\ell=1}^{L} f_{\ell}(\hatP) \geq   \|\hatP_{S^c}\|_{2,1},
 $$
by Lemma \ref{lem2} (iii). \\

 \noindent Next,  fix $\ell\geq L + 1$. By the definition of
 subgradient in~\eqref{eq:subdiff},  $A^{(\ell)}_{g_{L}}$ can be
 chosen to have arbitrary values (as long as $\|A^{(\ell)}_{g_\ell}\|_2\le1$), and we take $A^{(\ell)}_{g_{L}}=0$ because of the equality
 $$
 \left\langle (W^{(\ell)}*B)_{\gl},A^{(\ell)}_{\gl}\right\rangle=\|
   (W^{(\ell)}* B)_{\gl}\|_2\cdot \|A^{(\ell)}_{\gl}\|_2.
 $$
 Then
 \begin{align*}
 -\left\langle  W^{(\ell)}\ast  A^{(\ell)}, \hatP - B\right\rangle
 =& -\sum_{m=L + 1}^{\ell} \langle w_{\ell m} A^{(\ell)}_{s_m}, \hatP_{s_m} - B_{s_m}\rangle\\
 \leq&\sum_{m=L + 1}^{\ell}w_{\ell m}\|A_{s_m}^{(\ell)}\|_2\cdot \| \hatP_{s_m} - B_{s_m}\|_2
 \\
  \leq&\sqrt{\sum_{m=L + 1}^{\ell}w_{\ell m}^2 \| \hatP_{s_m} - B_{s_m}\|_2^2 }\cdot\sqrt{\sum_{m=L + 1}^{\ell}\|A_{s_m}^{(\ell)}\|_2^2}\\
  \leq&\sqrt{\sum_{m=L + 1}^{\ell}w_{\ell m}^2 \| \hatP_{s_m} - B_{s_m}\|_2^2 }.
 \end{align*}
 In the above we used the fact that $A_{g_{\ell}^c}^{(\ell)}=0$ and $\|A^{(\ell)}_{g_{\ell}}\|_2\leq 1$.
 It follows that
 \begin{align*}
 -\left \langle \sum_{\ell=L+1}^{p-1} W^{(\ell)}\ast  A^{(\ell)}, \hatP - B\right\rangle \leq&\sum_{\ell=L +1}^{p-1}\sqrt{\sum_{m=L + 1}^{\ell}w_{\ell m}^2 \| \hatP_{s_m} - B_{s_m}\|_2^2 }\\
=& \|\hatP _{S} - B_{S} \|_{2,1}^{\ast}.
 \end{align*}
 Therefore
 \begin{align*}
 -\left \langle \sum_{\ell=1}^{p-1} W^{(\ell)}\ast  A^{(\ell)}, \hatP - B\right\rangle  \leq  - \|\hatP_{S^c}\|_{2,1}  + \|\hatP _{S} - B_{S} \|_{2,1}^{\ast},
 \end{align*}
 and, by Lemma \ref{lem2} (i), 
 \begin{align*}
 &\|\hatP - B \|_{2,1}-
\left \langle \sum_{\ell=1}^{p-1} W^{(\ell)}\ast  A^{(\ell)}, \hatP - B\right\rangle\\
 \leq& \|\hatP_{S} - B_{S}\|_{2,1} + \|\hatP_{S^c} - B_{S^c}\|_{2,1}- \|\hatP_{S^c}\|_{2,1}  + \|\hatP_{S} - B_{S} \|_{2,1}^{\ast}\\
 \leq & 2 \|\hatP_{S} - B_{S}\|_{2,1}^*.
\end{align*}
Here we have used that $B_{S^c}=0$.  The proposition follows by noting
that $B_S = B$.
\end{proof}

\noindent We are now ready to prove Theorem \ref{thm:deter}.
 \begin{proof}[Proof of Theorem \ref{thm:deter}]

 By Proposition~\ref{prop:deter}, 
 \begin{align}
 \label{eq5}
 \| \hatP -\true \|_F^2 &\leq   \| \true - B\|_F^2  - \|\hatP - B\|_F^2 + \ 2\langle \S_{s_p} - \true_{s_p}, \hatP_{s_p} - B_{s_p}\rangle\\
 &+2\left(\|\S - \true\|_{2,\infty}-\lambda\right) \cdot \|\hatP - B \|_{2,1} + 
   4\lambda\|\hatP_{S(B)} - B\|_{2,1}^*. \nonumber 
 \end{align}
  
 \noindent  First we have
  \begin{equation}
  \label{eq5-1}
  2\langle \S_{s_p} - \true_{s_p}, \hatP_{s_p} - B_{s_p}\rangle \leq \|\S_{s_p} - \true_{s_p}\|_2^2 + \|\hatP_{s_p} - B_{s_p}\|_2^2.
  \end{equation}
Next,
\begin{align}
\|\hatP-B\|_{2,1}&=\sum_{\ell=1}^p w_{\ell} \|\hatP_{s_{\ell}}-B_{s_{\ell}}\|_2\nonumber\\
& \leq \sqrt{\sum_{\ell=1}^p w_{\ell}^2} \cdot \sqrt{\sum_{\ell=1}^p \|\hatP_{s_{\ell}}-B_{s_{\ell}}\|_2^2}\nonumber\\
& \leq \sqrt{\sum_{\ell=1}^p w_{\ell}^2}\cdot \|\hatP - B\|_F.\label{eq5-2}
\end{align}
Finally, since  $2\lambda b\le a\lambda^2+b^2/a$, for any $a > 0$, 
we obtain 
\begin{align*}
2\lambda \|\hatP_{S(B)} - B\|_{2,1}^*&= 2\lambda \sum_{\ell=L(B)+1}^{p-1} \sqrt{\sum_{m=L(B)+1}^{\ell}w_{\ell m}^2 \|\hatP_{s_m}-B_{s_m}\|_2^2}\\
  &\le K(B)\lambda^2a + \sum_{\ell=L(B)+1}^{p-1} \sum_{m=L(B)+1}^{\ell } w_{\ell m}^2 \|\hatP_{s_m} - B_{s_m}\|_2^2/a\\
  &\le  K(B)\lambda^2 a + \sum_{m=L(B)+1}^{p-1} \left(\sum_{\ell=m}^{p-1} w_{\ell m}^2\right) \|\hatP_{s_m} - B_{s_m}\|_2^2/a.
\end{align*}
Letting $a = 2w_0(L(B))= 2\max_{L(B) + 1\leq m\leq p-1} \sum_{\ell=m}^{p-1} w_{\ell m}^2$,
\begin{align}
\label{eq5-3}
2\lambda \|\hatP_{S(B)} - B\|_{2,1}^*\leq 2K(B)\lambda^2w_0(L(B)) + \frac{1}{2}\|\hatP_{S(B)} - B\|_F^2 -\frac{1}{2} \|\hatP_{s_p} - B_{s_p}\|_2^2.
\end{align}
Then combining~(\ref{eq5}), (\ref{eq5-1}), (\ref{eq5-2}) and (\ref{eq5-3}), 
\begin{equation*}
\begin{split}
\| \hatP -\true \|_F^2  \leq &\|\true - B\|_F^2 + \|\S_{s_p} - \true_{s_p}\|_2^2 + 4K(B) \lambda^2 w_0(L(B)) \\
&+ 2\left(\|\S - \true\|_{2,\infty}- \lambda\right)\cdot 1_{\left\{\|\S - \true\|_{2,\infty}\geq \lambda\right\}}\cdot \sqrt{\sum_{\ell=1}^p w_{\ell}^2}\cdot \|\hatP - B \|_F,
\end{split}
\end{equation*}
which concludes  the proof.
\end{proof}

\section{Proof of convergence in Frobenius norm}
\label{sec:proof-theor-refthm}

We will use the following lemma.
\begin{lem}\label{lem:stat}
We have
$$
  \|\S-\true\|_{2,\infty}\leq  \max_{i, j} |\S_{ij}-\true_{ij}| \cdot \max_{1\leq \ell \leq p-1}\sqrt{2\ell}/w_{\ell}.
$$
\end{lem}

\noindent The proof follows immediately from the definition of the $\| \ \|_{2, \infty}$ norm given in \eqref{eq:norms}. 

 \noindent {\bf Proof of Theorem \ref{thm:oracle}}
 \begin{proof}
 The first oracle inequality follows immediately from Theorem
 \ref{thm:deter}, the choice of $\lambda$ and $\mathcal A_x$, and the
 fact that $w_0(L(B)) \leq w_0 (0) \leq 4 p$ for the given weights.  We now focus on the bound for $\mathbb{E}\|\hatP - \true\|_F^2$.
 By Theorem \ref{thm:deter}, 
 \begin{equation}
\label{oracle_eq1}
\| \hatP -\true \|_F^2  \leq  R_1+ 2R_2 ( \|\hatP - \true \|_F + \|\true- B\|_F),
 \end{equation}
 where 
 \[
 R_1 = \|\true-B\|_F^2 +  \|\S_{s_p} - \true_{s_p}\|_2^2 + 4\lambda^2 K(B)  w_0(L(B))
 \]
 and 
 \[
 R_2 = \left(\|\S - \true\|_{2,\infty}- \lambda\right)\cdot 1_{\left\{\|\S - \true\|_{2,\infty}\geq \lambda\right\}}\cdot \sqrt{\sum_{\ell=1}^p w_{\ell}^2}.
 \]
Using that $2R_2\|\hatP-\true\|_F\le
2R_2^2+\half\|\hatP-\true\|_F^2$ and that $2R_2\|\true-B\|_F\le
R_2^2+\|\true-B\|_F^2$, it follows from \eqref{oracle_eq1} that
 \begin{equation}
 \label{oracle_eq2}
  \|\hatP -\true\|_F^2 \leq 6 R_2^2 + 2 \|\true-B\|_F^2 + 2R_1.
 \end{equation}
 With the given weights, $\sqrt{\sum_{\ell=1}^p w_{\ell}^2}\lesssim p$, hence
  \[
 R_2 \lesssim p \left(\|\S - \true\|_{2,\infty}- \lambda\right)\cdot 1_{\left\{\|\S - \true\|_{2,\infty}\geq \lambda\right\}}. \]
By Lemma \ref{lem:stat} and with the given weights, we obtain that 
 \[
 \|\S - \true\|_{2,\infty} \leq \max_{i,j} |\S_{ij} - \true_{ij}|.
 \]
 Let $D = \max_{i,j} |\S_{ij}- \true_{ij}|$. Then by Theorem
 \ref{thm:whp} in Section \ref{sec:lem-whp} and the given $\lambda$,
 \begin{align*}
 \mathbb{E}R_2^2 &\lesssim p^2 \mathbb{E}\left[ (D-\lambda)^2 \cdot 1_{\left\{D> \lambda\right\}}\right] \\
 & \lesssim p^2 \mathbb{E}\left[T^2\cdot 1_{\{D>\lambda\}}\right]\\
 &\lesssim p/n.
 \end{align*}
Also it is easy to show that
 \[
  \mathbb{E} \|\S_{s_p} - \true_{s_p}\|_2^2 = \sum_{j=1}^p \mathbb{E}(\S_{ij}-\true_{ij})^2 \lesssim \frac{p}{n}.
 \]
 It follows by (\ref{oracle_eq2}) that 
 \begin{align*}
 \mathbb{E} \|\hatP -\true\|_F^2 &\leq 6 \mathbb{E}R_2^2 + 2 \|\true-B\|_F^2 + 2\mathbb{E}R_1\\
 &\lesssim  \|\true-B\|_F^2 +\frac{p}{n}+ \lambda^2 K(B)  w_0(L(B)).
 \end{align*}
 Recalling that $w_0(L(B))\lesssim p$ for the given weights, the theorem now follows.
 \end{proof}

 \section{Proof of Frobenius norm lower bound}
\label{sec:proof-theor-Flower}
 \noindent {\bf Proof of Theorem \ref{thm:Flower}} 
 \begin{proof}
Fix $0 < \alpha < 1/2$.  Let $B_{k,\ell} = e_ke_{\ell}^T + e_{\ell}e_{k}^T$ where $e_k$ is the unit vector in 
$\mathbb{R}^p$ with the $k$th entry being 1 and $e_{\ell}$ similarly defined. Let $\Omega$ be the subset of
$\{0,1\}^{p(p-1)/2}$ such that if $\epsilon = (\epsilon_{12},
\epsilon_{13},\ldots, \epsilon_{1p},\epsilon_{23},\ldots, \epsilon_{2p},\ldots, \epsilon_{p-1,p})\in \Omega$, 
then $\epsilon_{k,\ell}=0$ whenever $|k-\ell|> K$. Denote by $N$ the number of entries in $\epsilon$ that are not fixed at 0, then $N= 2pK + o(pK)$.
By Varshamov-Gilbert's bound (see Lemma 2.9 in \citealt{Tsybakov09}), there exists a subset $\Omega_0$ of $\Omega$ such that: (i) $0\in \Omega_0$; (ii) $\text{Card}(\Omega_0)\geq 2^{N/8} + 1$; (iii) for any two distinct $\epsilon$ and $\epsilon^{\prime}$ in $\Omega_0$, the Hamming distance $\sum_{k,\ell} |\epsilon_{k,\ell}-\epsilon_{k,\ell}^{\prime}|\geq N/8$. 

Now for $\epsilon \in \Omega_0$,
define $\Sigma_{\epsilon} = I_p + \frac{\alpha}{\sqrt{n}} \sum_{k<\ell} \epsilon_{k,\ell}B_{k,\ell}$. Note that 
$\Sigma_{\epsilon}$ has bandwidth at most $K$. For any two distinct $\epsilon$ and $\epsilon^{\prime}$ in $\Omega_0$,
\begin{equation}
\label{cond2}
\|\Sigma_{\epsilon} - \Sigma_{\epsilon^{\prime}}\|_F^2 =\frac{2\alpha^2}{n}\sum_{k<\ell} |\epsilon_{k,\ell}-\epsilon_{k,\ell}^{\prime}|\geq \alpha^2N/(4n).
\end{equation}
It's easy to see that $\tr(\Sigma_{\epsilon}) = p$. Note that 
\[
\|\Sigma_{\epsilon}-I_p\|_{op} \leq \frac{\alpha}{\sqrt{n}} 2K <1.
\]
 Hence $\Sigma_{\epsilon}$ is positive definite.

With slight abuse of notation, let $\mathbb{P}_{\Sigma}$ denote the joint probability distribution of 
$\mathbf{X}_1,\ldots,\mathbf{X}_n$ and each $\mathbf{X}_i$ is from a multivariate normal distribution with mean zero and covariance $\Sigma$.
Let $\mathbb{K}(\mathbb{P}_{\Sigma_{\epsilon}}, \mathbb{P}_{I_p}) =
 \int \log\left(\frac{d\mathbb{P}_{\Sigma_{\epsilon}}}{d\mathbb{P}_{I_p}}\right) d\mathbb{P}_{\Sigma_{\epsilon}}$ be the Kullback-Leibler
 divergence. Then we can verify that
 \begin{align*}
 \mathbb{K}(\mathbb{P}_{\Sigma_{\epsilon}}, \mathbb{P}_{I_p}) &=n\left\{-\frac{p}{2} + \frac{1}{2}\tr(\Sigma_{\epsilon}) - \frac{1}{2}\log \text{det}(\Sigma_{\epsilon})\right\}\\
 &=-\frac{n}{2}\log\text{det}(\Sigma_{\epsilon})\\
 &=-\frac{n}{2}\sum_{k=1}^p \log\left\{1 + \lambda_k (\tilde\Sigma_{\epsilon})\right\},
 \end{align*}
 where $\tilde\Sigma_{\epsilon} = \Sigma_{\epsilon}-I_p$. By the fact that $\log (1 + x) \geq x - x^2/2$ for any $x\geq 0$ and that
 $\sum_{k=1}^p \lambda_k (\tilde\Sigma_{\epsilon}) = 0$, we obtain that 
 \begin{align*}
 \mathbb{K}(\mathbb{P}_{\Sigma_{\epsilon}}, \mathbb{P}_{I_p}) &\leq \frac{n}{4}\sum_{k=1}^p \lambda_k^2 (\tilde\Sigma_{\epsilon})\\
 &= \frac{n}{4}\|\tilde\Sigma_{\epsilon}\|_F^2\\
 &\leq \alpha^2 pK.
 \end{align*}
 Therefore,
 \[
 \frac{1}{\text{Card}(\Omega_0)}\sum_{\Sigma_{\epsilon}\in \Omega_0}  \mathbb{K}(\mathbb{P}_{\Sigma_{\epsilon}}, \mathbb{P}_{I_p}) 
 \leq \alpha^2 pK.
 \]
 Since $\log (\text{Card}(\Omega_0)-1) \geq \log(2) N/8$ and $N= 2pK +o(pK)$, for any $0< a <1/8$, we can choose $\alpha$ small enough (depends only
 on $a$) such that 
 \begin{equation}
 \label{cond3}
  \frac{1}{\text{Card}(\Omega_0)}\sum_{\Sigma_{\epsilon}\in \Omega_0} \mathbb{K}(\mathbb{P}_{\Sigma_{\epsilon}}, \mathbb{P}_{I_p}) 
 \leq a \log (\text{Card}(\Omega_0)-1).
 \end{equation}
 With (\ref{cond2}) and (\ref{cond3}), by Theorem 2.5 in \cite{Tsybakov09}, the theorem holds.
 \end{proof}

\section{Proof of convergence in operator norm}
\label{sec:proof-onorm}
\noindent {\bf Proof of Theorem \ref{thm:opnormexact2}}
\begin{proof}
The arguments given here hold on the set  $\mathcal{A}_x$ defined in \eqref{eq:whp}, with $x$ as in Theorem \ref{thm:oracle}. Since, under our assumptions, we have $\hat{K} = K$, with high probability, we further have: 
\begin{align*}
\|\hatP - \true\|_{op}
\leq &\|\hatP_S - \S_{S}\|_{op} + \|\S_{S}-\true_{S}\|_{op} + \|\hatP_{S^c}-\true_{S^c}\|_{op}\\
= &\|\hatP_S - \S_{S}\|_{op} + \|\S_{S}-\true_{S}\|_{op} +\|\true_{S^c}\|_{op}  \\
\lesssim &\|\hatP_S - \S_{S}\|_{1,1} + \|\S_{S}-\true_{S}\|_{1,1} + K\sqrt{\log p/n} \\
\lesssim &\max_i \sum_{|i-j|\leq K} |\hatP_{ij}-\S_{ij}| + K\sqrt{\log p/n}.
\end{align*}
We claim that:
{\it
 there exists a constant $c>0$ such that
\begin{equation}
\label{claim}
|\hatP_{ij}-\S_{ij}|\leq c \lambda \ \ \mbox{for all } |i-j|\leq K.
\end{equation}
}
Then we have 
$
\|\hatP-\true\|_{op} \lesssim K\sqrt{\log p/n}
$
and the proof is complete.
 
Next, we prove claim (\ref{claim}). By (\ref{recur}), we have for $\ell\geq L$ and $m\leq \ell$,
\begin{align*}
\hat R_{s_{m}}^{(\ell+1)} &= \frac{\hat \nu_{\ell}}{w_{\ell, m}^2 + \hat \nu_{\ell}} \hat R_{s_m}^{(\ell)}=\hat R_{s_m}^{(\ell)} - \frac{w_{\ell, m}^2}{w_{\ell, m}^2 + \hat \nu_{\ell}} \hat R_{s_m}^{(\ell)},
\end{align*}
where $\hat \nu_{\ell}$ satisfies
\begin{equation}
\label{nu_1}
\sum_{m=1}^{\ell} \frac{w_{\ell, m}^2}{(w_{\ell, m}^2 + \hat \nu_{\ell})^2} \|\hat R_{s_m}^{(\ell)}\|_2^2  = \lambda^2.
\end{equation}
Let $h_{\ell, m} =  \frac{w_{\ell, m}^2}{w_{\ell, m}^2 + \hat \nu_{\ell}}$, then we have 
$$
\hat R_{s_{m}}^{(\ell+1)} = (1-h_{\ell, m})\hat R_{s_{m}}^{(\ell)},
$$
and hence for $m\geq L+1$,
$$
\hatP_{s_m} = \S_{s_m}\prod_{\ell= m}^{p-1} (1-h_{\ell, m}).
$$
Let $h_m = \prod_{\ell= m}^{p-1} (1-h_{\ell, m})$, then $h_m < 1$ and

$$
h_m = 1 -\sum_{\ell=m}^{p-1}h_{\ell,m} + o\left\{\left(\sum_{\ell=m}^{p-1}h_{\ell,m}\right)^2\right\}.
$$
Note that if  we establish that
\begin{equation}
\label{hm}
\sum_{\ell=m}^{p-1}h_{\ell,m}\leq C \lambda,
\end{equation}
for some constant $C>0$, 
then, for each $ L+ 1 \leq m \leq p$ and each $(i,j) \in S_m$, we have 
\begin{eqnarray}
|\hatP_{ij} -\true_{ij}| &=& |h_m \S_{ij} -h_m\true_{ij} + h_m\true_{ij} - \true_{ij}| \nonumber \\
& \leq & c\lambda, \nonumber
\end{eqnarray}

\noindent for some sufficiently large $c$ that does not depend on $i$ or $j$. Therefore to prove (\ref{claim}), it suffices
to prove (\ref{hm}).

Now we focus on $\hat\nu_{\ell}$.
By~(\ref{nu_1}), if $\ell \geq L+1$,
$$
\sum_{m=L+1}^{\ell} \frac{w_{\ell, m}^2}{(w_{\ell, \ell}^2 + \hat \nu_{\ell})^2} \|\hat R_{s_m}^{(\ell)}\|_2^2\leq \lambda^2,
$$
which leads to
$$
w_{\ell,\ell}^2 + \hat\nu_{\ell} \geq \frac{\sqrt{\sum_{m=L+1}^{\ell} w_{\ell,m}^2\|\hat R_{s_m}^{(\ell)}\|_2^2}}{\lambda}\geq \frac{w_{\ell\ell} \|\hat R_{s_{\ell}}^{(\ell)}\|_2}{\lambda},
$$
since $\max_{m}w_{\ell m} = w_{\ell\ell}=: w_{\ell}$. 
Note that, by (\ref{recur}), for every $m \geq L+1$, we have  $\|\hat R_{s_m}^{(\ell)}\|_2 =\|\S_{s_m}\|_2 \prod_{\ell^{\prime} = m}^{\ell-1}(1-h_{\ell^{\prime},m})$ and $\|\hat R_{s_{\ell}}^{(\ell)}\|_2 = \|\S_{s_{\ell}}\|_2$.  Then
\begin{equation}
\label{nu_3}
\hat\nu_{\ell} \geq \frac{w_{\ell} \|\S_{s_{\ell}}\|_2}{\lambda} - w_{\ell}^2.
\end{equation}
Since $
h_{\ell, m} =  \frac{w_{\ell, m}^2}{w_{\ell, m}^2 + \hat \nu_{\ell}},
$
we derive that
\begin{align*}
\sum_{\ell=m}^{p-1} h_{\ell m} &\leq \sum_{\ell=m}^{p-1}\frac{w_{\ell m}^2}{\hat\nu_{\ell}}\\
&\leq \frac{\sum_{\ell=m}^{p-1}w_{\ell m}^2/(2\ell)}{\min_{\ell=L+1}^{p-1} \hat\nu_{\ell}/(2\ell)}\\
&\lesssim \frac{1}{\min_{\ell=L+1}^{p-1} \hat\nu_{\ell}/(2\ell)},
\end{align*}
where the last inequality follows with the two sets of given weights.

Therefore, to prove (\ref{hm}), we just need to show that 
\[
\min_{L+1\leq \ell \leq p-1} \hat\nu_{\ell}/\ell \gtrsim \lambda^{-1},
\]
which follows immediately by the signal strength condition and by the fact that 
\[
\|\S_{s_{\ell}}\|_2 - \|\true_{s_{\ell}}\|_2 \gtrsim -\lambda \sqrt{2\ell}
\]
uniformly for all $\ell$.
\end{proof}

\bibliography{references}

\end{document}